\documentclass[11pt]{article}
\pdfoutput=1

\usepackage{pstricks}

\usepackage{eucal}
\usepackage{color}

\usepackage{amsmath,bbm}
\usepackage{amssymb}
\usepackage{epic}

\usepackage{graphicx}
\usepackage{pict2e}

\title{Double Aztec Diamonds and the Tacnode Process}

\author{Mark Adler\thanks{2000
{\em Mathematics Subject Classification}. Primary:
60G60, 60G65, 35Q53; secondary: 60G10, 35Q58. {\em Key
words and Phrases}: Domino tilings, Aztec diamonds, Dyson's Brownian motion, Airy and tacnode
processes, extended kernels. \newline
 Department of Mathematics, Brandeis University,
Waltham, Mass 02453, USA. E-mail: adler@brandeis.edu.
The support of a National Science Foundation grant \#
DMS-07-00782 is gratefully acknowledged.}~~~~~~Kurt Johansson\thanks{Department of Mathematics,
KTH Royal Institute of Technology, Stockholm, Sweden. E-mail: kurtj@kth.se. The support of the Swedish Research Council (VR) and grant KAW 2010.0063 of the Knut and Alice Wallenberg Foundation are gratefully acknowledged.} ~~~~~ Pierre
van Moerbeke\thanks{ Department of Mathematics,
Universit\'e de Louvain, 1348 Louvain-la-Neuve, Belgium
and Brandeis University, Waltham, Mass 02453, USA. E-mail: pierre.vanmoerbeke@uclouvain.be and
vanmoerbeke@brandeis.edu. The support of a National Science
Foundation grant \# DMS-07-00782, FNRS, PAI grants is
gratefully acknowledged.\newline
The authors thank the Mathematical Sciences Research Institute (MSRI, Berkeley) for an inspiring semester on ``Random Matrix Theory, Interacting Particle Systems and Integrable Systems", where this work was initiated.}
}

\date{}

\newcommand{\MAT}[1]{\left(\begin{array}{*#1c}}
\newcommand{\mat}{\end{array}\right)}
\newcommand{\qed}{\leavevmode\unskip\nobreak\penalty200\hskip2pt\null
\nobreak\hfill\rule{1.1ex}{1.1ex}
\medbreak }

\newcommand{\rg}{\rightarrow}
\newcommand{\Ai}{\mathrm{Ai}}

\newcommand{\I}{{\rm i}}
\newcommand{\AR}{{\cal A}}
\newcommand{\CR}{{\cal C}}

\newcommand{\LR}{{\cal L}}

\newcommand{\BC}{{\mathbb C}}

\newcommand{\BP}{{\mathbb P}}

\newcommand{\BZ}{{\mathbb Z}}

\newcommand{\iy}{\infty}
\newcommand{\pl}{\partial}
\newcommand{\al}{\alpha}

\newcommand{\Id}{\mathbbm{1}}

\newcommand{\un}{\mbox{1\hspace{-5.6pt}I}}

	\newcommand{\tc}{\mbox{\tiny$\circ$} }

\newenvironment{proof}{\medskip\noindent{\it Proof:\/} }{\qed}

\newcommand{\vp}{\varphi}
\newcommand{\la}{\langle}
\newcommand{\ra}{\rangle}
\newcommand{\ga}{\gamma}

\newcommand{\dt}{\delta}
\newcommand{\Dt}{\Delta}
 
\newcommand{\sg}{\sigma}
\newcommand{\BR}{{\mathbb R}}
\newcommand{\lb}{\lambda}

\newcommand{\dis}{\displaystyle}

\newcommand{\BK}{{\mathbb K}}

\def\be#1\ee{\begin{equation}#1\end{equation}}
\def\bea#1\eea{\begin{eqnarray}#1\end{eqnarray}}
\def\bean#1\eean{\begin{eqnarray*}#1\end{eqnarray*}}

 \newtheorem{definition}{Definition}[section]
 \newtheorem{theorem}[definition]{Theorem}
 \newtheorem{lemma}[definition]{Lemma}
 
 \newtheorem{proposition}[definition]{Proposition}

\catcode `!=11

\newdimen\squaresize
\newdimen\thickness
\newdimen\Thickness
\newdimen\ll! \newdimen \uu! \newdimen\dd! \newdimen \rr! \newdimen
\temp!

\def\sq!#1#2#3#4#5{%
\ll!=#1 \uu!=#2 \dd!=#3 \rr!=#4
\setbox0=\hbox{%
 \temp!=\squaresize\advance\temp! by .5\uu!
 \rlap{\kern -.5\ll!
 \vbox{\hrule height \temp! width#1 depth .5\dd!}}%
 \temp!=\squaresize\advance\temp! by -.5\uu!
 \rlap{\raise\temp!
 \vbox{\hrule height #2 width \squaresize}}%
 \rlap{\raise -.5\dd!
 \vbox{\hrule height #3 width \squaresize}}%
 \temp!=\squaresize\advance\temp! by .5\uu!
 \rlap{\kern \squaresize \kern-.5\rr!
 \vbox{\hrule height \temp! width#4 depth .5\dd!}}%
 \rlap{\kern .5\squaresize\raise .5\squaresize
 \vbox to 0pt{\vss\hbox to 0pt{\hss $#5$\hss}\vss}}%
}
 \ht0=0pt \dp0=0pt \box0
}

\def\vsq!#1#2#3#4#5\endvsq!{\vbox to \squaresize{\hrule
width\squaresize height 0pt%
\vss\sq!{#1}{#2}{#3}{#4}{#5}}}

\newdimen \LL! \newdimen \UU! \newdimen \DD! \newdimen \RR!

\def\vvsq!{\futurelet\next\vvvsq!}
\def\vvvsq!{\relax
  \ifx     \next l\LL!=\Thickness \let\continue=\skipnexttoken!
  \else\ifx\next u\UU!=\Thickness \let\continue=\skipnexttoken!
  \else\ifx\next d\DD!=\Thickness \let\continue=\skipnexttoken!
  \else\ifx\next r\RR!=\Thickness \let\continue=\skipnexttoken!
  \else\def\continue{\vsq!\LL!\UU!\DD!\RR!}%
  \fi\fi\fi\fi
  \continue}

\def\skipnexttoken!#1{\vvsq!}

\def\place#1#2#3{\vbox to 0pt{\vss
\rlap{\kern#1\squaresize
  \raise#2\squaresize\hbox{$#3$}}
\vss}}

\catcode `!=12

\begin{document}

\sloppy
\maketitle

\begin{abstract}
 Discrete and continuous non-intersecting random processes have given rise to critical ``{\em infinite dimensional diffusions}", like the Airy process, the Pearcey process and variations thereof. It has been known that domino tilings of very large Aztec diamonds lead macroscopically to a disordered region within an inscribed ellipse (arctic circle in the homogeneous case), and a regular brick-like region outside the ellipse. The fluctuations near the ellipse, appropriately magnified and away from the boundary of the Aztec diamond, form an Airy process, run with time tangential to the boundary. 
 
 This paper investigates the domino tiling of two overlapping Aztec diamonds; this situation also leads to non-intersecting random walks and an induced point process; this process is shown to be determinantal. In the large size limit,  when the overlap is such that the two arctic ellipses for the single Aztec diamonds merely touch, a new critical process will appear near the point of osculation (tacnode), which is run with a time in the direction of the common tangent to the ellipses: this is the {\em tacnode process}. It is also shown here that this tacnode process is universal: it coincides with the one found in the context of two groups of non-intersecting random walks or also Brownian motions, meeting momentarily.    \end{abstract}

\tableofcontents


\vspace*{1cm}


 Discrete and continuous non-intersecting random processes have given rise to critical ``{\em infinite-dimensional diffusions}", like the Airy process and the Pearcey process and variations thereof. These problems have been widely studied during these last ten years.

One of these problems has been  the study of the behavior of $n$ non-intersecting Brownian motions on $\BR$, when $n\to \infty$. When the starting and end points of these particles are pinned down (say, for $t=0$ and $t=1$), the cloud of particles, when $n\to \infty$, sweeps out in space-time a certain region. The Brownian motions have been shown to fluctuate like the Airy process near the generic points of the boundary of that region; see \cite{PS02,Jo03b,Johansson3,TW03,AvM05}. The Airy process is governed by an Airy kernel, a double integral of a ratio of the exponential of cubic polynomials. Similar results have been obtained for the boundary of the frozen region of domino tiling problems \cite{Johansson3} and random 3D partitions \cite{OR07}.

The boundary of that space-time region or boundary of the frozen region may have some singularities. A singularity can be a cusp, which corresponds to a bifurcation of particles or two sets of particles merging in the Brownian case; or a cusp in the boundary of the frozen region due to some non-convexity of the boundary of the model. The fluctuations of the particles near the cusp or the statistical behavior of the tilings are described by the Pearcey kernel, which is a double integral, with integrand given by the ratio of the exponential of quartic polynomials; see  \cite{TW06,OR07,AOvM10,ACvM10}, among others.

Another situation is the one where two sets of particles meet momentarily and then separate again. Locally, the boundary of the two sets of particles has a singularity, which looks like two circles touching; such a singularity is a tacnode. It was an intriguing open problem to understand the local fluctuations of the particles in the neighborhood of this tacnode. Adler-Ferrari-van Moerbeke made an attempt in \cite{AFvM08} by considering two sets of Brownian particles leaving from and forced to two points, and by letting first the number of one set go to $\infty$, while leaving the second finite. This led to a rational perturbation of the Airy kernel at the point of encounter of the two sets of particles. Then, letting the number of particles in the other set of particles go to $\infty$ as well, one would expect to find the tacnode statistics. This approach remains an open problem! 

In \cite{AFvM12}, Adler-Ferrari-van Moerbeke resolved the tacnode problem for two groups of non-intersecting random walks ({\em discrete space and continuous time}); an explicit kernel was found, which is representable as a double integral, for which the limit could be taken, thus leading to a kernel expressible -roughly speaking- as the sum of four Airy-like double integrals. 

Delvaux-Kuijlaars-Zhang \cite{DKZ10} then found a kernel defined in terms of the solution of a $4\times 4$ Riemann-Hilbert problem for two groups of non-intersecting Brownian motions ({\em continuous space and continuous time}). K. Johansson \cite{Joh10} then found an explicit kernel for non-intersecting Brownian motions, which seemed quite different from the one in \cite{AFvM12}.

The present paper deals with the random domino tilings of two overlapping Aztec diamonds: it investigates the fluctuations of the domino tilings near the region of overlap. The problem of a single Aztec diamond has been widely investigated by the combinatorics and probability community; the highlight was the existence of an inscribed arctic circle: inside the circle the domino's display a disordered pattern and outside a regular brick wall pattern; see \cite{EKLP,EKLP2,JPS,FS03,Johansson3,Jo02b}. Johansson \cite{Johansson3} then showed that the domino tilings near the arctic circle fluctuate like the Airy process, upon observing the boundary with an appropriate magnifying glass; the process is run with a time which is tangential to the boundary. This was done by showing that domino tilings of Aztec diamonds can be translated into non-intersecting random walks and a point process, which turn out to be a determinantal process. 

The problem of two overlapping Aztec diamonds is new. 
It translates into two distinct groups of non-intersecting random walks, which are also determinantal ({\em discrete space and discrete time}). In the limit, when the squares of the diamonds get smaller, we give three equivalent kernels for the fluctuations of the domino tilings near the overlap: the two first ones are closely related; the third one coincides with the one obtained for the two groups of non-intersecting random walks above, as in  \cite{AFvM12}; similar methods will be used to obtain the result. The latter is very different from the first two ones. Surprisingly so, they differ due to the use of the Christoffel-Darboux formula in the proof for one representation and not for the other. The first one is shown to coincide with the one found in \cite{Joh10} for the two groups of non-colliding Brownian motions.



This paper is a first step towards a universality result for the statistical fluctuations near a tacnode, let it be a tacnode in the boundary of a frozen region or a tacnode for two groups of random processes meeting momentarily, including all variations on that theme, discrete and continuous time and discrete and continuous space.  Finally, A. Borodin and M. Duits \cite{BD10} investigate a Markov process with interlacing particles, which leads to a tacnode, but which is a different situation from our paper.



\section{Double Aztec Diamonds and main results}

   \vspace*{-4cm}

    \hspace*{-1cm}  \includegraphics[width=131mm,height=154mm]{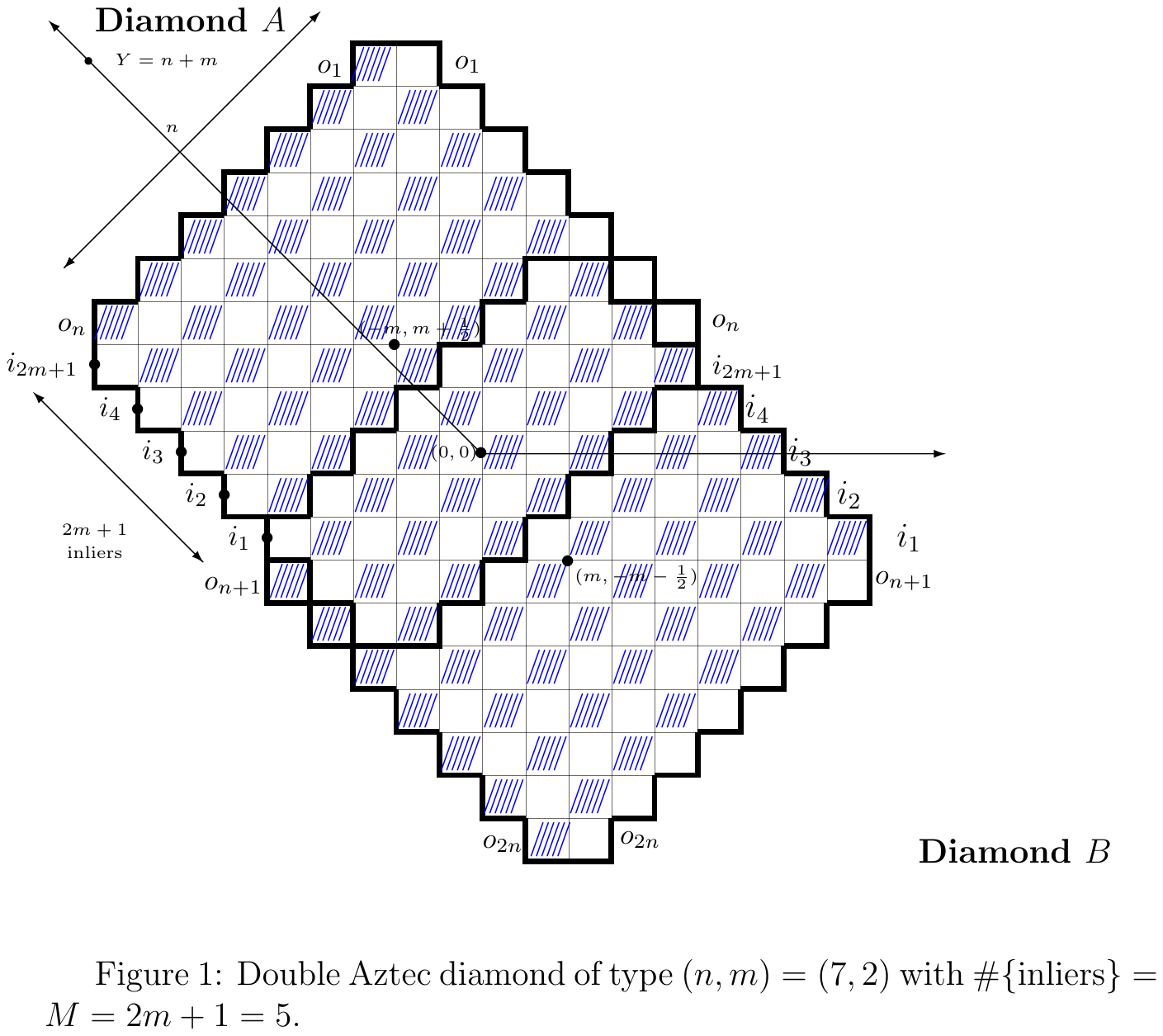}
   
   \vspace*{-4cm}
 
 \noindent  {\bf Domino tilings of double Aztec diamonds.} Consider two Aztec diamonds (checker boards) $A$ and $B$ of equal sizes $n$ ($=$ number of squares on the upper-left side) with opposite orientations; i.e., the upper-left square for diamond $A$ is black and is white for diamond $B$. The diamonds $A$ and $B$ overlap a certain amount, according to the pattern indicated in Figure 1. Let $M =:2m+1$ (taken to be odd) be the number of white squares, belonging to the left-lower boundary of the double diamond $A\cup B$; they correspond to the labels $i_1,\ldots,i_{2m+1}$ in Figure 1. $M$ also equals the number of black squares on the upper-right edge of $A\cup B$. The amount of overlap is given by $n-(2m+1)$. Cover this double diamond randomly by domino's, horizontal and vertical ones, as in Figure 2. The position of a domino on the Aztec diamond corresponds to four different patterns, given in Figure 4 below: North, South, East and West.
 
   Together with this arbitrary domino-tiling of the double Aztec diamond $A\cup B$, one defines a height function $h$ specified by the heights, prescribed on the single domino's according to figure 4 below. Let the upper-most edge of the double diamond $A\cup B$ have height $h=0$. Then, regardless of the covering by domino's, the height function along the boundary of the double diamond will always be as indicated in Figure 2, with height $h=2n$ along the lower-most edge of the double diamond. Away from the boundary the height function will, of course, depend on the tiling; the associated heights are given in Figure 2.

  The height function $h$ obtained in this way defines the domino tiling in a {\em unique} way, because a white square together with its height specifies in a unique way to which domino it belongs to: North, South, East and West; the same holds for black squares. For example, a white square with a constant height along the four edges can only be covered by a North domino. 
   
   The dual height function is obtained by polar reflection through the center of the four domino's, leading to the interchanges North$\leftrightarrow$South and East$\leftrightarrow$West; compare Figures 4 and 5. The dual height function leads to different boundary conditions on the double Aztec diamond, as shown in Figure 3.

 
     \vspace*{-2cm}

      ~$$  \hspace*{-2cm}
    \mbox{\includegraphics[width=131mm,height=154mm]{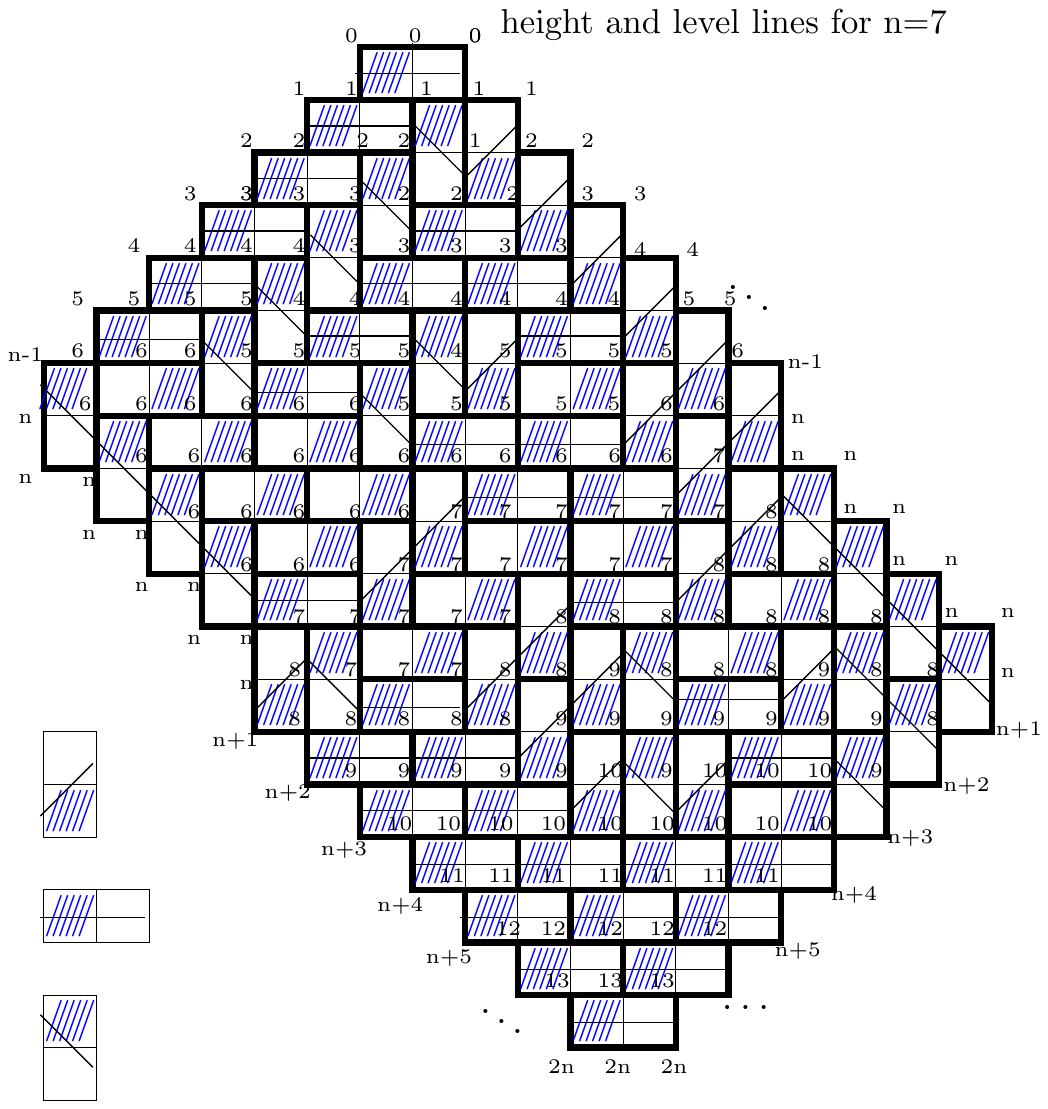}    \hspace{-7cm}\includegraphics[width=131mm,height=154mm]{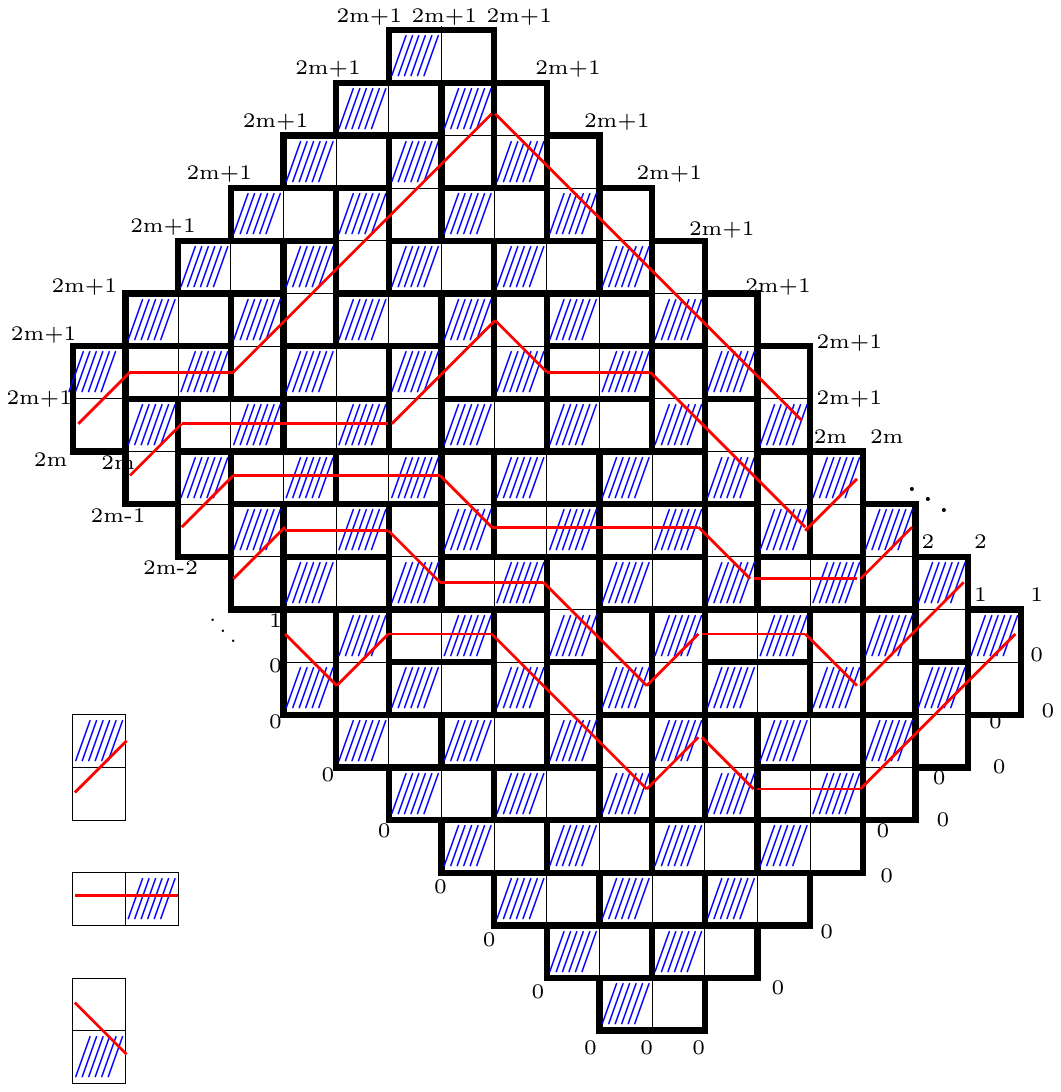}
    }$$

   \vspace*{-6cm}
   
   Fig 2.  Domino-tiling of a double Aztec-diamond, the associated height function $h$ and its level lines.

   Fig 3.  Domino-tiling of a double Aztec-diamond, the dual height $\tilde h$ and its level lines.
   
   \bigbreak
   
   \noindent   {\bf Extension of the tiling region}. It is convenient to extend the upper-left side of the Aztec diamond $A$ by means of horizontal South dominoes and doing the same to the right of the diamond $B$, as depicted in figure 8. To be precise, at the upper level, one adds $n$ South domino's, at the next level $n-1$, etc... So, the extension has a fixed tiling. Then, extending the height function to this new region, one has that the whole upper edge of the double Aztec diamond has height $=0$ and the whole lower edge height $=2n$.

     \newpage

    \vspace*{-4cm}

  $$\mbox{ \hspace*{-1cm}  \includegraphics[width=131mm,height=154mm]{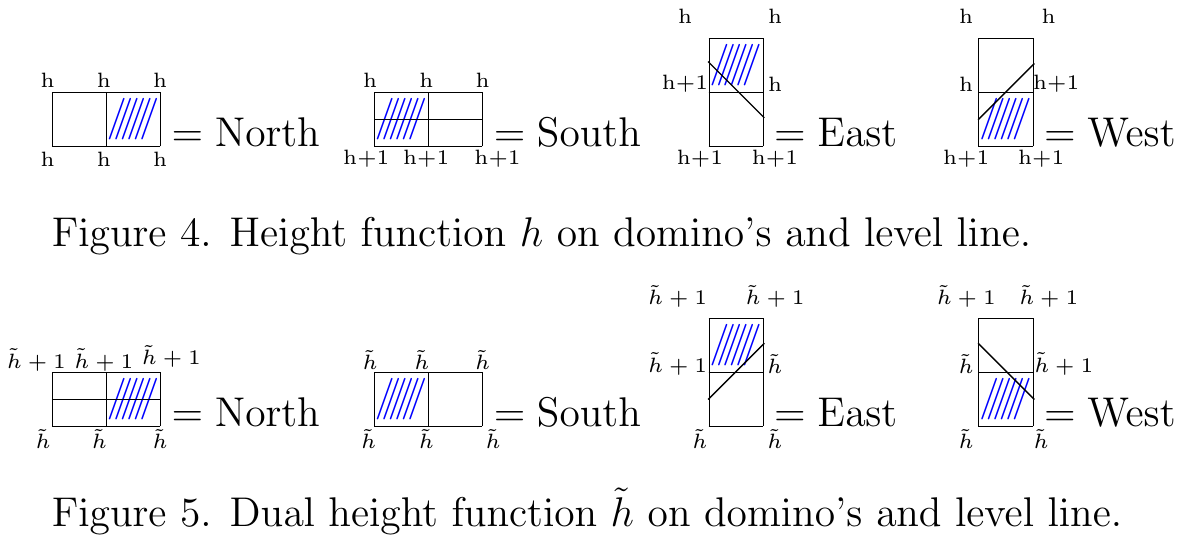}
  }$$
   
    \vspace*{-10cm}

   
   

\noindent     {\bf The $2n$ outlier paths}: denote by $o_k$ the $2n$ level lines for this height function at levels $h= k-1/2 $ for $1\leq k\leq 2n $. These paths will be called the ``{\em outliers}". Since these paths correspond to distinct levels, they do not intersect each other. They consist of two groups, an upper-group of $n$ paths $o_1,\ldots ,o_{ n}$, starting from the middle points of the left-edges of the black squares on the the upper-left side of domino $A$ and ending up at the corresponding points to the right of the same diamond $A$, as in Figure 2. Then there is a lower-group of paths $o_{n+1},\ldots ,o_{ 2n}$ starting from the middle points of the left-edges of the black squares on the the lower-left side of domino $B$ and ending up at the corresponding points to the right of the same diamond $B$, also as indicated in Figure 2.

 Given the extension of the tiling region, the outlier paths $o_1,\ldots,o_n$ are now extended by horizontal lines on the left and the paths $o_{n+1},\ldots ,o_{ 2n}$ by horizontal lines to the right; they remain level lines for the height function. 
 
 These outlier paths will now define two kind of particles, blue ``{\em dot-particles}" and blue ``{\em circle-particles}", according to the recipe indicated on the left hand side of Figure 6: put a dot-particle on the left-hand side of the segment of path traversing West and South domino's and a circle-particle on the right hand side of that segment, possibly with superposition of dots- and circle-particles, when edges of dominoes are in common; see Figure 8. 
     Notice that a North domino will never be visited. Both the upper-outliers and  lower-outliers thus have
    $$ \#\{\mbox{circles}\}=\#\{\mbox{dots}\}=n+1 .$$
  
  \noindent {\bf The $2m+1$ inlier paths} $i_{-m},\ldots, i_{m}$ are the non-intersecting level lines for the dual height function $\tilde h$ corresponding to the levels $\tilde h=k-1/2$, for $1\leq k\leq 2m+1$. This height function $\tilde h$ is prescribed for single domino's in Figure 5, and is applied to an arbitrary domino tiling of the double diamond $A\cup B$. Setting the height $\tilde h=0$ at the lower edge of $A\cup B$, the heights of the boundary edge is also completely determined, whatever be the tiling;
 
  \newpage


    \vspace*{-7cm}

 $$\mbox{  \hspace*{-1cm}  \includegraphics[width=131mm,height=154mm]{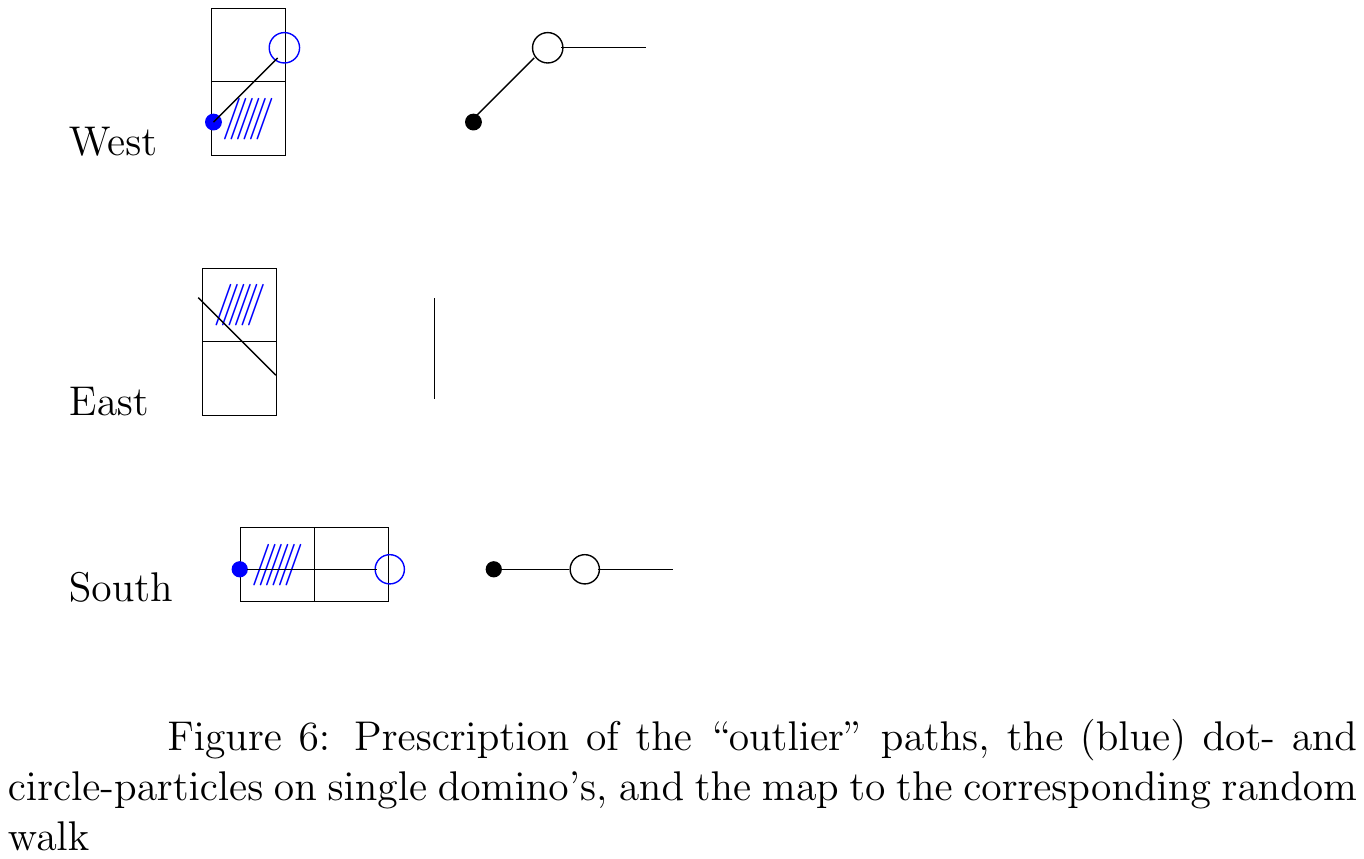}
 }$$
   
    \vspace*{-9cm}

   \hspace*{-1cm}  \includegraphics[width=131mm,height=154mm]{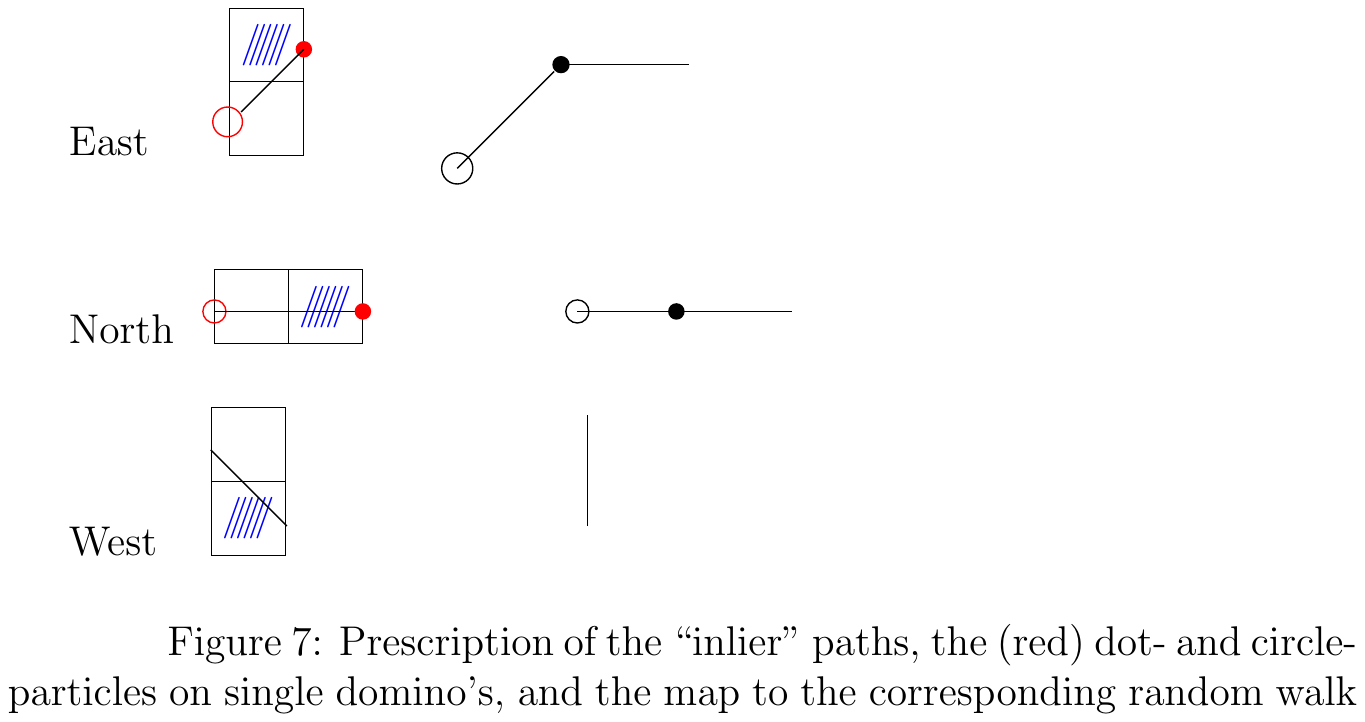}
   
    \vspace*{-7cm}

     \medbreak

\noindent  the height of the upper-edge of $A\cup B$ is then $\tilde h=2m+1$, as indicated in Figure 3.
The inlier paths depart from the middle points $i_1,\ldots,i_{2m+1}$ of the edges of the $2m+1$ white squares on the left boundary of the diamonds $A$ and end up at the corresponding locations on the diamond $B$. 

Here also, one defines two kind of  particles, red ``{\em dot-particles}" and red ``{\em circle-particles}", according to the recipe indicated on the left of Figure 7: put a circle-particle on the left-hand side of the segment of path traversing East and North domino's and a dot-particle on the right hand side of that segment. The paths belonging to a vertical domino of type West contain neither circles nor dots. Here South domino's will never be visited.




Circle- and dot-particles can then be superimposed when edges of dominoes are in common. As boundary condition, one always puts a dot at the initial points of the paths $i_1,\ldots, i_{2m+1}$ and a circle at the end points. So, as depicted in Figure 9, each path has therefore
    $$
    \#\{\mbox{circles}\}= \#\{\mbox{dots}\}=n+1.
    $$
  The inlier paths will not be used in the statement of the main Theorems, but they will play a crucial role in the arguments, which will first be made for the inlier paths and then, by duality, for the outlier paths.  
  
  \bigbreak
 
{\bf Remark} The height functions $h$ and $\widetilde h$ are not the standard ones, as exhibited in \cite{EKLP}. They can be obtained from one another by an affine transformation.  

\bigbreak
      
    \noindent  {\bf The system of coordinates $(\ell,Y)$, the dot and circle particles} are defined as follows: consider a horizontal ``time-axis" $\ell$ and oblique axes $Y_{\ell}$'s through the values $0\leq \ell \leq 2n+1$; namely, {\em (i)} axes $Y_{2r-1}=(B_r,D_r)$, for $1\leq r\leq n$, traversing white squares and {\em (ii)} axes $Y_{2r}=(A_{r+1},C_{r+1})$ for $1\leq r\leq n$ traversing black squares. 
   The axes $Y_{2r} $ (resp. $Y_{2r-1} $) take on the values $-n-m\leq k\leq n+m$ given by the location of the dots (resp. circles), obtained by moving the dots (resp. circles) belonging to the line (halfway in between the lines $Y_{2r-1}$ and $Y_{2r}$) horizontally towards the axis $Y_{2r}$ (resp. towards the axis $Y_{2r-1}$). The value of the dots and circles is specified by the equidistant horizontal lines beginning with the first one passing through the middle of the bottom most domino (level $Y=-n-m$) and the highest one passing through the middle of the top most domino (level $Y=m+n$).

    
    In other terms, an integer corresponds to a gap on the $Y_{2r-1}$-line, when the $Y_{2r-1}$-axis traverses a domino, not containing a circle-particle along that axis; similarly, an integer corresponds to a gap on the $Y_{2r}$-axis, when the $Y_{2r}$-axis traverses a domino, not containing a dot-particle along that axis. See Figure 8 for an example, involving the domino tiling given in Figure 2. So, the locations of the gaps on the axis $Y_{2r-1}$, say,  indicate a change in pattern precisely at those locations.
    
    Another way of describing the dot particles on $Y_{2r}$ and the circles on $Y_{2r-1}$ is as follows: put a dot on $Y_{2r}$, in the middle of the black square, whenever the axis $Y_{2r}$ intersects an outlier path in that black square. Put a circle on $Y_{2r-1}$, in the middle of the white square, whenever the axis $Y_{2r-1}$  intersects an outlier path in that white square.
    
\bigbreak

\noindent  {\bf A weight on domino's, a probability on domino tilings and on non-intersecting random walks}: put the weight $0<a<1$ on vertical dominoes and the weight $1$ on horizontal dominoes, so that the probability of a tiling configuration $T$ can be expressed as 
\be\BP(\mbox{domino tiling}~T)=\frac{\dis a^{\#\mbox{vertical domino's in $T$} }}{\dis\sum_{\mbox{\tiny all possible tilings $T$}}{a^{\#\mbox{vertical domino's in}~T}}}
\label{Ptiling}\ee
     
The definition (\ref{Ptiling}) of the probability measure on domino tilings is not affected by the extension of the tiling region, since the extension has a fixed

\newpage

        \newpage

   \vspace*{-5cm}


\hspace*{-1cm}  \includegraphics[width=131mm,height=154mm]{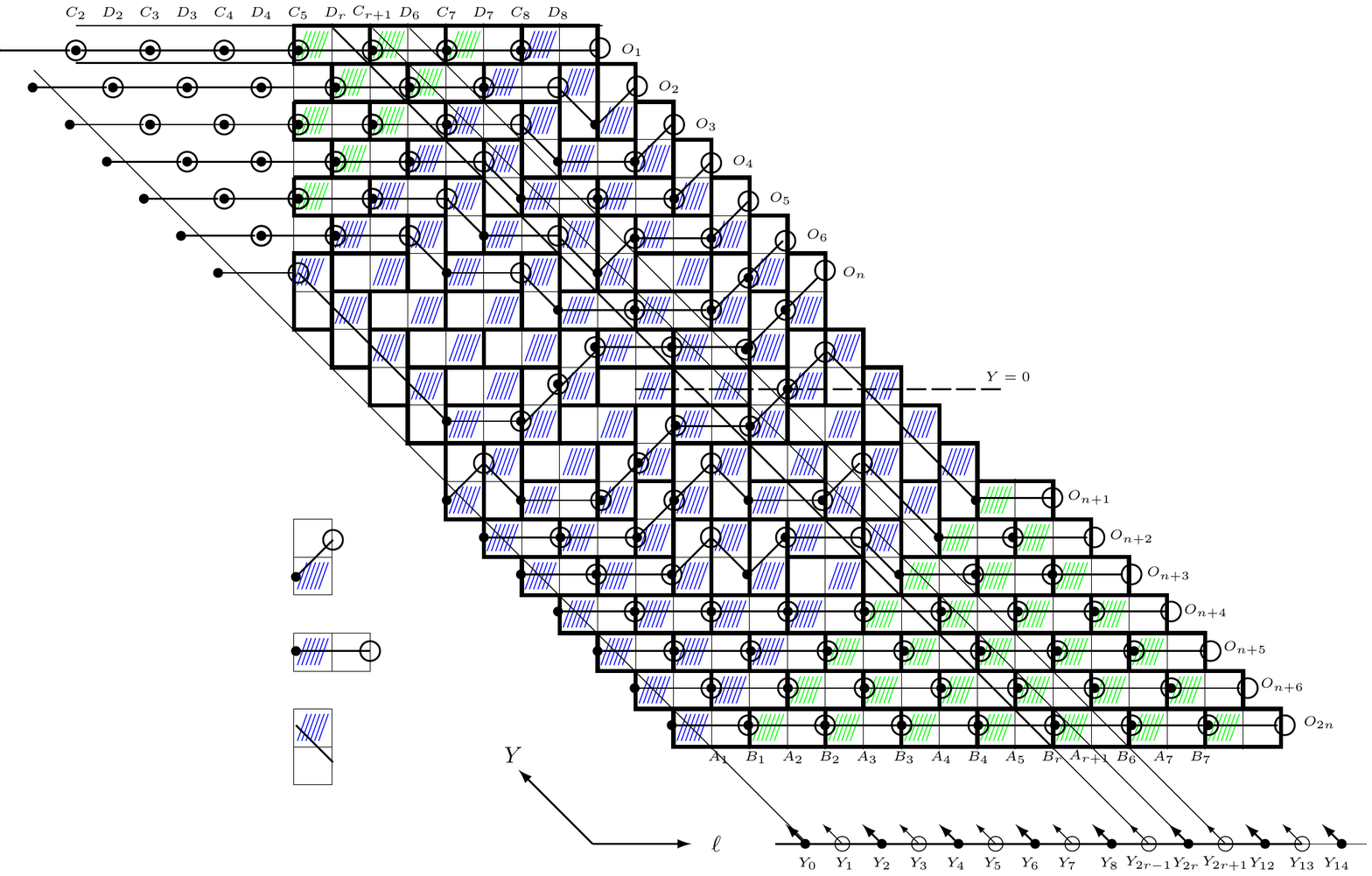}

   \vspace*{-5.5cm}

 {\footnotesize  Figure 8. Outlier non-intersecting domino paths, circles and dots. The coordinate $Y_{2r} $ (resp. $Y_{2r-1} $) records the location of the dots (resp. circles), obtained by moving the dots (resp. circles) in between $Y_{2r} $ and $Y_{2r-1} $ to the axis $Y_{2r} $ (resp. $Y_{2r-1} $). The values thus obtained correspond exactly to the location of the dots (resp. circles) on $Y_{2r}$ (resp. $Y_{2r-1} $) in the lattice paths description of Figure 16. E.g. for $r=5$, $Y_{2r}$ takes on the values $-9=-n-m\leq k\leq n+m=9$, except for gaps at $-4,~-2,~0,~4,~6$.}

   \vspace*{-2.5cm}

   \hspace*{-1cm}  \includegraphics[width=131mm,height=154mm]{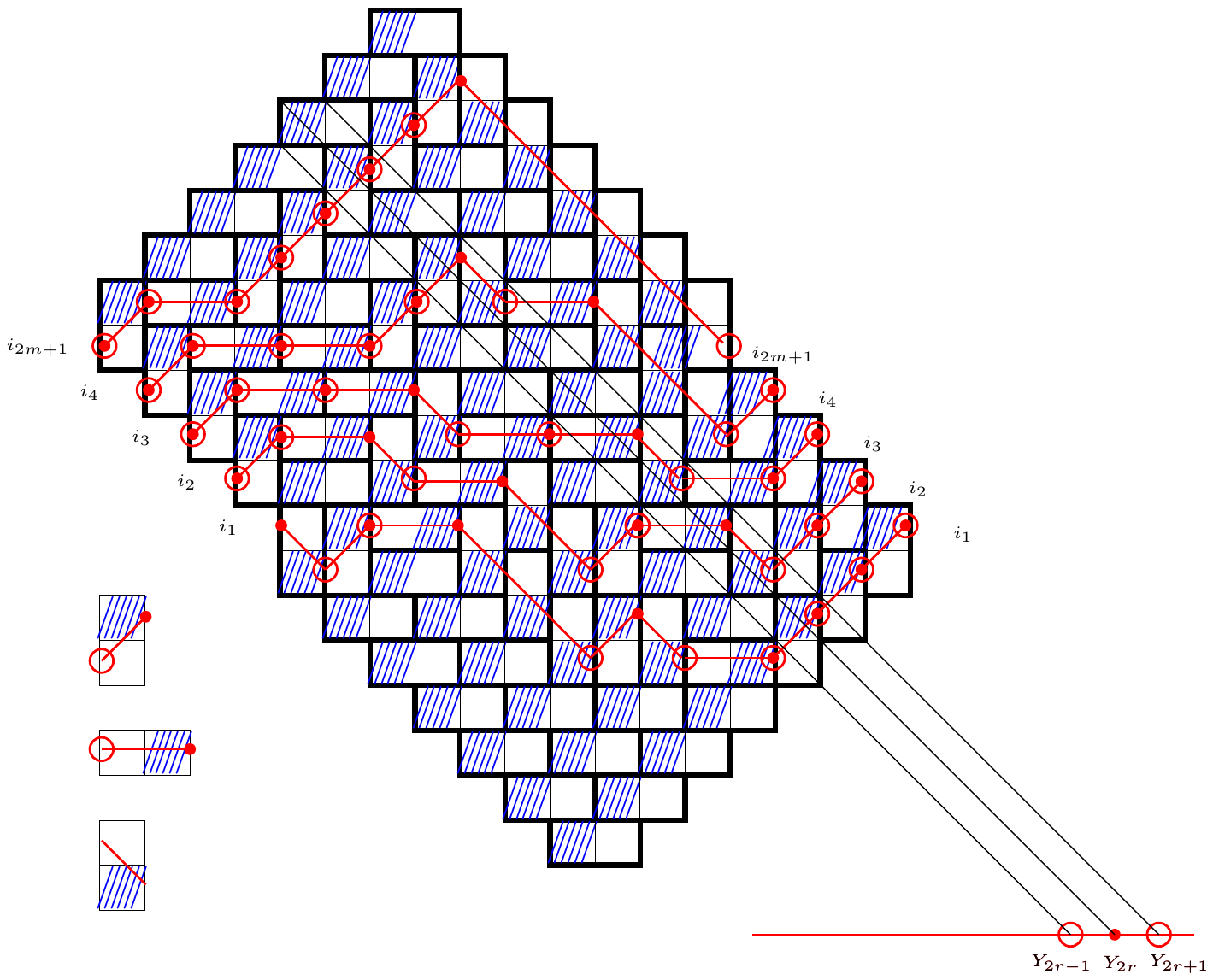}
   
   \vspace*{-6.5cm}

  {\footnotesize Figure 9. Inlier non-intersecting domino paths $i_1,\ldots , i_{2m+1}$. Here the coordinate $Y_{2r} $ (resp. $Y_{2r-1} $) records the location of the red dots (resp. red circles), obtained by moving horizontally the red dots (resp. red circles) in between $Y_{2r} $ and $Y_{2r+1} $ to the axis $Y_{2r} $ (resp. $Y_{2r+1} $). The values thus obtained correspond exactly to the location of the dots (resp. circles) on $Y_{2r}$ (resp. $Y_{2r+1} $) in the lattice paths description of Figure 15. E.g. $Y_{2r}$ takes on the values $-4,~-2,~0,~4,~6$.}

\vspace*{1cm}

 \newpage
 



    %
%
        
     
    \noindent tiling. This domino measure induces a probability on the outlier configurations. This point process is a determinantal process, whose kernel will be given in Theorem \ref{main1'}. This measure also maps to a probability measure on non-intersecting random walk paths. The duality between the ``inliers" and ``outliers", to be shown in Section \ref{sect2}, will enable us to obtain the formula of the correlation kernel for the outlier point process.

\bigbreak

\noindent {\bf The Frozen Region for the Double Aztec Diamond, when $n\to \infty$}. It is known from Jockush, Propp and Shor \cite{JPS} 
 that in the limit for a single diamond the stochastic region lies inside an ``arctic" ellipse and the frozen region outside that region; if, say, diamonds $A$ or $B$ would be by themselves, the ellipses would be given by 
$$  \frac{(x \pm\tfrac r2)^2}{p}+\frac{(y \pm\tfrac r2)^2}{q}=1,~~ \mbox{with} ~q=\frac{a }{a+a^{-1}}\mbox{  and  }p=1-q=\frac{a^{-1} }{a+a^{-1}},
$$
in terms of the weight $a$ and in terms of the coordinates $(x,y)$, centered in the middle of the double diamond $A\cup B$ rescaled by $1/n$. Thus the frozen region for each individual diamond would lie outside these ellipses.
If there would be not much overlapping of the double Aztec diamonds $A\cup B$, so that the two ellipses have no point in common, it seems reasonable to assume that in the limit the frozen region will be outside these two same ellipses. Augmenting the amount of overlap will bring these two individual ellipses to a point where they merely touch, as indicated in Figure 10. Geometrically, this would be the case when 
 the amount of overlap $ n-2m-1 $ of the two diamonds, rescaled by $1/n$ for $n$ large, equals
$$
\frac{n-2m-1}n=1-\frac{2m+1}n\simeq 1-\frac 2 {a+a^{-1}}=1-2\sqrt{pq}=:1-r.
$$
Indeed, the centers $(\mp m,\pm (2m+1))$ of the diamond $A$ and $B$, rescaled by $1/n$, are then given by 
$$\left(\mp \frac {m}n,\pm \frac {m+\tfrac 12}n\right) \simeq \Bigl(\mp \frac 1 {a+a^{-1}},\pm \frac 1 {a+a^{-1}}\Bigr)=(\mp \tfrac r2,\pm \tfrac r2),
    $$
and thus these two ellipses are tangent to each other at $(0,0)$. 
So, one expects the frozen region for a domino covering of the double Aztec diamond $A\cup B$ to be given by the region outside the two ellipses within $A\cup B$; this is not shown in this paper. Still this argument suggest the scaling 
$$\frac mn \simeq  \frac 1 {a+a^{-1}}.$$
One might conjecture that the two regions, in between the ellipses and the boundary of the double diamond (having the origin in common), are frozen with North domino's.

   \vspace*{-3cm}

    \hspace*{-1cm}  \includegraphics[width=131mm,height=154mm]{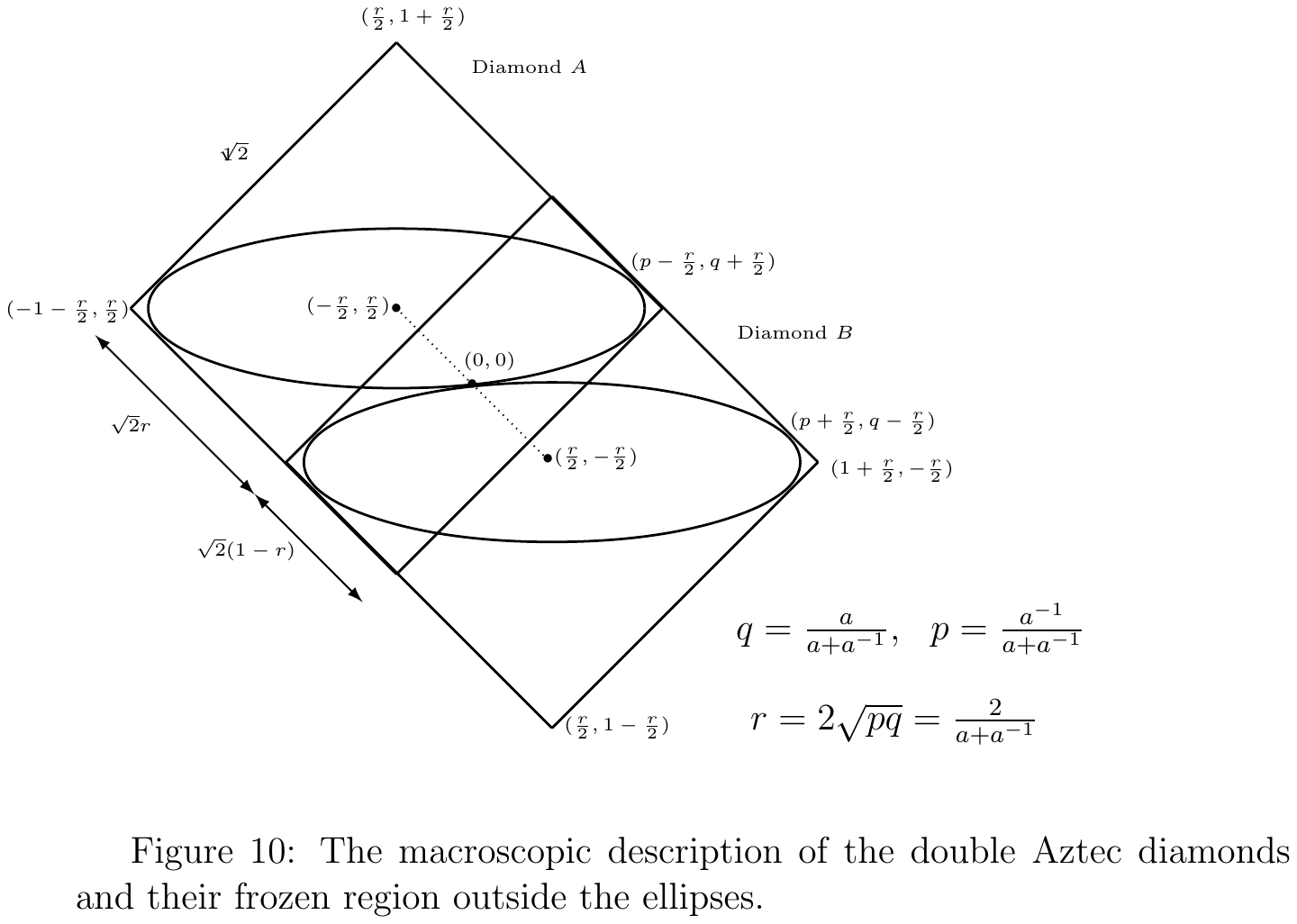}
   
   \vspace*{-6cm}



%


          \bigbreak 
          
         
%

        
\noindent{\bf Main results}: In view of Theorem \ref{main1'}, define the kernel for the one-Aztec diamond obtained from \cite{Johansson3} by setting $u\mapsto -u,~v\mapsto -v$,
\be\begin{aligned}
 \lefteqn{{\mathbb K}_{n }^{\mbox{\tiny OneAztec}}
 (2r,x;2s,y)}\\
 &\!=\!\frac {(-1)^{x\!-\!y}}{(2\pi \I)^2} \oint_{\gamma_{r_3}}du 
  \oint_{\gamma_{r_2}}\frac{dv}{v\!-\!u} 
 \frac{v^{-x}}{u^{1-y}} \frac{(1+au)^{n -s}(1-\tfrac au)^{s}}
 {(1+av)^{n -r}(1-\tfrac av)^{r}}
   -\Id_{s>r}\psi_{2(s-r)}(x,y)
 , \end{aligned} \label{OneAztec}\ee 
with the integral $\psi_{2(s-r)} $, defined below.
$\gamma_{r}$ refers to circles of radius $r$ around $0$. In (\ref{OneAztec}) and throughout the paper many different radii will be considered, subjected to the following inequalities:
  \be 
0<a<r_3<r_2<s_2<s_1<r_1<s_3<a^{-1}~~\mbox{  and  }~~ 0<a<\rho<1.\label{rad'}\ee
The contours 
$\Gamma_0$ or $\Gamma_{0,a}$  refers to a contour containing $0$ or $0,a$, but no other pole of the integrand; setting $n$ even, we now define the following functions, used throughout the paper:    
        $$
\begin{aligned}
S(2r,x; ~2s, y)  
& : =\frac{(-1)^{x-y}} {(2\pi \I)^2}\!\oint_{\gamma_{r_3}}du
 \oint_{\gamma_{r_2}}\frac{dv}{v\!-\!u}\frac{v ^{x-m-1 } }{u ^{y-m } } 
 \frac{(1+au)^s(1-\tfrac au)^{n-s+1}}
 {(1+av)^r(1-\tfrac av)^{n-r+1}}
 \\
 a_{x,s}(k) 
 &:=\frac{(-1)^{k-x}}{(2\pi \I)^2}\oint_{\gamma_{s_1}}du
 \oint_{\gamma_{r_1}}\frac{dv}{u-v}~
 \frac{v^{x+m}}{u^{k+1}}\frac{(1+av)^s(1-\frac av)^{n-s+1}}
 {\vp_a(2n;u)
 }\\
  b _{y,r}(\ell) 
   &:=\frac{(-1)^{\ell-y}}{(2\pi \I)^2}\oint_{\gamma_{r_2}}du
 \oint_{\gamma_{s_2}}\frac{dv}{u-v}~
 \frac{v^{\ell}}{u^{y+m+1}}\frac{\vp_a(2n;v)
 }{(1+au)^r(1-\frac au)^{n-r+1}}
 \\
 %
    K (k,\ell)  &:=
 \frac{(-1)^{k+\ell}}{(2\pi \I)^2}  \oint_{\gamma_\rho }du\oint_{ { \gamma_{\rho^{-1}} }}\frac{dv}{v-u}\frac{u^{\ell}}{v^{k+1}}
    \frac
 {\vp_a(2n,u)
 }
   {\vp_a(2n,v)
   }=(K ^{(1)} (0)  )_{k,\ell} 
  \end{aligned}$$
  $$\begin{aligned} 
    (K ^{(1)} (\tfrac 1z)  )_{k,\ell}  &:=
 \frac{(-1)^{k+\ell}}{(2\pi \I)^2}  \oint_{\gamma_\rho }du\oint_{ { \gamma_{\rho^{-1}} }}\frac{dv}{v-u}\frac{u^{\ell}}{v^{k+1}}
    \frac
 {\vp_a(2n,u)
 }
   {\vp_a(2n,v)
   }\frac{z-u}{z-v},\\
\psi_{ 2s}(x,y)&:= \oint _{\Gamma_{0,a}} \frac{dz}{2\pi iz} z^{x-y}
\left(\frac{1+az}{1-\frac az}\right)^{s},
~~~ \vp_a(2n;z):={(1+az)^n}{ (1-\tfrac az)^{  n+1 }}
\\   
%
 %
 \end{aligned} $$
 \be \begin{aligned}
g^{(1)}_{\ell}(n)&= -\oint_{\Gamma_{0}} \frac{du}{2\pi \I}(-u)^\ell \vp_a(2n;u)
,~~~
g^{(2)}_{\ell}(n) =  -\oint_{\Gamma_{0,a}} \frac{du}{2\pi \I}\frac{(-u)^{-\ell-2}}{
\vp_a(2n,u)}
\\
h^{(1)}_{k}(\tfrac 1z)&
:= \oint_{\Gamma_{0,a}} \frac {-dv}{2\pi \I(v\!-\!z)} \frac {(-v)^{-k-1}}{\vp_a(2n,v)}    
,~~~h^{(2)}_{\ell}(w)  :=  \oint_{\Gamma_{0}} \frac{-dv}{2\pi \I(v\!-\!\tfrac 1w)}\frac{\vp_a(2n;v^{-1})}{(-v)^{\ell+1}} 
 \end{aligned}\label{Idef}\ee
 For future use, $g^{(i)}(n)$ is a vector whose $\ell$th component is $g^{(i)}_\ell (n)$; similarly, $h^{(i)} (.)$ is a vector whose $\ell$th component is $h_\ell^{(i)} (.)$.  Also set (sometimes $p$ will be inside and sometimes outside)\be \begin{aligned}
  H_p^{(1)}(z^{-1})  := \det(\Id-K^{(1)}(z^{-1}))_{\ell^2(p,p+1,...)}& = \det(I-K_p^{(1)}(z^{-1}))\\&=\det(\Id-K^{(1)}(z^{-1}))_p ,\\
   \mbox{  with } K_p^{(1)} =\chi_{[p,p+1,\ldots]} &(K^{(1)})_{k,\ell}\chi_{[p,p+1,\ldots]} .\end{aligned}\label{three}\ee
   The same notation will be used for kernels on the reals, like 
   $$
\det(\Id-K)_{\sigma}=\det (\Id -   \raisebox{1mm}{$\chi$}{}_{[ \sg,\infty)}K \raisebox{1mm}{$\chi$}{}_{[ \sg,\infty)} )  
\mbox{  and  }
(\Id -    K  )^{-1}_{\sg}=
(\Id- \raisebox{1mm}{$\chi$}{}_{[ \sg,\infty)}K \raisebox{1mm}{$\chi$}{}_{[ \sg,\infty)})^{-1}.
$$
  We now state one of the two main Theorems, where $\BP$ refers to the probability (\ref{Ptiling}), defined above. The forms of the same kernel obtained in Theorem  \ref{main1'} are quite different; it is the use of the {\em Christoffel-Darboux formula} at some point of the proof which is responsible for the seemingly entirely different forms (\ref{K1}) and (\ref{K2}) of that same kernel.
  \noindent The main event for which the probability will be computed in this paper concerns the dot-particles, thus belonging the even lines $Y_{2r}$. The following events are all  equivalent: 
 $$\begin{aligned}
& \left\{\mbox{The line}~Y_{2r} ~\mbox{has a gap} \supset [k,\ell]\right\}
 \\&:=\left\{\mbox{The line $Y_{2r}$ has no dot-particles along the interval $[k,\ell] \subset Y_{2r}$}\right\}
\\&= \left\{\begin{aligned}&\mbox{The domino's covering black squares along the  }\\
& \mbox{interval $[k,\ell] \subset Y_{2r}$ are oriented left or down}\end{aligned}\right\}
 \\&=\left\{\mbox{The height function $h$ is flat along the interval $[k,\ell] \subset Y_{2r}$}\right\}
\\ &=\left\{\begin{aligned}&\mbox{The line $Y_{2r} $ does not intersect any of the outlier}\\
  &\mbox{paths in the black squares along $[k,\ell]\subset Y_{2r}$ }\end{aligned}\right\}
 \end{aligned}
 $$


         \vspace*{-3cm}

   \hspace*{-1cm}  \includegraphics[width=131mm,height=154mm]{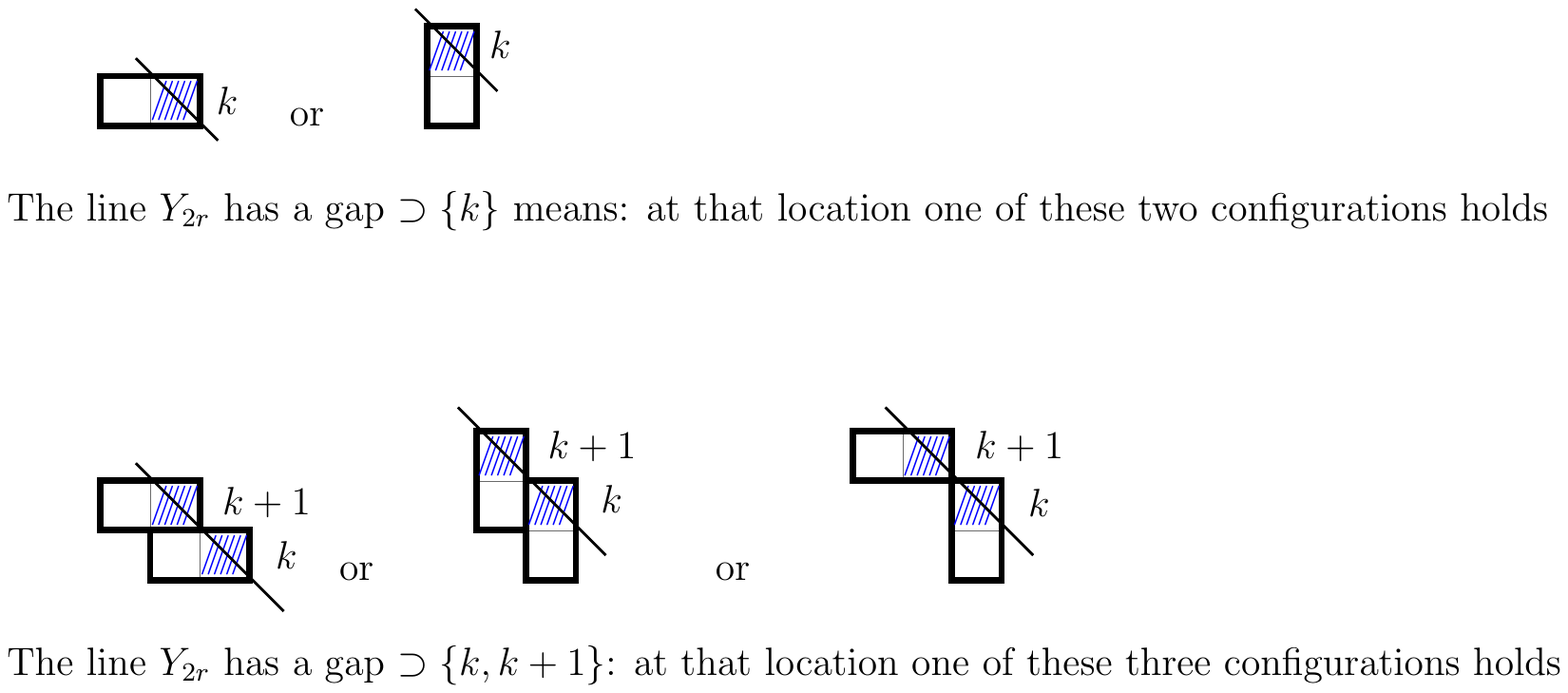}

   \vspace*{-8cm}
   
   Figure 11: The line has gaps $\supset \{k\}$ or $\{k,k+1\}$.
   
    \vspace*{1cm}

 In Theorem \ref{main1'}, one shows that the dot-particles along the lines $Y_{2r}$ with even indices form a determinantal process. The circle-particles and the combination of circle- and dot-particles form, at odd and even line $Y_\ell$ respectively, a determinantal process with a kernel which could be computed as well. 
      
 \bigbreak

 \begin{theorem} \label{main1'} The dot-particles for the outlier paths form a determinantal point process with correlation kernel  $\widetilde\BK^{\rm ext}_{n,m}$ given by (\ref{K1}) under the probability measure induced by (\ref{Ptiling}). In particular, given integers $-n-m<k\leq \ell< n+m $ ($n$ even), the following probability holds for the single line $Y_n$ ($Y$-axis in the middle) :
        \be
        {\mathbb P}\left(\mbox{The line}~Y_n ~\mbox{has a gap} \supset [k,\ell]\right)
        =\det\left(\Id-\chi_{[k,\ell]}\tilde{\mathbb K}_{n,m}\chi_{[k,\ell]}\right)
        \label{I1}\ee
 where $
 \widetilde \BK_{n,m}(x,y):=
 \widetilde\BK^{\rm ext}_{n,m}(n,x;n,y) 
,$ the extended kernel given below. For multiple lines $Y_{2r_i}$, $1\leq i\leq s$,
        \be\begin{aligned}
        \lefteqn{
        {\mathbb P}\left(\bigcap_{i=1}^{s}\left\{\mbox{the line}~Y_{2r_i} ~\mbox{has a gap} \supset [k_i,\ell_i]\right\}\right)
   }\\
  & \hspace*{2cm}     =\det\left(\Id-\left[\chi_{[k_i,\ell_i]}\tilde{\mathbb K}^{\rm ext}_{n,m}(2r_i,x_i;2r_j,x_j)\chi_{[k_j,\ell_j]}\right]_{1\leq i,j\leq s}\right).
    \end{aligned}  \label{I2}  \ee
$\tilde{\mathbb K}^{\rm ext}_{n,m}$ is a perturbation of the one-Aztec diamond extended kernel, with an inner-product 
 involving the resolvent of the kernel\footnote{$\la f(k), g(k)\ra_{\ell^2 (2m+1,\ldots)}=\sum_{2m+1}^{\infty} f(k)g(k).
$ Also remember the notations in (\ref{three}).} $K
 $, defined in (\ref{Idef}):
\be
  \begin{aligned}
\lefteqn{ \hspace*{-1cm}{(-1)^{x-y}   \widetilde\BK^{\rm ext}_{n,m}(2r,x;2s,y) }}\\
&= {\mathbb K}_{n+1}^{\mbox{\tiny \rm OneAztec}}
  \bigl(2(n-r+1 ),m-x+1;2(n-s+1 ),m-y+1\bigr) %
 \\  
 &~~~~+\left\la(\Id-K )_{2m+1}^{-1}a_{-y,s}(k),b_{-x,r}(k) \right\ra _{\ell^2 (2m+1,\ldots)}    ,\end{aligned}  %
 \label{K1}\ee
 with
 \be
 \begin{aligned}
  {\mathbb K}_{n+1}^{\mbox{\tiny \rm OneAztec}}
 (2(n-r+1 )&,m-x+1 ;2(n-s+1 ),m-y+1)
 \\
 &=-\Id_{s<r}  (-1)^{x-y}\psi_{2(s-r)}(x,y)  +S(2r,x;2s,y)
\end{aligned} \label{OneAztec1}\ee
 An alternative expression is given by:  \be
\begin{aligned}
\lefteqn{(-1)^{x-y}\frac{H_{2m+2}(0)}{H_{2m+1}(0)}
\widetilde\BK^{\rm ext}_{n,m}(2r, x;2s, y)
=
-\Id_{s<r} (-1)^{x-y}\psi_{2(s-r)}( x, y) \frac{H_{2m+2}(0)}{H_{2m+1}(0)}
} 
%
\\& + \frac{\dt_{x\neq y}}{2\pi i}
 \oint_{\Gamma_{0,a}}\frac{dz}{(-z)^{y-x+1}}
   (1-R^{(1)}_{ }(z^{-1}))(1-R^{(2)}_{}(z))
  \left( \tfrac{1+az}{1-\frac az}\right)^{s-r}\\
&+  \Bigl ( \frac{1}{2\pi i} \Bigr)^2
\oint_{\Gamma_{0,a}}dz 
\oint_{\Gamma_{0,a,z}}  \frac{dw}{z\!-\!w} ~(1-R^{(1)}_{ }(z^{-1}))(1-R^{(2)}_{ }(w))
\frac{\Bigl(\frac{1+az}{1-\tfrac aw}\Bigr)^s}
 {\Bigl(\frac{1+ aw}{1-\frac az}\Bigr)^r}
\\ \\
&~~~~~\times
 \left(
 \frac{(-w)^{-y-m-1}}{(-z)^{-x-m}}
 \left(\frac{1+az}{1+aw}\right)^{n-s-r}+
   \frac{(-z)^{-y+m}}{(-w)^{-x+m+1} }
   \left(\frac{1-\tfrac az}{1-\tfrac aw}\right)^{n-s-r+1}
  \right)
  \end{aligned}
  \label{K2}\ee
 with
 \be
  \begin{aligned}
R^{(1)}(z^{-1})&=
 \left\la(\Id-K^\top)_{2m+1}^{-1}g^{(1)}(n),h^{(1)}(z^{-1})\right\ra _{\ell^2 (2m+1,\ldots)}
 \\ R^{(2)}(w)&=
 \left\la(\Id-K)_{2m+1}^{-1}g^{(2)}(n),h^{(2)}(w)\right\ra _{\ell^2 (2m+1,\ldots)}.
  \end{aligned}\label{IR}\ee

\end{theorem}


\bigbreak\vspace*{1cm}

Given the weight $0<a<1$ on vertical dominoes, the following quantities $v_0<0$, and $A,~\rho ,~\theta >0$ will come into play:
\be\begin{aligned}
v_0&:= -\frac{1-a}{1+a} <0,~~A^3:=\frac{a(1+a)^5}{(1-a)(1+a^2)}, 
~\rho  :=-Av_0=\frac{(1-a^2)^{2/3}}{(a+a^{-1})^{1/3}}>0
\\
 \theta &:=\sqrt{\rho(a+a^{-1})}=\Bigl( \frac{1-a^4}{a}\Bigr)^{1/3}>0,~\mbox{and so }~A^2=\theta \frac{a(1+a)^3}{(1-a)(1+a^2)}.
\end{aligned}\label{4}\ee
Consider the scaling (\ref{2}) of the following quantites: the type $(n,m)$ of the double Aztec diamond, the space and time variables $x$ and $s$ in terms of the size $n=2t$ of each diamond, where $\sigma$ is a parameter, which measures the pressure between the two diamonds: (for future use set $\tilde \sigma:=2^{2/3} \sigma$)
\be
\begin{aligned}
n&=2t , ~~~\hspace*{3cm}~~m 
 =\frac{2t}{a+a^{-1}}+  \sigma  \rho
  t^{1/3}
\\
 x&=2a^2\theta \tau_1 t^{2/3}+\xi_1 \rho t^{1/3},~~~~~y=2a^2\theta \tau_2 t^{2/3}+\xi_2 \rho t^{1/3}
\\
r& =t+(1+a^2)\theta\tau_1 t^{2/3},~~~~~~~s =t+(1+a^2)\theta\tau_2 t^{2/3}
,\\
\end{aligned}
\label{2}\ee

In the second main Theorem, three formulas will be given for the limiting kernel. To state the first two ones,  the following functions are needed. 
Given the parameter $s$, define an extension $\Ai^{(s)}(x)$ of the usual Airy function $\Ai (x)$; further an Airy-type kernel, defined by means of $\Ai^{(s)}$ and finally a new function ${\cal A}_\xi^\tau$:
\be \begin{aligned}
\Ai^{(s)}(x)&:=\frac{1}{2\pi i}\int _{\nearrow \atop \nwarrow }  dz~e^{ z^3/3+z^2s- zx }=e^{ sx+\tfrac 2 3s^3 }\mbox{Ai}(x+s^2),
\\
K_{ \Ai }^{(\alpha,-\beta)}(x,y)&:=\!\int_0^{\infty}\!\!
\Ai^{(\alpha)}(x\!+\!u)  \Ai^{(-\beta)}(y\!+\!u)du,~~ \mbox{with}~K_{ \Ai }=K_{ \Ai }^{(\alpha,-\beta)}\Bigr|_{\alpha=\beta=0}
\\&=\tfrac{e^{\alpha x+(2/3) \alpha^3}}{e^{\beta y+(2/3) \beta^3}}\int_0^\infty du~ e^{-(\beta-\alpha)u}\Ai(x+\alpha^2+u)\Ai(y+\beta^2+u) ,
\\
{\cal A}^{\tau}_\xi(\kappa) 
&:= \Ai^{(\tau)} ( \xi+2^{1/3} \kappa)- \int_0^\infty \Ai^{(\tau)} (-\xi+2^{1/3} \beta )
\Ai (\kappa+ \beta )d\beta  ,  \label{E6}
 \end{aligned}
\ee
and
$$
p(\tau;\xi_1,\xi_2) :=\frac{e^{-\frac{(\xi_1-\xi_2)^2}{4 \tau }}}{\sqrt{4\pi  \tau }} .
$$ 
Also define the extended kernel for the Airy process (see \cite{Johansson3,Joh10})
\be
\begin{aligned}
{\mathbb K}&^{\mbox{\tiny AiryProcess}}(\tau_1,\xi_1;\tau_2,\xi_2) =\int_0^{\infty}e^{\lambda (\tau_2-\tau_1)}
\Ai (\xi_1+\lambda)  \Ai (\xi_2+\lambda)d\lambda\\
&-
\tfrac{\Id_{\tau_2>\tau_1}}{\sqrt{4\pi (\tau_2\!-\!\tau_1)}}\mbox{exp} \left(-\tfrac{(\xi_1-\xi_2)^2}{4(\tau_2-\tau_1)}+\tau_1(\xi_1\!+\!\sigma)-\tau_2(\xi_2\!+\!\sigma)+\tfrac 23 (\tau_1^3-\tau_2^3)\right)
.\end{aligned}\label{AiryP}
\ee
The following expressions will be used in the second formula for the limiting kernel:
 \be
\begin{aligned} 
 {\cal Q}(\kappa)&:=
 \left[(\Id - \chi_{\tilde \sigma}K_{\Ai}\chi_{\tilde \sigma})^{-1}
\chi_{\tilde \sigma}  \Ai\right]
(\kappa)
\\
 \hat  {\cal Q}(\zeta)&:=
 \int_{\tilde \sigma}^{\infty}{\cal Q}(\kappa) e^{\kappa 2^{1/3}\zeta} d\kappa
 \\
\hat {\cal P}(\zeta)&:=- \int_{\tilde \sigma}^{\infty}{\cal Q}(\kappa)
 d\kappa
\int_0^{\infty} e^{-2^{1/3}\zeta \beta} \Ai(\kappa+\beta)d\beta,
\end{aligned}\label{57}\ee
together with the following function\footnote{$(\xi\leftrightarrow -\xi)$means: replace $\xi\to -\xi$ in the prior expression.}:
\begin{multline}\label{eqSpaceTimeC}
{\cal C}(s,\xi):=2^{-1/3}\int_{\tilde\sigma}^\infty d\kappa \,{\cal Q}(\kappa) \bigg[ \Ai^{(2^{-2/3}s)}(\kappa+2^{-1/3}\xi)\\
+ \int_{\tilde\sigma}^\infty d\lambda \,{\cal Q}(\lambda) \int_0^\infty d\alpha\, \Ai(\alpha+\lambda) \Ai^{(2^{-2/3}s)}(\alpha+\kappa+2^{-1/3}\xi)\bigg]+(\xi\leftrightarrow -\xi).
\end{multline}

The tacnode process measures the probability that the height function $h$ is flat along small intervals of length ${\cal O}(n^{-2/3})$ in the $Y$-direction (see Figure 8), which themselves are ${\cal O}(n^{-1/3})$-close to the center of Figure 8; the tacnode process is run with a time scale of the order ${\cal O}(n^{-1/3})$, near the halfway point $\ell=n$ (see Figure 8); i.e., near the tangency point in Figure 10. 

\begin{theorem} \label{th:main} Given the scaling (\ref{2}), depending on the arbitrary parameter $\sigma$, with $\tilde \sigma:=2^{2/3} \sigma >0$,
the scaling limit of the dot-particles, along the lines $Y_{2r}$, near the tangency point of the two ellipses is a determinantal process with kernel:
\be
\lim_{t\to \infty} 
		 (-v_0)^{y-x+r-s}(-1)^{y-x}\widetilde \BK^{\rm ext}_{n,m}(2r,x ;2s,y)\rho t^{1/3}  ={\mathbb K}^{\rm tac} (\tau_1, \xi_1;\tau_2,\xi_2 )
		 \ee
where the tacnode kernel ${\mathbb K}^{\rm tac} (\tau_1, \xi_1;\tau_2,\xi_2 )$ has three equivalent forms. \newline {\bf (i)} A first form is a perturbation of the Airy process kernel (\ref{AiryP}):
\begin{equation}\begin{aligned}
  {\mathbb K}^{\rm tac} (\tau_1, \xi_1;\tau_2,\xi_2 )  
&=-\Id_{\tau_1>\tau_2}
p(\tau_1-\tau_2;\xi_1,\xi_2)+ K_{\Ai}^{(\tau_1,-\tau_2)}(\sg -\xi_1,\sg-\xi_2)
\\
&~~~+
 2^{1/3} \int_{\tilde\sg}^{\iy}\left((\Id- K_{\Ai}
 )_{\tilde\sg}^{-1}
 \AR_{\xi_1-\sg}^{ \tau_1}\right)(\lb)
\AR_{\xi_2-\sg}^{ -\tau_2 }(\lb)d\lb.
\end{aligned}
\label{E127}\end{equation}
%
%
%
%
%
where~\footnote{with $q_\sigma(\tau,\xi):=e^{\tau (\sigma-\xi )+(2/3)\tau ^3}$. 
}
\be\begin{aligned}
-\Id_{\tau_1>\tau_2}
p(\tau_1-\tau_2;&\xi_1,\xi_2)+ K_{\Ai}^{(\tau_1,-\tau_2)}(\sg -\xi_1,\sg-\xi_2)
\\
&=\frac{q_\sigma (\tau_1,\xi_1)}{q_\sigma (\tau_2,\xi_2)} {\mathbb K}^{\mbox{\tiny \rm AiryProcess}}(\tau_2,\sigma-\xi_2+\tau_2^2;\tau_1,\sigma-\xi_1+\tau_1^2) 
\label{AiryP1}\end{aligned}\ee
 {\bf (ii)} It can also be expressed as
 \begin{equation}\begin{aligned}
 \lefteqn{ {\mathbb K}^{\rm tac} (\tau_1, \xi_1;\tau_2,\xi_2 )  }
\\&=-\Id_{\tau_1>\tau_2}
p(\tau_1-\tau_2;\xi_1,\xi_2)
\\
&~~~+
 2^{\frac 13} \!\!\!\int_{\tilde\sg}^{\iy}\!\left[\left((\Id\!-\! K_{\Ai}
 )_{\tilde\sg}^{-1}
 \AR_{\xi_1-\sg}^{ \tau_1}\right)(\lb)
 \Ai^{(-\tau_2)} ( \xi_2-\sigma+2^{1/3} \lb)+\{\xi_i\leftrightarrow -\xi_i\}\right]d\lb.
\end{aligned}
\label{E127K}\end{equation}
{\bf (iii)} Another equivalent limiting kernel is given by the sum of the four double integrals, where $\delta>0$ is arbitrary:
\begin{equation}\label{main1}
\begin{aligned}
\lefteqn{{\mathbb K}^{\rm tac}(\tau_1,\xi_2;\tau_2,\xi_2) } 
\\=&- {\Id_{[\tau_2<\tau_1]}}  p(\tau_1-\tau_2;\xi_1,\xi_2)
+{\cal C}(\tau_1-\tau_2,\xi_1-\xi_2)\\
&+\frac{1}{(2\pi \I)^2}\int_{\delta +\I\BR}\hspace{-1em}du\int_{-\delta +\I\BR}\hspace{-1em}dv\,
\frac{e^{\frac{u^3}3-\sigma u}}{e^{\frac{v^3}3-\sigma v}}\frac{e^{\tau_1 u^2}}{e^{\tau_2 v^2}}
\left(\frac{e^{\xi_1 u}}{e^{\xi_2 v}}+\frac{e^{-\xi_1 u}}{e^{-\xi_2 v}}\right)
\frac{(1-\hat{\cal P}(u))(1-\hat{\cal P}(-v))}{u-v}\\
&-\frac{1}{(2\pi \I)^2}\int_{2\delta +\I\BR} \hspace{-1em}du \int_{\delta +\I\BR}\hspace{-1em}dv\,
\frac{e^{\frac{u^3}3-\sigma u}}{e^{-\frac{v^3}3-\sigma v}}\frac{e^{\tau_1 u^2}}{e^{\tau_2 v^2}}
\left(\frac{e^{\xi_1 u}}{e^{\xi_2 v}}+\frac{e^{-\xi_1 u}}{e^{-\xi_2 v}}\right)
\frac{(1-\hat{\cal P}(u)) \hat{\cal Q}(-v)}{u-v}\\
&-\frac{1}{(2\pi \I)^2}\int_{-\delta +\I\BR}\hspace{-1em}du\int_{-2\delta +\I\BR}\hspace{-1em}dv\,
\frac{e^{-\frac{u^3}3-\sigma u}}{e^{\frac{v^3}3-\sigma v}} \frac{e^{\tau_1 u^2}}{e^{\tau_2 v^2}}
\left(\frac{e^{\xi_1 u}}{e^{\xi_2 v}}+\frac{e^{-\xi_1 u}}{e^{-\xi_2 v}}\right)
\frac{(1-\hat{\cal P}(-v)) \hat{\cal Q}(u)}{u-v}\\
&+\frac{1 }{(2\pi \I)^2}\int_{-\delta +\I\BR}\hspace{-1em}du \int_{\delta +\I\BR}\hspace{-1em}dv\,
\frac{e^{-\frac{u^3}3-\sigma u}}{e^{-\frac{v^3}3-\sigma v}}
\frac{e^{\tau_1 u^2}}{e^{\tau_2 v^2}}
\left(\frac{e^{\xi_1 u}}{e^{\xi_2 v}}+\frac{e^{-\xi_1 u}}{e^{-\xi_2 v}}\right)
\frac{\hat{\cal Q}(u) \hat{\cal Q}(-v)}{u-v}.
\end{aligned}
\end{equation}

\end{theorem}

 Consider two groups of Brownian motions, leaving from and forced to distinct points at times $t=0$ and $1$, tuned in such a way that they meet momentarily at $t=1/2$. Then, when the number of particles gets very large, so as to preserve this tuning, the statistical behavior of the particles near the point of encounter are governed by the following kernel, as is shown by Johansson in \cite{Joh10} \footnote{The kernel below differs from Johansson's kernel in \cite{Joh10} by a transpose and a conjugation by an exponential; i.e., multiplication of the kernel by $\frac{e^{\tau_2(\xi_2+\sigma)+ (2/3) \tau_2^3}}{e^{\tau_1(\xi_1+\sigma)+  (2/3) \tau_1^3}}$.}:
\be\begin{aligned}
\lefteqn{{\mathbb K}_{br}^{\rm tac}(\tau_1,\xi_2;\tau_2,\xi_2) }\\
=&- {\Id_{[\tau_2<\tau_1]}}  p(\tau_1-\tau_2;\xi_1,\xi_2)\\
&+K_{\Ai}^{(\tau_1,-\tau_2)}(\sg -\xi_1,\sg-\xi_2)
+K_{\Ai}^{(\tau_1,-\tau_2)}(\sg +\xi_1,\sg+\xi_2)
\\
&+2^{\frac 13}
\int_{\tilde\sg}^{\iy}\left((\Id\!-\! K_{\Ai}
 )_{\tilde\sg}^{-1}
 \AR_{\xi_2-\sg}^{- \tau_2}\right)(\kappa )
\left(\AR_{\xi_1-\sg}^{  \tau_1 }(\kappa)-\Ai^{(\tau_1)}(\xi_1-\sg+2^{\frac 13}\kappa)\right)d\kappa
\\
&+2^{\frac 13}
\int_{\tilde\sg}^{\iy}\left((\Id\!-\! K_{\Ai}
 )_{\tilde\sg}^{-1}
 \AR_{-\xi_2-\sg}^{- \tau_2}\right)(\kappa)
\left(\AR_{-\xi_1-\sg}^{  \tau_1 }(\kappa)-\Ai^{(\tau_1)}(-\xi_1-\sg+2^{\frac 13}\kappa\right)d\kappa
\end{aligned}
\label{E127L}\ee

\begin{theorem}\label{th:main3}
The Brownian motion tacnode kernel (\ref{E127L}) is equivalent to anyone of the kernels in Theorem \ref{th:main}.
\end{theorem}

  It follows from the above result that the correlation kernels obtained by Adler, Ferrari and van Moerbeke in \cite{AFvM12}, by Delvaux, Kuijlaars and Zhang in \cite{DKZ10} and by Johansson in \cite{Joh10} are all equivalent, i.e., {\em they all describe the same tacnode process!}

  
  \section{From Aztec Diamond paths to Lattice paths, via Zig-zag paths}\label{sect2}

    {\bf Duality and Zig-zag paths}.  The two sets of lattice paths are dual in the following precise sense. The circle-particles (dot-particles) for one set of lattice paths fill the circle-gaps (dot-gaps) for the other set; see Figure 17. To see this, we define the {\em zig-zag paths}, initiated in \cite{EKLP, EKLP2,Jo03b},. There are two types, zig-zag paths around black squares and zig-zag paths around white squares. 
     
     Remember the axis $Y_{2r}$, connecting $C_{r+1}$ and $A_{r+1}$, as in Figure 8, traverses $2n+2m+1$ {\em black squares} of the double Aztec diamond; it does so in two different ways, which can be described as follows: the black square can belong to a South or West domino, in which case the height $h$ goes up by $1$ from top to bottom, as in the left hand side of figure 12 below; put a blue dot-particle to the left of the oblique line. Or the black square belongs to a North or East domino, in which case the height of the oblique line remains unchanged, as in the right hand side of  figure 12 below; put a red dot-particle to the right of the oblique line.


     \vspace*{-1.5cm}

   \hspace*{-1cm}  \includegraphics[width=131mm,height=154mm]{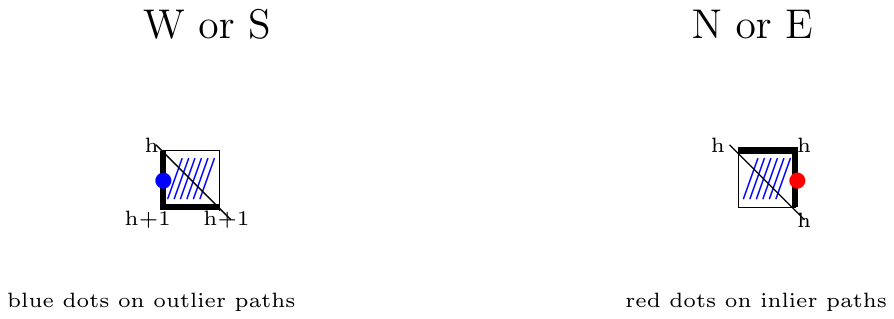}
   
   \vspace*{-11cm}
   
   Figure  12.   Zigzag paths around black squares
   

   
   \vspace*{1cm}

       
   \noindent So, in the left configuration, the line, running from top to bottom goes up by $h=1$  and for the second, it stays flat. Since the line must go up from $h=0$ to $h=2n$, the line must traverse $2n$ black squares of the first type, to reach height $h=2n$ and must stay flat along the $2m+1$ remaining squares, since the total number of black squares the oblique line traverses equals $2n+2m+1$. Therefore the red dots appear there, where the blue dots are absent. Moreover the blue dots must belong to W or S domino's and the red dots to E and N domino's.

   \vspace*{-1.5cm}

   \hspace*{-1cm}  \includegraphics[width=131mm,height=154mm]{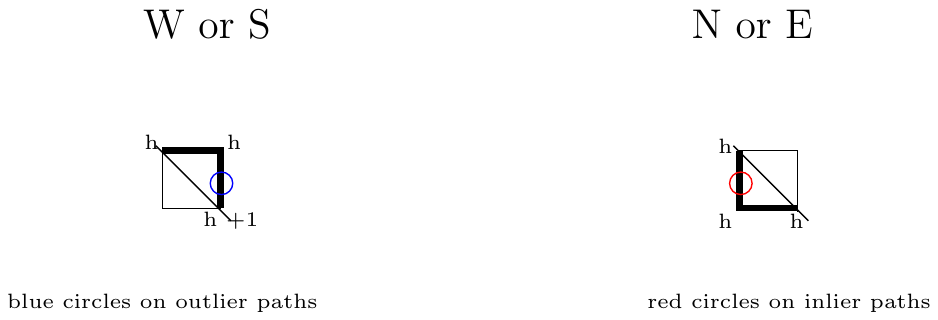}
   
    \vspace*{-11cm}

Figure 13. Zigzag paths around white squares

     \vspace{1cm}

   Similarly, the  axis $Y_{2r-1}$ connecting the points $D_r$ and $B_r$ (as in Figure 8), traverses $2n+2m+1$ {\em white squares}, also starting at height $h=0$ until it hits height $h=2n$. In order to reach that height it must traverse $2n$ white squares of the left type of figure 13, for which the oblique line goes up by $h=1$ and stay flat for the others (white squares of the right type in Figure 13); so, these two situations are dual to each other. Then put a blue circle-particle to the right side of the left square of Figure 13 and a red circle-particle to the left side of the right square of Figure 13.  The white squares, carrying the blue circle-particle (left of Figure 13) must belong to W or S domino's, and those carrying a red circle 
(right of figure 13) must belong to N or E domino's.

So the axis $Y_{2r}$, as in Figure 8, records the location of the blue dot-particles of the line $[C_{r+1},A_{r+1}]$ and the axis $Y_{2r-1}$  records the location of the blue circle-particles of the line $[D_{r},B_{r}]$. To show that the blue dot- and circle-particles belong to the level paths for the height function $h$, as constructed in section 1, it suffices to notice that the blue circle-particles belong to S and W domino's and the blue dot-particles as well, from the previous considerations. Upon connecting them, they are part of level paths for the height function $h$. This shows  that the two constructions, the one using the zig-zag paths and the one using the arguments of section 1 are the same. One does the same for the red dot- and circle-particles, which have been shown to be dual to the blue particles.

Consider now the red dots and red circles on the line halfway in between $Y_{2r}$ and $Y_{2r+1}$, as in Figure 9. Move the red dots horizontally to the left towards the axis $Y_{2r}$ and the red circles to the right towards the axis $Y_{2r+1}$. Then in comparison with Figure 8, the axes  $Y_{2r}$ and $Y_{2r+1}$ will have red dots and red circles precisely at the sites which have no blue dots and circles, showing the duality. 


   \vspace*{-2cm}

 
   \includegraphics[width=121mm,height=140mm]{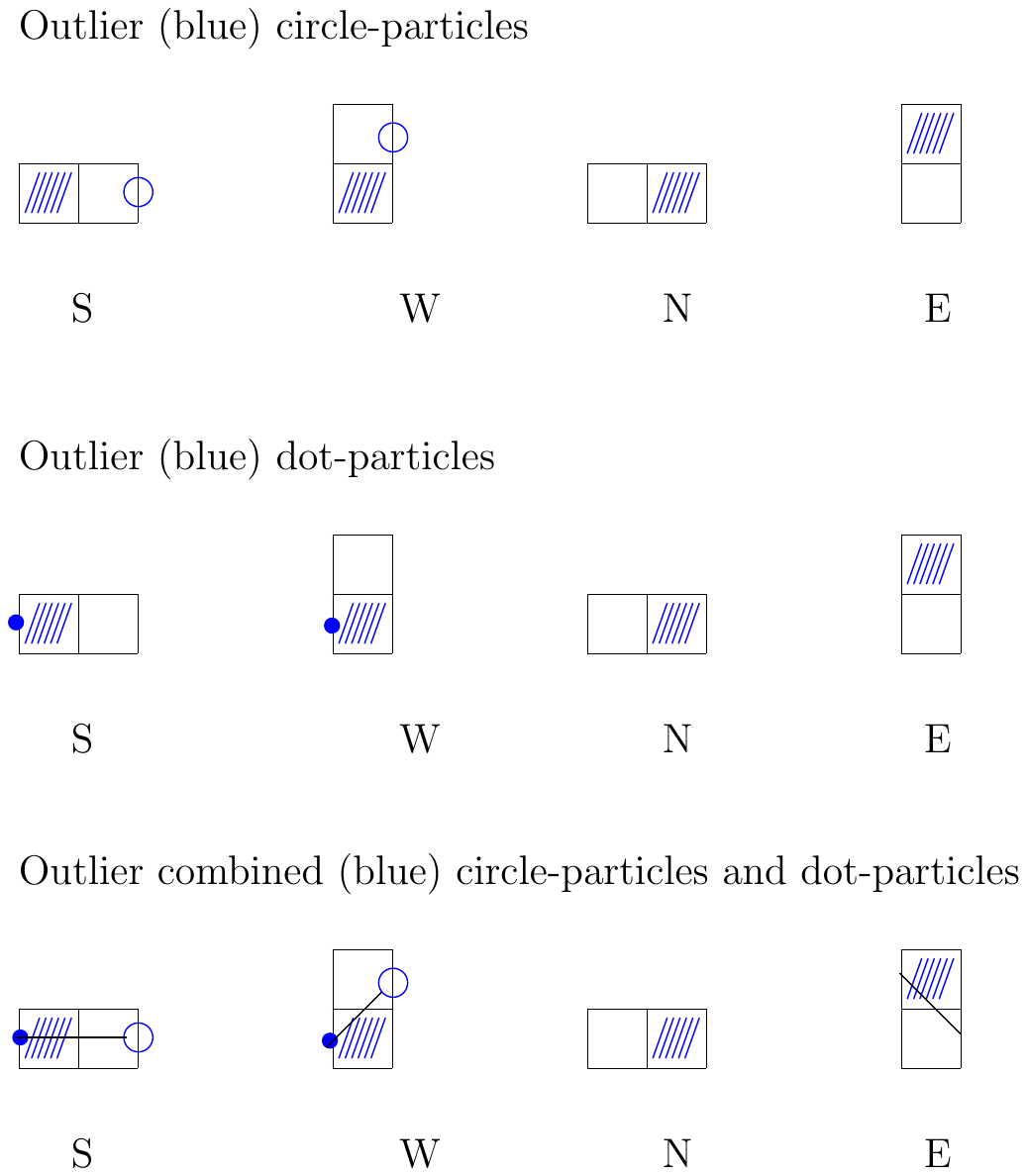}
   
     \vspace*{-5cm}
    
   Figure 14a. The three paths above are level paths for the height function $h$, as in Figure 4.
   
     \vspace*{1cm}

     
      \vspace*{-3cm}

   \includegraphics[width=121mm,height=140mm]{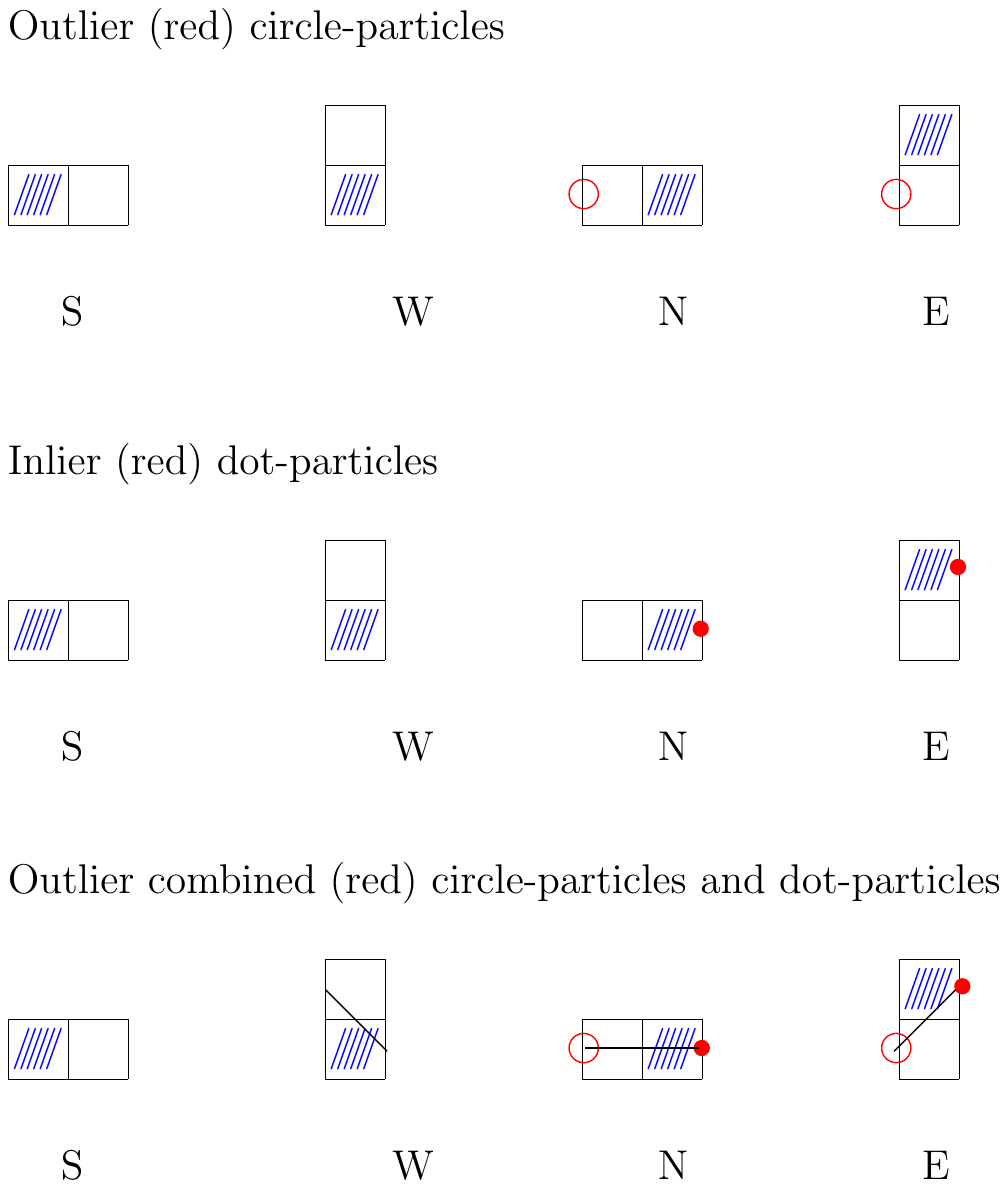}
   
    \vspace*{-5cm}
    
   Figure 14b. The three paths above are level paths for the dual height function $\widetilde h$, as in Figure 5.
   
     \vspace*{1cm}

  {\bf The extended outlier domino-paths are transformed to non-intersecting  paths on a lattice}, by means of the correspondence of Figure 6. Here each path on the double Aztec diamond is made of segments covering 
   West, South and East domino's. The path on the lattice is obtained by sewing together the corresponding segments on the right hand side of Figure 6, with the accompanying dots and circles. 
  To summarize, the outlier lattice-paths will thus contain $2n+1$ steps starting with a dots and ending with a with a circle, so that each path has   $
  \#\{\mbox{dots}\}=\#\{\mbox{circles}\}=n+1
  .$ 
  
  Notice that the location of the circles on the axis $Y_{2r-1}$ and the dots on $Y_{2r}$ in the double Aztec diamond picture (Figure 8) corresponds exactly to the circles and dots on the axes $Y_{2r-1}$ and $Y_{2r}$ of Figure 16. The map of Figure 6 from domino paths to lattice paths was designed to achieve this precise correspondence between the $Y_k$'s in the domino description and the lattice description. 
  
  \bigbreak
 
 {\bf The inlier domino-paths are transformed to non-intersecting lattice paths}, by means of the recipe in Figure 7. Indeed, each path $i_1,\ldots,i_{2m+1}$ on $A\cup B$ consists of segments covering East, North and West domino's, together with their appropriate dots and circles; the recipe is to map  the paths on the domino to a path on a lattice by sewing together the
  
   \bigbreak

  \newpage

\vspace*{0cm}
\setlength{\unitlength}{.020cm}
\textcolor[rgb]{0.00,0.00,1.00}{$ \begin{picture}(0,0)
 %
 %
  \multiput(0,140)(20,00){17}{\circle*{1}} 
  \multiput(0,120)(20,00){17}{\circle*{1}} 
  \multiput(0,100)(20,00){17}{\circle*{1}}
   \multiput(0,80)(20,00){17}{\circle*{1}}
  \multiput(0,60)(20,00){17}{\circle*{1}}
 \multiput(0,40)(20,00){17}{\circle*{1}}
 \multiput(0,20)(20,00){17}{\circle*{1}}
\multiput(0,0)(20,00){17}{\circle*{1}}
\multiput(0,-20)(20,00){17}{\circle*{1}}
\multiput(0,-40)(20,00){17}{\circle*{1}}
\multiput(0,-60)(20,00){17}{\circle*{1}}
\multiput(0,-80)(20,00){17}{\circle*{1}}
\multiput(0,-100)(20,00){17}{\circle*{1}}
 \multiput(0,-120)(20,00){17}{\circle*{1}}
 \multiput(0,-140)(20,00){17}{\circle*{1}}
 \multiput(0,-160)(20,00){17}{\circle*{1}}
 \multiput(0,-180)(20,00){17}{\circle*{1}}
 \multiput(0,-200)(20,00){17}{\circle*{1}}
 \multiput(0,-220)(20,00){17}{\circle*{1}}
   \multiput(0,0)(00,-20){5}{\circle*{5}}  
   \put(0,-80){\line(1,0){20}} 
   \multiput(0,-80)(00,20){5}{\line(1,0){20}} 
   \multiput(20,-60)(00,20){4}{\circle{5}} 
   \multiput(20,-60)(00,20){4}{\line(1,1){20}} 
    \multiput(40,-40)(00,20){4}{\circle*{5}}
     \multiput(60,-40)(00,20){4}{\line(-1,0){20}} 
    \multiput(60,-40)(00,20){4}{\line(1,0){20}} 
    \multiput(60,-40)(00,20){4}{\circle{5}}  
     \multiput(80,-40)(00,20){4}{\circle*{5}}  
    \multiput(80,-40)(00,20){4}{\line(1,0){20}} 
      \multiput(100,-20)(00,20){3}{\circle {5}}  
    \put(20,-80){\circle*{1}}   
   \put(20,-80){\line(0,-1){20}} 
      \put(40,-80){\circle*{5}}   
       \put(40,-80){\line(-1,-1){20}} 
       \put(40,-80){\line(1,0){20}}
        \put(60,-80){\circle{5}}  
        \put(60,-80){\line(1,0){20}}
          \put(80,-80){\circle*{5}}  
           \put(80,-80){\line(1,0){20}}
            \put(100,-80){\circle*{1}} 
            \put(100,-80){\line(0,-1){60}}
            \put(20,-100){\circle{5}}  
            \put(100,-140){\circle {5}} 
            \put(100,-140){\line(1,1){20}}
            \put(120,-120){\circle* {5}}
             \put(120,-120){\line( 1, 0){20}}
               \put(140,-120){\line( 0, -1){20}}
                \put(140,-140){\circle  {5}}
                \put(140,-140){\line( 1, 0){20}}  
                 \put(160,-140){\circle* {5}}
                  \put(160,-140){\line( 1, 0){20}}  
                  \put(180,-140){\circle {5}}
                   \put(180,-140){\line( 1, 1){20}}  
                 \put(200,-120){\circle* {5}}
                  \put(200,-120){\line( 1, 0){20}}  
                 \put(220,-120){\circle  {5}}
                  \put(220,-100){\circle  {5}}
         \put(220,-60){\line( 1, 0){20}}  
                  \put(220,-60){\circle  {5}}
                   \put(220,20){\circle  {5}}
                    \put(220,80){\circle  {5}}
                     \put(220,80){\line( 1, 1){20}}
                     \put(240,100){\circle* {5}}
                 \put(220,-120){\line( 1, 1){20}}  
                 \put(240,-100){\circle* {5}}
                  \put(100,20){\circle{5}}
                   \put(100,20){\line( 1, 1){20}}  
                    \put(120,40){\circle*{5}}
                   \put(120,40){\line( 1, 0){20}}  
                     \put(140,40){\circle{5}}
                   \put(140,40){\line( 1, 1){20}} 
                    \put(160,60){\circle*{5}}
                   \put(160,60){\line( 1, 0){20}} 
                    \put(180,60){\circle {5}}
                   \put(180,60){\line( 1, 1){20}}   
                     \put(200,80){\circle*{5}}
                   \put(200,80){\line( 1, 0){20}} 
                     \put(100, 0){\circle{5}}
                   \put(100,0){\line( 1, 0){20}} 
                    \put(120, 0){\circle*{5}}
                   \put(120,0){\line( 1, 0){20}}
                    \put(140, 0){\circle{5}}
                   \put(140,0){\line( 1, 1){20}} 
                    \put(160, 20){\circle*{5}}
                   \put(160,20){\line( 1, 0){20}}
                   \put(180, 20){\circle{5}}
                    \put(180,20){\line( 1, 1){20}}
                     \put(200, 40){\circle*{5}}
                    \put(200,40){\line( 1, 0){20}} 
                     \put(220,40){\line( 0,-1){20}} 
                     \put(220, 20){\circle{5}}
                    \put(220,20){\line( 1, 0){20}}
                     \put(240,20){\line( 1, 0){20}} 
                     \put(240,20){\circle*{5}}
                     \put(240,20){\line( 1, 0){20}} 
                      \put(260,20){\line( 0, -1){60}} 
                       \put(100,-40){\line( 0,-1){20}}
                        \put(100,-60){\circle{5}}
                          \put(100,-60){\line( 1,0){20}}
                           \put(120,-60){\circle*{5}}
                           \put(120,-60){\line( 1,0){20}}
                      \put(100,-20){\line( 1,0){20}}
                        \put(120,-20){\circle*{5}}
                          \put(120,-20){\line( 1,0){20}}
                           \put(140,-20){\line( 0,-1){20}}
                            \put(140,-40){\circle{5}}
                              \put(140,-40){\line( 1, 0){20}} 
                     \put(160,-40){\circle*{5}}
                     \put(160,-40){\line( 1, 0){20}} 
                      \put(180,-40){\circle {5}}
                      \put(180,-40){\line( 1,0){20}}
                       \put(200,-40){\circle*{5}}
                     \put(200,-40){\line( 1, 0){20}} 
                       \put(220,-40){\line( 0,-1){20}}
              \put(120,-80){\circle*{1}}
               \put(140,-80){\circle*{1}} 
               \put(140,-80){\line(0,-1){20}}
                \put(140,-80){\line(0,1){20}}
                \put(160,-80){\circle*{5}} 
                 \put(160,-80){\line(-1,-1){20}}
                  \put(160,-80){\line(1,0){20}}
                  \put(180,-80){\circle{5}} 
                  \put(180,-80){\line(1,0){20}}
                    \put(200,-80){\circle*{5}}
                    \put(200,- 80){\line(1,0){20}}
                      \put(220,-80){\circle*{1}} 
                       \put(220,-80){\line(0,-1){20}}
                         \multiput(240,-60)(0,-20){3}{\circle*{5}} 
                         \put(240,-80){\line(-1,-1){20}}
                         \put(240,-100){\line(1,0){20}}
   \multiput(260,-100)( 0,20){4}{\circle{4}} 
                            \multiput(260,-100)( 0,20){4}{\line( 1,1){20}}
                              \multiput(280,-80)( 0,20){4}{\line( 1,0){20}} 
                              \multiput(280,-80)( 0,20){4}{\circle*{5}}
                            \multiput(260,-100)( 0,20){3}{\line(-1,0){20}}
                             \multiput(300,-80)( 0,20){5}{\circle{5}}
  %
  %
               \put(140,-100){\circle{5}}
 %
 %
     \put(300,0){\line(0,1){120}} 
        \put(280,120){\line(1,0){20}}
          \put(280,120){\circle*{5}} 
           \put(280,120){\line(-1,-1){20}}
            \put(260,100){\circle{5}} 
             \put(260,100){\line(-1,0){20}}
        \multiput(10,-220)(40,0){8}{\makebox(0,0){\tiny ${B}$}}
         \multiput(30,-220)(40,0){7}{\makebox(0,0){\tiny ${A}$}}
\put(0,-200){\vector(1,0){20}}
\put(0,-200){\vector(0,1){20}}
\put(-5,-180) {\makebox(0,0){\tiny ${x}$}}
 \end{picture} 
 $} 
%


\vspace*{1.5cm}
 
 \hspace*{7cm} \setlength{\unitlength}{.020cm}
 \begin{picture}(0,0)

  \multiput(0,240)(20,00){17}{\circle*{1}} 
  \multiput(0,220)(20,00){17}{\circle*{1}} 
  \multiput(0,200)(20,00){17}{\circle*{1}}
   \multiput(0,180)(20,00){17}{\circle*{1}}
  \multiput(0,160)(20,00){17}{\circle*{1}}
    
  \multiput(0,140)(20,00){17}{\circle*{1}} 
  \multiput(0,120)(20,00){17}{\circle*{1}} 
  \multiput(0,100)(20,00){17}{\circle*{1}}
   \multiput(0,80)(20,00){17}{\circle*{1}}
  \multiput(0,60)(20,00){17}{\circle*{1}}
 \multiput(0,40)(20,00){17}{\circle*{1}}
 \multiput(0,20)(20,00){17}{\circle*{1}}
\multiput(0,0)(20,00){17}{\circle*{1}}
\multiput(0,-20)(20,00){17}{\circle*{1}}
\multiput(0,-40)(20,00){17}{\circle*{1}}
\multiput(0,-60)(20,00){17}{\circle*{1}}
\multiput(0,-80)(20,00){17}{\circle*{1}}
\multiput(0,-100)(20,00){17}{\circle*{1}}
 \multiput(0,-120)(20,00){17}{\circle*{1}}

 \multiput(300, 120)(00,20){7}{\circle {5}}

 \multiput(0,120)( 0,20){7}{\circle*{5}}
  \multiput(0,120)(  0,20){7}{\line( 1,0){20}}
  \multiput(20,120)( 0,20){7}{\circle {5}}
 \multiput(20,120)(  0,20){7}{\line( 1,0){20}}
  \multiput(40,140)( 0,20){6}{\circle* {5}}
 \multiput(40,140)(  0,20){6}{\line( 1,0){20}}
  \multiput(60,140)( 0,20){6}{\circle {5}}
\multiput(60,140)(  0,20){6}{\line( 1,0){20}}
  \multiput(80,140)( 0,20){6}{\circle*{5}}
   \multiput(80,140)(  0,20){6}{\line( 1,0){20}}
  \multiput(100,140)( 0,20){6}{\circle {5}}
   \multiput(100,140)(  0,20){6}{\line( 1,0){20}}
    \multiput(120,160)( 0,20){5}{\circle* {5}}
     \multiput(120,160)(  0,20){5}{\line( 1,0){20}}
       \multiput(140,160)( 0,20){5}{\circle {5}}
     \multiput(140,160)(  0,20){5}{\line( 1,0){20}}

 \put(40, 120){\line(0,-1){80}}
 \put(40,0) {\circle*{5}}

 \multiput(0,0)( 0,-20){7}{\circle*{5}}
  \multiput(0,-20)( 0,-20){6}{\line( 1,0){20}}
  
  \multiput(20,-20)( 0,-20){6}{\circle {5}}
  \multiput(20,-20)( 0,-20){6}{\line( 1,0){20}}
  
  \multiput(40,0)( 0,-20){7}{\circle* {5}}
  \multiput(40, 0)( 0,-20){7}{\line( 1,0){20}}
  
  \multiput(60,0)( 0,-20){7}{\circle {5}}
  \multiput(60, 0)( 0,-20){7}{\line( 1,0){20}}
  
  \multiput(80,0)( 0,-20){7}{\circle* {5}}
  \multiput(80, 0)( 0,-20){3}{\line( 1,1){20}}
  
   \multiput(80, -60)( 0,-20){4}{\line( 1,0){20}}
    \multiput(80,-60)( 0,-20){4}{\circle* {5}}
   
    \multiput(80, -60)( 0,-20){4}{\line( 1,0){20}}
    \multiput(100,-60)( 0,-20){4}{\circle  {5}}
    
      \multiput(100, -60)( 0,-20){4}{\line( 1,0){20}}
    \multiput(120,-60)( 0,-20){4}{\circle* {5}}
   
    \multiput(120, -60)( 0,-20){4}{\line( 1,0){20}}
    \multiput(140,-60)( 0,-20){4}{\circle  {5}}
 
   \multiput(140, -60)( 0,-20){4}{\line( 1,0){20}}
    \multiput(160,-60)( 0,-20){4}{\circle* {5}}
   
    \multiput(160, -60)( 0,-20){4}{\line( 1,0){20}}
    \multiput(180,-60)( 0,-20){4}{\circle  {5}}
  
    \multiput(180, -60)( 0,-20){4}{\line( 1,0){20}}
    \multiput(200,-40)( 0,-20){5}{\circle * {5}}
 
   \multiput(200, -40)( 0,-20){5}{\line( 1,0){20}}%
    \multiput(220,-40)( 0,-20){5}{\circle  {5}}
   
    \multiput(220, -40)( 0,-20){5}{\line( 1,0){20}}
    \multiput(240,-20)( 0,-20){6}{\circle* {5}}
 
  \multiput(240, -20)( 0,-20){6}{\line( 1,0){20}}
    \multiput(260,-20)( 0,-20){6}{\circle  {5}}
   
    \multiput(260, -20)( 0,-20){6}{\line( 1,0){20}}
    \multiput(280, 0)( 0,-20){7}{\circle* {5}}
 
\multiput(280, 0)( 0,-20){7}{\line( 1,0){20}}
    \multiput(300, 0)( 0,-20){7}{\circle  {5}}
 
 \put(280, 0)  {\line( 0,1){80}}
 
    \put(0, 0){\line(1,1){20}}
     \put(20, 20){\circle{5}}
       \put(20, 20){\line(1,0){20}}
        \put(40, 20){\line(0,-1){20}}

\multiput(38,242)(0,-2){63}{\line(0,-1){.2}}
\multiput(278,0)(0,-2){63}{\line(0,-1){.2}}

\multiput(0,240)(40,0){8}{\circle*{5}}
\multiput(20,240)(40,0){8}{\circle {5}}
 \put(0, 240){\line(1,0){320}}
 
 \multiput(0,220)(40,0){7}{\circle*{5}}
\multiput(20,220)(40,0){8}{\circle {5}}
 \put(0, 220){\line(1,0){280}}
  \put(300, 220){\line(1,0){20}}
  
   \multiput(0,200)(40,0){6}{\circle*{5}}
\multiput(20,200)(40,0){6}{\circle {5}}
 \put(0, 200){\line(1,0){240}}
  \put(300, 200){\line(1,0){20}}
  \put(300, 180){\line(1,0){20}}
 \put(300, 160){\line(1,0){20}}
 \put(300, 140){\line(1,0){20}}
 \put(300, 120){\line(1,0){20}}
 
  \put(300, 220){\line(-1,-1){20}}
   \put(300, 200){\line(-1,-1){20}}
  \put(300, 180){\line(-1,-1){20}}
 \put(300, 160){\line(-1,-1){20}}
 \put(300, 140){\line(-1,-1){20}}
 \put(300, 120){\line(-1,-1){20}}
 
  \multiput(280,200)(0,-20){6}{\circle*{5}}
  \multiput(280,180)(0,-20){6}{\line(-1,0){20}}
 \multiput(260,180)(0,-20){6}{\circle {5}}
 
 \put(40, 40){\circle*{5}}
 \put(40, 40){\line( 1,0){40}}
 \put(60, 40){\circle{5}}
 \put(80, 40){\circle*{5}}
 \put(80, 40){\line( 1,1){20}}
 \put(100, 60){\circle{5}}
  \put(100, 60){\line( 1,0){20}}
  \put(120, 60){\circle*{5}}
 \put(120, 60){\line( 1,1){20}}
 \multiput(140, 80)(40,0){3}{\circle{5}}
 \multiput(160, 80)(40,0){3}{\circle*{5}}
  \put(140, 80){\line( 1,0){100}}
  \put(240, 80){\line( 1,1){20}}
  
  \put(120, 140){\line( 0,-1){20}} 
   \put(120, 120){\circle*{5}}
   \put(120, 120){\line( 1,0){40}}
   \put(140, 120){\circle{5}}
    \put(160, 120){\line( 0,-1){20}} 
     \put(160, 100){\line( 1,0){80}} 
     \multiput(160, 100)(40,0){3}{\circle*{5}}
     \multiput(180, 100)(40,0){3}{\circle{5}}  
     \put(240, 100){\line( 1,1){20}}
 
 \multiput(100, 20)(0,-20){3}{\circle{5}}
 \multiput(100, 20)(0,-20){3}{\line( 1,0){20}}
 \multiput(120, 20)(0,-20){2}{\circle*{5}}
  \multiput(120, 20)(0,-20){2}{\line( 1,1){20}}
  
  \put(280, 220){\line( 0,-1){20}}
  \put(240, 200){\line( 0,-1){20}}
  \put(240, 180){\circle*{5}}
  \put(240, 180){\line( 1,0){20}}
  
   \put(160, 180){\circle*{5}}
    \put(160, 180){\line( 1,0){40}}
    \put(180, 180){\circle{5}}
    \put(200, 180){\line( 0,-1){20}}
    \multiput(200, 160)(40,0){2}{\circle*{5}}
     \put(200, 160){\line( 1,0){60}}
      \put(220, 160){\circle{5}}

    \put(160, 160){\line( 0,-1){20}}
    \put(160, 140){\circle*{5}}
     \put(160, 140){\line( 1,0){40}}
      \put(200, 140){\line( 0,-1){20}}
       \put(200, 120){\circle*{5}}
       \put(200, 120){\line( 1,1){20}}
       \put(220, 140){\line( 1,0){40}}
       \put(240, 140){\circle*{5}}
       
    \multiput(180, 140)(40,0){3}{\circle{5}}
    
    \multiput(140, 40)(0,-20){2}{\circle{5}}
    \multiput(160, 40)(40,0){2}{\circle*{5}}
    \put(180, 40) {\circle{5}}
     \put(200, 40){\line( 1,1){20}}
      \put(220, 60) {\circle{5}}
       \put(220, 60){\line( 1,0){20}}
      \put(240, 60) {\circle*{5}}
       \put(240, 60){\line( 1,1){20}}
    \put(140, 40){\line( 1,0){60}}
    
     \put(140, 20){\line( 1,0){20}}
      \put(160, 20){\line( 0,-1){20}}
       \put(160, 0) {\circle*{5}}
       \put(180, 0) {\circle{5}}
       \put(160, 0){\line( 1,0){40}}
       \put(200, 0) {\circle*{5}}
      \put(240, -20){\line( 0,1){40}}
      \put(200, 0){\line( 1,1){20}}
      \put(220, 20){\circle{5}}
      \put(220, 20){\line( 1,0){20}}
      
      \put(120, -20){\line( 0,-1){20}}
      \put(120, -40){\circle*{5}}
       \multiput(140, -20)(40,0){2}{\circle{5}}
      \put(120, -40){\line( 1,1){20}}
      \put(140, -20){\line( 1,0){60}}
       \put(160, -20){\circle*{5}}
        \put(200, -20){\line( 0,-1){20}}

  \put(20,-135) {\makebox(0,0){\tiny ${Y_1}$}}
   \put(40,-135) {\makebox(0,0){\tiny ${Y_2}$}}
    \put(60,-135) {\makebox(0,0){\tiny ${Y_3}$}}
     \put(80,-135) {\makebox(0,0){\tiny ${Y_4}$}}
     \put(100,-135) {\makebox(0,0){\tiny ${Y_5}$}}
     \put(120,-135) {\makebox(0,0){\tiny ${Y_6}$}}
     \put(140,-135) {\makebox(0,0){\tiny ${Y_7}$}}
      \put(160,-135) {\makebox(0,0){\tiny ${Y_8}$}}
     \put(180,-135) {\makebox(0,0){\tiny ${Y_{9}}$}}
     \put(200,-135) {\makebox(0,0){\tiny ${Y_{10}}$}}
       \put(220,-135) {\makebox(0,0){\tiny ${Y_{11}}$}}
         \put(240,-135) {\makebox(0,0){\tiny ${Y_{12}}$}}
           \put(260,-135) {\makebox(0,0){\tiny ${Y_{13}}$}}
             \put(280,-135) {\makebox(0,0){\tiny ${Y_{14}}$}}

              \put(-345,-135) {\makebox(0,0){\tiny ${Y_1}$}}
   \put(-325,-135) {\makebox(0,0){\tiny ${Y_2}$}}
    \put(-305,-135) {\makebox(0,0){\tiny ${Y_3}$}}
     \put(-285,-135) {\makebox(0,0){\tiny ${Y_4}$}}
     \put(-265,-135) {\makebox(0,0){\tiny ${Y_5}$}}
     \put(-245,-135) {\makebox(0,0){\tiny ${Y_6}$}}
     \put(-225,-135) {\makebox(0,0){\tiny ${Y_7}$}}
      \put(-205,-135) {\makebox(0,0){\tiny ${Y_8}$}}
     \put(-185,-135) {\makebox(0,0){\tiny ${Y_{9}}$}}
     \put(-165,-135) {\makebox(0,0){\tiny ${Y_{10}}$}}
       \put(-145,-135) {\makebox(0,0){\tiny ${Y_{11}}$}}
         \put(-125,-135) {\makebox(0,0){\tiny ${Y_{12}}$}}
           \put(-105,-135) {\makebox(0,0){\tiny ${Y_{13}}$}}
             \put(-85,-135) {\makebox(0,0){\tiny ${Y_{14}}$}}
 
 \end{picture} 

\vspace{3cm}

Figure 15: Lattice paths for Fig.4  \hspace{.3cm} Figure 16: Lattice paths for Fig.5 
\\\hspace*{2cm}($2m+1$ Inliers)  \hspace{4cm}  ($2n$ Outliers)

   \vspace*{7cm}
   
 \setlength{\unitlength}{.020cm}
 \begin{picture}(0,0)

  \multiput(0,240)(20,00){17}{\circle*{1}} 
  \multiput(0,220)(20,00){17}{\circle*{1}} 
  \multiput(0,200)(20,00){17}{\circle*{1}}
   \multiput(0,180)(20,00){17}{\circle*{1}}
  \multiput(0,160)(20,00){17}{\circle*{1}}
    
  \multiput(0,140)(20,00){17}{\circle*{1}} 
  \multiput(0,120)(20,00){17}{\circle*{1}} 
  \multiput(0,100)(20,00){17}{\circle*{1}}
   \multiput(0,80)(20,00){17}{\circle*{1}}
  \multiput(0,60)(20,00){17}{\circle*{1}}
 \multiput(0,40)(20,00){17}{\circle*{1}}
 \multiput(0,20)(20,00){17}{\circle*{1}}
\multiput(0,0)(20,00){17}{\circle*{1}}
\multiput(0,-20)(20,00){17}{\circle*{1}}
\multiput(0,-40)(20,00){17}{\circle*{1}}
\multiput(0,-60)(20,00){17}{\circle*{1}}
\multiput(0,-80)(20,00){17}{\circle*{1}}
\multiput(0,-100)(20,00){17}{\circle*{1}}
 \multiput(0,-120)(20,00){17}{\circle*{1}}

 \multiput(300, 120)(00,20){7}{\circle {5}}

 \multiput(0,120)( 0,20){7}{\circle*{5}}
  \multiput(0,120)(  0,20){7}{\line( 1,0){20}}
  \multiput(20,120)( 0,20){7}{\circle {5}}
 \multiput(20,120)(  0,20){7}{\line( 1,0){20}}
  \multiput(40,140)( 0,20){6}{\circle* {5}}
 \multiput(40,140)(  0,20){6}{\line( 1,0){20}}
  \multiput(60,140)( 0,20){6}{\circle {5}}
\multiput(60,140)(  0,20){6}{\line( 1,0){20}}
  \multiput(80,140)( 0,20){6}{\circle*{5}}
   \multiput(80,140)(  0,20){6}{\line( 1,0){20}}
  \multiput(100,140)( 0,20){6}{\circle {5}}
   \multiput(100,140)(  0,20){6}{\line( 1,0){20}}
    \multiput(120,160)( 0,20){5}{\circle* {5}}
     \multiput(120,160)(  0,20){5}{\line( 1,0){20}}
       \multiput(140,160)( 0,20){5}{\circle {5}}
     \multiput(140,160)(  0,20){5}{\line( 1,0){20}}

 \put(40, 120){\line(0,-1){80}}
 \put(40,0) {\circle*{5}}

 \multiput(0,0)( 0,-20){7}{\circle*{5}}
  \multiput(0,-20)( 0,-20){6}{\line( 1,0){20}}
  
  \multiput(20,-20)( 0,-20){6}{\circle {5}}
  \multiput(20,-20)( 0,-20){6}{\line( 1,0){20}}
  
  \multiput(40,0)( 0,-20){7}{\circle* {5}}
  \multiput(40, 0)( 0,-20){7}{\line( 1,0){20}}
  
  \multiput(60,0)( 0,-20){7}{\circle {5}}
  \multiput(60, 0)( 0,-20){7}{\line( 1,0){20}}
  
  \multiput(80,0)( 0,-20){7}{\circle* {5}}
  \multiput(80, 0)( 0,-20){3}{\line( 1,1){20}}
  
   \multiput(80, -60)( 0,-20){4}{\line( 1,0){20}}
    \multiput(80,-60)( 0,-20){4}{\circle* {5}}
   
    \multiput(80, -60)( 0,-20){4}{\line( 1,0){20}}
    \multiput(100,-60)( 0,-20){4}{\circle  {5}}
    
      \multiput(100, -60)( 0,-20){4}{\line( 1,0){20}}
    \multiput(120,-60)( 0,-20){4}{\circle* {5}}
   
    \multiput(120, -60)( 0,-20){4}{\line( 1,0){20}}
    \multiput(140,-60)( 0,-20){4}{\circle  {5}}
 
   \multiput(140, -60)( 0,-20){4}{\line( 1,0){20}}
    \multiput(160,-60)( 0,-20){4}{\circle* {5}}
   
    \multiput(160, -60)( 0,-20){4}{\line( 1,0){20}}
    \multiput(180,-60)( 0,-20){4}{\circle  {5}}
  
    \multiput(180, -60)( 0,-20){4}{\line( 1,0){20}}
    \multiput(200,-40)( 0,-20){5}{\circle * {5}}
 
   \multiput(200, -40)( 0,-20){5}{\line( 1,0){20}}%
    \multiput(220,-40)( 0,-20){5}{\circle  {5}}
   
    \multiput(220, -40)( 0,-20){5}{\line( 1,0){20}}
    \multiput(240,-20)( 0,-20){6}{\circle* {5}}
 
  \multiput(240, -20)( 0,-20){6}{\line( 1,0){20}}
    \multiput(260,-20)( 0,-20){6}{\circle  {5}}
   
    \multiput(260, -20)( 0,-20){6}{\line( 1,0){20}}
    \multiput(280, 0)( 0,-20){7}{\circle* {5}}
 
\multiput(280, 0)( 0,-20){7}{\line( 1,0){20}}
    \multiput(300, 0)( 0,-20){7}{\circle  {5}}
 
 \put(280, 0)  {\line( 0,1){80}}
 
    \put(0, 0){\line(1,1){20}}
     \put(20, 20){\circle{5}}
       \put(20, 20){\line(1,0){20}}
        \put(40, 20){\line(0,-1){20}}

\multiput(38,242)(0,-2){63}{\line(0,-1){.2}}
\multiput(278,0)(0,-2){63}{\line(0,-1){.2}}

\multiput(0,240)(40,0){8}{\circle*{5}}
\multiput(20,240)(40,0){8}{\circle {5}}
 \put(0, 240){\line(1,0){320}}
 
 \multiput(0,220)(40,0){7}{\circle*{5}}
\multiput(20,220)(40,0){8}{\circle {5}}
 \put(0, 220){\line(1,0){280}}
  \put(300, 220){\line(1,0){20}}
  
   \multiput(0,200)(40,0){6}{\circle*{5}}
\multiput(20,200)(40,0){6}{\circle {5}}
 \put(0, 200){\line(1,0){240}}
  \put(300, 200){\line(1,0){20}}
  \put(300, 180){\line(1,0){20}}
 \put(300, 160){\line(1,0){20}}
 \put(300, 140){\line(1,0){20}}
 \put(300, 120){\line(1,0){20}}
 
  \put(300, 220){\line(-1,-1){20}}
   \put(300, 200){\line(-1,-1){20}}
  \put(300, 180){\line(-1,-1){20}}
 \put(300, 160){\line(-1,-1){20}}
 \put(300, 140){\line(-1,-1){20}}
 \put(300, 120){\line(-1,-1){20}}
 
  \multiput(280,200)(0,-20){6}{\circle*{5}}
  \multiput(280,180)(0,-20){6}{\line(-1,0){20}}
 \multiput(260,180)(0,-20){6}{\circle {5}}
 
 \put(40, 40){\circle*{5}}
 \put(40, 40){\line( 1,0){40}}
 \put(60, 40){\circle{5}}
 \put(80, 40){\circle*{5}}
 \put(80, 40){\line( 1,1){20}}
 \put(100, 60){\circle{5}}
  \put(100, 60){\line( 1,0){20}}
  \put(120, 60){\circle*{5}}
 \put(120, 60){\line( 1,1){20}}
 \multiput(140, 80)(40,0){3}{\circle{5}}
 \multiput(160, 80)(40,0){3}{\circle*{5}}
  \put(140, 80){\line( 1,0){100}}
  \put(240, 80){\line( 1,1){20}}
  
  \put(120, 140){\line( 0,-1){20}} 
   \put(120, 120){\circle*{5}}
   \put(120, 120){\line( 1,0){40}}
   \put(140, 120){\circle{5}}
    \put(160, 120){\line( 0,-1){20}} 
     \put(160, 100){\line( 1,0){80}} 
     \multiput(160, 100)(40,0){3}{\circle*{5}}
     \multiput(180, 100)(40,0){3}{\circle{5}}  
     \put(240, 100){\line( 1,1){20}}
 
 \multiput(100, 20)(0,-20){3}{\circle{5}}
 \multiput(100, 20)(0,-20){3}{\line( 1,0){20}}
 \multiput(120, 20)(0,-20){2}{\circle*{5}}
  \multiput(120, 20)(0,-20){2}{\line( 1,1){20}}
  
  \put(280, 220){\line( 0,-1){20}}
  \put(240, 200){\line( 0,-1){20}}
  \put(240, 180){\circle*{5}}
  \put(240, 180){\line( 1,0){20}}
  
   \put(160, 180){\circle*{5}}
    \put(160, 180){\line( 1,0){40}}
    \put(180, 180){\circle{5}}
    \put(200, 180){\line( 0,-1){20}}
    \multiput(200, 160)(40,0){2}{\circle*{5}}
     \put(200, 160){\line( 1,0){60}}
      \put(220, 160){\circle{5}}

    \put(160, 160){\line( 0,-1){20}}
    \put(160, 140){\circle*{5}}
     \put(160, 140){\line( 1,0){40}}
      \put(200, 140){\line( 0,-1){20}}
       \put(200, 120){\circle*{5}}
       \put(200, 120){\line( 1,1){20}}
       \put(220, 140){\line( 1,0){40}}
       \put(240, 140){\circle*{5}}
       
    \multiput(180, 140)(40,0){3}{\circle{5}}
    
    \multiput(140, 40)(0,-20){2}{\circle{5}}
    \multiput(160, 40)(40,0){2}{\circle*{5}}
    \put(180, 40) {\circle{5}}
     \put(200, 40){\line( 1,1){20}}
      \put(220, 60) {\circle{5}}
       \put(220, 60){\line( 1,0){20}}
      \put(240, 60) {\circle*{5}}
       \put(240, 60){\line( 1,1){20}}
    \put(140, 40){\line( 1,0){60}}
    
     \put(140, 20){\line( 1,0){20}}
      \put(160, 20){\line( 0,-1){20}}
       \put(160, 0) {\circle*{5}}
       \put(180, 0) {\circle{5}}
       \put(160, 0){\line( 1,0){40}}
       \put(200, 0) {\circle*{5}}
      \put(240, -20){\line( 0,1){40}}
      \put(200, 0){\line( 1,1){20}}
      \put(220, 20){\circle{5}}
      \put(220, 20){\line( 1,0){20}}
      
      \put(120, -20){\line( 0,-1){20}}
      \put(120, -40){\circle*{5}}
       \multiput(140, -20)(40,0){2}{\circle{5}}
      \put(120, -40){\line( 1,1){20}}
      \put(140, -20){\line( 1,0){60}}
       \put(160, -20){\circle*{5}}
        \put(200, -20){\line( 0,-1){20}}
\end{picture} 

\vspace*{-2.5cm}

\setlength{\unitlength}{.020cm}
\textcolor[rgb]{0.00,0.00,1.00}{$ \begin{picture}(0,0)
 %
 %
  \multiput(0,140)(20,00){17}{\circle*{1}} 
  \multiput(0,120)(20,00){17}{\circle*{1}} 
  \multiput(0,100)(20,00){17}{\circle*{1}}
   \multiput(0,80)(20,00){17}{\circle*{1}}
  \multiput(0,60)(20,00){17}{\circle*{1}}
 \multiput(0,40)(20,00){17}{\circle*{1}}
 \multiput(0,20)(20,00){17}{\circle*{1}}
\multiput(0,0)(20,00){17}{\circle*{1}}
\multiput(0,-20)(20,00){17}{\circle*{1}}
\multiput(0,-40)(20,00){17}{\circle*{1}}
\multiput(0,-60)(20,00){17}{\circle*{1}}
\multiput(0,-80)(20,00){17}{\circle*{1}}
\multiput(0,-100)(20,00){17}{\circle*{1}}
 \multiput(0,-120)(20,00){17}{\circle*{1}}
 \multiput(0,-140)(20,00){17}{\circle*{1}}
 \multiput(0,-160)(20,00){17}{\circle*{1}}
 \multiput(0,-180)(20,00){17}{\circle*{1}}
 \multiput(0,-200)(20,00){17}{\circle*{1}}
 \multiput(0,-220)(20,00){17}{\circle*{1}}
   \multiput(0,0)(00,-20){5}{\circle*{5}}  
   \put(0,-80){\line(1,0){20}} 
   \multiput(0,-80)(00,20){5}{\line(1,0){20}} 
   \multiput(20,-60)(00,20){4}{\circle{5}} 
   \multiput(20,-60)(00,20){4}{\line(1,1){20}} 
    \multiput(40,-40)(00,20){4}{\circle*{5}}
     \multiput(60,-40)(00,20){4}{\line(-1,0){20}} 
    \multiput(60,-40)(00,20){4}{\line(1,0){20}} 
    \multiput(60,-40)(00,20){4}{\circle{5}}  
     \multiput(80,-40)(00,20){4}{\circle*{5}}  
    \multiput(80,-40)(00,20){4}{\line(1,0){20}} 
      \multiput(100,-20)(00,20){3}{\circle {5}}  
    \put(20,-80){\circle*{1}}   
   \put(20,-80){\line(0,-1){20}} 
      \put(40,-80){\circle*{5}}   
       \put(40,-80){\line(-1,-1){20}} 
       \put(40,-80){\line(1,0){20}}
        \put(60,-80){\circle{5}}  
        \put(60,-80){\line(1,0){20}}
          \put(80,-80){\circle*{5}}  
           \put(80,-80){\line(1,0){20}}
            \put(100,-80){\circle*{1}} 
            \put(100,-80){\line(0,-1){60}}
            \put(20,-100){\circle{5}}  
            \put(100,-140){\circle {5}} 
            \put(100,-140){\line(1,1){20}}
            \put(120,-120){\circle* {5}}
             \put(120,-120){\line( 1, 0){20}}
               \put(140,-120){\line( 0, -1){20}}
                \put(140,-140){\circle  {5}}
                \put(140,-140){\line( 1, 0){20}}  
                 \put(160,-140){\circle* {5}}
                  \put(160,-140){\line( 1, 0){20}}  
                  \put(180,-140){\circle {5}}
                   \put(180,-140){\line( 1, 1){20}}  
                 \put(200,-120){\circle* {5}}
                  \put(200,-120){\line( 1, 0){20}}  
                 \put(220,-120){\circle  {5}}
                  \put(220,-100){\circle  {5}}
         \put(220,-60){\line( 1, 0){20}}  
                  \put(220,-60){\circle  {5}}
                   \put(220,20){\circle  {5}}
                    \put(220,80){\circle  {5}}
                     \put(220,80){\line( 1, 1){20}}
                     \put(240,100){\circle* {5}}
                 \put(220,-120){\line( 1, 1){20}}  
                 \put(240,-100){\circle* {5}}
                  \put(100,20){\circle{5}}
                   \put(100,20){\line( 1, 1){20}}  
                    \put(120,40){\circle*{5}}
                   \put(120,40){\line( 1, 0){20}}  
                     \put(140,40){\circle{5}}
                   \put(140,40){\line( 1, 1){20}} 
                    \put(160,60){\circle*{5}}
                   \put(160,60){\line( 1, 0){20}} 
                    \put(180,60){\circle {5}}
                   \put(180,60){\line( 1, 1){20}}   
                     \put(200,80){\circle*{5}}
                   \put(200,80){\line( 1, 0){20}} 
                     \put(100, 0){\circle{5}}
                   \put(100,0){\line( 1, 0){20}} 
                    \put(120, 0){\circle*{5}}
                   \put(120,0){\line( 1, 0){20}}
                    \put(140, 0){\circle{5}}
                   \put(140,0){\line( 1, 1){20}} 
                    \put(160, 20){\circle*{5}}
                   \put(160,20){\line( 1, 0){20}}
                   \put(180, 20){\circle{5}}
                    \put(180,20){\line( 1, 1){20}}
                     \put(200, 40){\circle*{5}}
                    \put(200,40){\line( 1, 0){20}} 
                     \put(220,40){\line( 0,-1){20}} 
                     \put(220, 20){\circle{5}}
                    \put(220,20){\line( 1, 0){20}}
                     \put(240,20){\line( 1, 0){20}} 
                     \put(240,20){\circle*{5}}
                     \put(240,20){\line( 1, 0){20}} 
                      \put(260,20){\line( 0, -1){60}} 
                       \put(100,-40){\line( 0,-1){20}}
                        \put(100,-60){\circle{5}}
                          \put(100,-60){\line( 1,0){20}}
                           \put(120,-60){\circle*{5}}
                           \put(120,-60){\line( 1,0){20}}
                      \put(100,-20){\line( 1,0){20}}
                        \put(120,-20){\circle*{5}}
                          \put(120,-20){\line( 1,0){20}}
                           \put(140,-20){\line( 0,-1){20}}
                            \put(140,-40){\circle{5}}
                              \put(140,-40){\line( 1, 0){20}} 
                     \put(160,-40){\circle*{5}}
                     \put(160,-40){\line( 1, 0){20}} 
                      \put(180,-40){\circle {5}}
                      \put(180,-40){\line( 1,0){20}}
                       \put(200,-40){\circle*{5}}
                     \put(200,-40){\line( 1, 0){20}} 
                       \put(220,-40){\line( 0,-1){20}}
              \put(120,-80){\circle*{1}}
               \put(140,-80){\circle*{1}} 
               \put(140,-80){\line(0,-1){20}}
                \put(140,-80){\line(0,1){20}}
                \put(160,-80){\circle*{5}} 
                 \put(160,-80){\line(-1,-1){20}}
                  \put(160,-80){\line(1,0){20}}
                  \put(180,-80){\circle{5}} 
                  \put(180,-80){\line(1,0){20}}
                    \put(200,-80){\circle*{5}}
                    \put(200,- 80){\line(1,0){20}}
                      \put(220,-80){\circle*{1}} 
                       \put(220,-80){\line(0,-1){20}}
                         \multiput(240,-60)(0,-20){3}{\circle*{5}} 
                         \put(240,-80){\line(-1,-1){20}}
                         \put(240,-100){\line(1,0){20}}
   \multiput(260,-100)( 0,20){4}{\circle{4}} 
                            \multiput(260,-100)( 0,20){4}{\line( 1,1){20}}
                              \multiput(280,-80)( 0,20){4}{\line( 1,0){20}} 
                              \multiput(280,-80)( 0,20){4}{\circle*{5}}
                            \multiput(260,-100)( 0,20){3}{\line(-1,0){20}}
                             \multiput(300,-80)( 0,20){5}{\circle{5}}
  %
  %
               \put(140,-100){\circle{5}}
 %
 %
     \put(300,0){\line(0,1){120}} 
        \put(280,120){\line(1,0){20}}
          \put(280,120){\circle*{5}} 
           \put(280,120){\line(-1,-1){20}}
            \put(260,100){\circle{5}} 
             \put(260,100){\line(-1,0){20}}
  %
   %
%
%
 \end{picture} 
 $}
 
 \vspace{5cm}

Figure 17: Lattice paths: superimposing in- and outliers.

 

 \bigbreak

\noindent paths on the right hand side of Figure 7; this leads to a path on the lattice with alternating circles and dots. As is suggested by Figure 7, a path on a West domino gets transformed in a vertical line (of unit length) on the lattice, possibly repeated as many times as the path traverses West dominoes. The dots as initial boundary condition for the inlier paths translate into a dot with an added horizontal segment on the lattice. The final boundary condition for the inlier path, namely the circle translates into simply putting a circle at the end of the path on the lattice. 
 
 One thus distinguishes two kinds of steps for the inlier lattice paths:
 $$\begin{aligned}
 \mbox{$A$-steps}&=\{\mbox{steps from $\circ$ to $\bullet$}\}=
\{\mbox{steps from $x\rg x$ or $x\rg x+1$}\}\\
  \mbox{$B$-steps}&=\{\mbox{steps from  $\bullet$ to $\circ$ }\}
  =
\{\mbox{steps from $x\rg x-k$, with $k\geq 0$}\}.
 \end{aligned}$$

\section{Expressing the kernel in terms of orthogonal polynomials on the circle}

Consider the following weights on the unit circle in $\BC$,
\be
\rho^L_a(z)=\left(\frac{1+az}{1-\frac{a}{z}}\right)^{n/2}\mbox{~and~}\rho^R_a(z)=\frac{(1+az)^{n/2}}{\left(1-\frac{a}{z}\right)^{n/2+1}},
\label{rho}\ee
and
\be
\rho_a(z):=\rho^L_a(z)\rho^R_a(z)=\frac{(1+az)^n}{\left(1-\frac{a}{z}\right)^{n+1}}.
\label{13}\ee
The following operation $\ast=\ast_{_{\!\! x}}$ will be used throughout\footnote{The following formula involving the $\ast$-operation will often be used, e.g. in (\ref{ES}),
$$
\oint_{\Gamma_{0,a}}\frac{dz}{2\pi iz}g_1(z)z^{-x}~\ast_{_{\!\! x}} ~
 \oint_{\Gamma_{0,a}}\frac{dz}{2\pi iz}g_2(z)z^{ x}
= 
\oint_{\Gamma_{0,a}}\frac{dz}{2\pi iz}g_1(z)
g_2(z).
$$}:
\be
f(x)\ast g(x):=f(x)\ast_{_{\!\! x}} g(x):=\sum_{x\in \BZ}f(x)g(x).
\label{ast}\ee Also remember from (\ref{Idef}) the definition for $0\leq	 s\in \BZ$, and $x,y \in \BZ$,
\be
\begin{aligned}
\psi_{ 2s}(x,y)&= \oint _{\Gamma_{0,a}} \frac{dz}{2\pi iz} z^{x-y}
\left(\frac{1+az}{1-\frac az}\right)^{s},
\end{aligned}
\label{7}\ee
 which satisfies the semi-group property with respect to the operation $\ast$,
\be\psi_{2s}(x,\cdot)\ast\psi_{2r}(\cdot,y)=\psi_{2(s+r)}(x,y),\mbox{   for   }s,r\in\BZ.
\label{ES}\ee
We now state the Proposition below, which uses arguments similar to the ones in \cite{AFvM12}.
\begin{proposition}\label{prop} The point process on the even lines associated with the two groups of non-intersecting outlier lattice paths, corresponding to the double Aztec diamond of type $(n,m)$, is a determinantal point process with the following correlation kernel:
\begin{equation}\begin{aligned}
\lefteqn{ \widetilde {\mathbbm K}_{n,m}^{{\rm ext}}(2r,x;2s,y)} 
\\
&=-\Id_{s<r}\psi_{2(s-r)}(x,y)+\psi_{n-2r}(x,\cdot)\ast\widetilde  {\mathbbm K}_{n,m}^{{\rm ext}}(n,\cdot ~;n,\tc)\ast\psi_{ 2s-n }(\tc~,y),
\end{aligned}
\label{E1}\end{equation}
 expressed in terms of the kernel at the half-way axis $2r=2s=n$ :
$$
\widetilde {\mathbbm K}_{n,m}^{{\rm ext}}(n,x ~;n,y)=\widetilde\BK_{n,m}(x,y):=(\Id - \BK_{n,m})(x,y),
$$
where
\begin{equation}\begin{aligned}
\BK_{n,m}(x,y)& 
= \frac{1}{(2\pi i)^2}\oint_{\Gamma_{0,a}}\frac{dz}{z}\rho^R_a(z)\oint_{\Gamma_{0,a,z}}\frac{dw}{w}\rho^L_a(w)\frac{w^{-y-m}}{z^{-x-m}}\sum_{k=0}^{2m}\hat P_k(z^{-1})P_k(w)
\end{aligned}\label{14}
\end{equation}
with bi-orthonormal polynomials $P_k$ and $\hat P_k$ on the unit circle, for the weight $\rho_a(z)\frac{dz}{z}$.
\end{proposition}
\begin{proof}
{\it Step 1: The kernel at time $2r=2s=n$.} Remember the weight $0<a<1$ for vertical dominoes and the weight $1$ for horizontal dominoes. We first consider the {\it inlier lattice-paths}, as depicted in the right hand side of Figure 7 and Figure 15. At the level of the inlier lattice-paths, the weight will figure in only when $x\rg x+k$ for $k\neq 0$. Therefore, by \cite{Johansson3}, the following weighting holds for $j=0,\ldots, n$ for $B$-steps (from dots to circles) and for $j=0,\ldots, n-1$ for $A$-steps (from circles to dots) \footnote{$B$-steps for outlier paths correspond to $A$-steps for inlier paths.}:
\be
\begin{aligned}
B \mbox{-steps}:&~~~\psi_{2j,2j+1}(x,y): =\left\{ 
\begin{aligned}
 &a^{x-y}, \mbox{  if  } y-x\leq 0
\\
&0 \mbox{  otherwise  }
\end{aligned}
\right\} = \int_{\Gamma_{0,a}}  \frac{dz}{2\pi \I z} ~\frac{z^{x-y}}{1-\frac az} 
\\   \\
A \mbox{-steps}:&~~~\psi_{2j+1,2j+2}(x,y) :=\left\{ 
\begin{aligned}
&a, \mbox{  if  } y-x=1
\\
&1, \mbox{  if  } y-x=0
\\
&0, \mbox{  otherwise  }
\end{aligned}
\right\} =\int_{\Gamma_0}  \frac{dz}{2\pi \I z}  {z^{x-y}}{(1+az)} .
\end{aligned}
\label{6}\ee
 Then define
\be
\begin{aligned}
\psi_{r,s}&=\psi_{r,r+1}\ast \ldots \ast\psi_{s-1,s} , \mbox{  if $s>r$}\\
&=0, \mbox{if $s\leq r$}.
\end{aligned}
\label{sgrp}\ee
Then\footnote{The subscript $0$ will sometimes be omitted in $\psi_{0,k}$, which makes it compatible with formula (\ref{7}).}, for instance, for $k$ even, one needs $\tfrac{k}2$ $B$-steps and $\tfrac{k}2$ $A$-steps to go from $-m+i-1$ to $x$ in time $k$, where $-n-m\leq x\leq n+m$ and $1\leq i\leq 2m+1$,
\be
\begin{aligned}
\psi_{0,k}(-m+i-1, x)&= \oint _{\Gamma_{0,a}} \frac{dz}{2\pi iz} z^{-m+i-1-x}
\left(\frac{1+az}{1-\frac az}\right)^{k/2}
\end{aligned}
\label{7a}\ee
and $n+1-\tfrac k2$ $B$-steps and $n+1-\tfrac k2$ $A$-steps to go from $x$ to $-m+i-1$ in time $N-k$:
\be
\begin{aligned}
\psi_{k,2n+1}( x, -m+i-1)
&=\frac{1}{2\pi \I}\oint _{\Gamma_{0,a}} \frac{dz}{z} z^{m-i+1+x}
 \frac{(1+az)^{(2n-k)/2}}{~(1-\frac az)^{(2n-k+2)/2}} 
\end{aligned}
\label{8}\ee
Then from the equivalence between the double Aztec domino configurations and the inlier paths, from the fact that the weighting (\ref{6}) follows from the probability (\ref{Ptiling}) of a domino tiling and from the
Lindstr\"om-Gessel-Viennot formula for non-intersecting paths, it follows that the probability (inherited from the step weights (\ref{6})), of the paths being at $x_1,\ldots,x_{2m+1}$ at time $r$ is given by
$$
\begin{aligned}
\lefteqn{\BP(r;x_1,\ldots,x_{2m+1})}  \\
&=\frac{1}{Z_{n,m}}
\det (\psi_{0,r}(-m+i-1, x_j))_{1\leq i,j\leq 2m+1}
\det (\psi_{r,2n+1}( x_j, -m+i-1))_{1\leq i,j\leq 2m+1}
\end{aligned}
$$
Hence, by the Eynard-Mehta Theorem \cite{EM97,Jo03b}, the extended kernel is given by (Note $\Id_{r<s}$ can be omitted in view of definition (\ref{sgrp}) of $\psi_{r,s}$)
\be
\begin{aligned}
\lefteqn{\BK^{\rm ext}_{n,m}(r,x;s,y)}\\
&=\sum_{i,j=1}^{2m+1}\psi_{r,2n+1}(x,-m+i-1)([A^{-1}]_{ij})
\psi_{0,s}( -m+j-1,y)
 -\Id_{r<s}\psi_{r,s}(x,y),
 \end{aligned}
\label{9}\ee
with the entries of the $(2m+1)\times (2m+1)$ matrix $A$,
$$
A_{ij} =\psi_{0,2n+1}(-m+i-1,-m+j-1)
=f_i\ast g_j,$$
defined in terms of the functions 
\be\begin{aligned}
f_j(y)=\psi_{0,n}(-m+j-1,y)&=&
 \oint_{\Gamma_{0,a}}\frac{dz}{2\pi iz}z^{-m+j-1-y}\rho^L_a(z) 
\\
g_i(x)=\psi_{n,2n+1}(x,-m+i-1)&=&
 \oint_{\Gamma_{0,a}}\frac{dz}{2\pi iz}z^{m-i+1+x}\rho^R_a(z),
\end{aligned}\label{11}\ee
with $\rho^L_a(z)$ and $\rho^R_a(z)$ as in (\ref{rho}).
 Consider now the kernel at a single (half-way) time $n$, {\em which we already assumed to be even}:
\be\begin{aligned}
\BK_{n,m} (x,y):=\BK_{n,m}^{\rm ext}(r,x;s,y)\Bigr|_{r=s=n}
&=\sum^{2m+1}_{i,j=1}g_i(x)([A^{-1}]_{ij})f_j(y).
\end{aligned}\label{10}\ee
Then one performs row operations in the determinants of the Lindstr\"om-Gessel-Viennot formula; this amounts to taking linear combination $\tilde f_k(x)$ and $\tilde g_k(x)$ of the $f_i(x)$ and $g_i(x)$ with $1\leq i\leq k$,
\begin{equation}\begin{aligned}
\tilde f_k(y)&= \frac{1}{2\pi i}\int_{\Gamma_{0,a}}\frac{dz}{z} \rho^L_a(z)\frac{P_{k-1}(z)}{z^{y+m}}\\
\\
\tilde g_k(x)&= \frac{1}{2\pi i}\int_{\Gamma_{0,a}}\frac{dz}{z}  \rho^R_a(z)z^{x+m}\hat P_{k-1}(z^{-1})    
\end{aligned}
\label{12}\end{equation}
such that the new matrix $A_{k,\ell}$ in (\ref{10}) becomes the identity; i.e.,
\be\begin{aligned}
\tilde f_k(x)\ast_{_{\!\! x}} \tilde g_{\ell}(x)=
\sum_{x\in\BZ}\tilde f_k(x)\tilde g_{\ell}(x)&=\frac{1}{2\pi i}\int_{\Gamma_{0,a}}P_{k-1}(z)\hat P_{\ell -1}(z^{-1})\rho_a(z)\frac{dz}{z}\\
&= :\ll P_{k-1},\hat P_{\ell -1}\gg =\dt_{k,\ell},
\end{aligned}\label{12'}\ee
thus leading to orthogonal polynomials on the circle for the inner product
 $\ll ~,~\gg$ with regard to the weight $\rho_a(z)\frac{dz}{z}$, as defined in (\ref{13}). 
With these new functions $\tilde f_k(y)$ and $\tilde g_k(x)$, the kernel (\ref{10}) has the simple form:
\begin{equation}\begin{aligned}
\BK_{n,m}(x,y)&= \sum^M_{k=1}\tilde g_k(x)\tilde f_k(y) \label{14'}
\\
&= \frac{1}{(2\pi i)^2}\oint_{\Gamma_{0,a}}\frac{dz}{z}\rho^R_a(z)\oint_{\Gamma_{0,a,z}}\frac{dw}{w}\rho^L_a(w)\frac{w^{-y-m}}{z^{-x-m}}\sum_{k=0}^{2m}\hat P_k(z^{-1})P_k(w),
\end{aligned}
\end{equation}
where it is legitimate to include $z$ in the contour of integration, since the $w$-part of the integrand is holomorphic in $w\neq 0,a$.
According to \cite{BOO00}, the statistics of the non-intersecting outlier lattice paths is then described by its dual kernel, 
\be
\begin{aligned}
\widetilde {\mathbbm K} _{n,m}(x,y)=\dt_{x,y}-{\mathbbm K} _{n,m}( x; y).
\end{aligned}
\label{dualK}\ee
This establishes formula (\ref{E1}) for $2r=2s=n$, with kernel $  {\mathbbm K} _{n,m}(x,y)$ as in (\ref{14}).  \vspace{.3cm}


{\it Step 2 : The extended kernel}. 
%
%
Comparing the form of the extended kernel ${\mathbbm K}^{{\rm ext}}_{n,m}(2r,x;2s,y)$ in (\ref{9}) and the form of the kernel ${\mathbbm K}_{n,m}
  (x,y) = {\mathbbm K}_{n,m} ^{{\rm ext}}
   (n,x;n,y)$ in (\ref{10}), 
    and using the semi-group relations (\ref{ES}) on the $\psi$-functions, namely
   \be\begin{aligned}
\psi_{n-2r}(x,\cdot)\ast \psi_{n,2n+1}(\cdot,-m+i-1)&=\psi_{2r,2n+1}(x,-m+i-1)\\
 \\
\psi_{0,n}(-m+j-1,\cdot)\ast\psi_{2s-n}(\cdot,y)&=\psi_{0,2s}(-m+j-1,y),
\end{aligned}\label{Ed}\ee
  one checks that (these arguments appear in Lemma 5.2 of \cite{AFvM12})
 $$  
\begin{aligned}
\lefteqn{{\mathbbm K}^{\rm ext}_{n,m}(2r,x;2s,y)}\\
&=
-\Id_{r<s}\psi_{2(s-r)}(x,y)
+\psi_{n-2r}(x,\cdot)\ast {\mathbbm K}^{\rm ext}
_{n,m}(n,\cdot~;n,\tc)\ast\psi_{ 2s-n }(\tc,y).
\end{aligned}
$$
%
The dual extended kernel is then given by
$$
\begin{aligned}
\lefteqn{\widetilde {\mathbbm K}_{n,m}^{{\rm ext}}(2r,x;2s,y)=\dt_{rs}\dt_{xy}- {\mathbbm K}_{n,m}^{{\rm ext}}(2r,x;2s,y)}  
\\
&=\dt_{rs}\dt_{xy}+\Id_{r<s}\psi_{(2s-2r)}(x,y)-\psi_{n-2r}(x,\cdot)\ast {\mathbbm K}^{{\rm ext}}_{n,m}(n,\cdot;n,\tc)\ast\psi_{ 2s-n }(\cdot,y) 
\\
&=(\dt_{rs}\dt_{xy}+\Id_{r<s}\psi_{(2s-2r)}(x,y)-\psi_{n-2r}(x,\cdot)\Id (\cdot,\tc)\ast \psi_{ 2s-n }(\tc,y)) 
\\
&\quad +\psi_{n-2r}(x,\cdot)\ast (\Id -{\mathbbm K}_{n,m} )(\cdot,\tc)\ast\psi_{ 2s-n }(\tc,y) 
\\
&=-\Id_{s<r}\psi_{(2s-2r)}(x,y)+\psi_{n-2r}(x,\cdot)\ast\widetilde {\mathbbm K} _{n,m}( \cdot~; \tc)\ast\psi_{(2s-n)}(\tc ,y),
\end{aligned}
$$
where one uses (\ref{dualK}) and the identity
$$
\begin{aligned}
\lefteqn{ \hspace*{-1cm}~ \dt_{rs}\dt_{xy}+\Id_{r<s}\psi_{ 2s-2r }(x,y)-\psi_{n-2r}(x,\cdot)\ast\Id(\cdot,\tc)\ast\psi_{ 2s-n }(\tc,y)~~~~~~~~~~~}\\
&~~\hspace*{4cm}~~~~~~=-\Id_{s<r}\psi_{(2s-2r)}(x,y),
\end{aligned}
$$
concluding the proof of Proposition \ref{prop}.

\end{proof}


\section{Orthogonal polynomials on the circle 
and Fredholm determinants}

In order to be able to use formula (\ref{14}) for the kernel ${\mathbb K}_{n,m}(x,y)$, one needs to express both, orthogonal polynomials and Darboux sums of orthogonal polynomials on the circle, in terms of Fredholm determinants of certain kernels. 
 In the last part, we shall state some Fredholm determinants identities from \cite{Joh10}.

 In order to do this, one first introduces the inner-product
\begin{equation}\label{in-prod3}
\la f,g\ra_{{\bf t},{\bf s}}:=\frac{1}{2\pi i}\oint_{S^1}\frac{du}{u}f(u)g(u^{-1})e^{\sum_{j=1}^{\infty}(t_ju^j-s_ju^{-j})},
\end{equation}
upon setting ${\bf t}:=(t_1,t_2,\ldots)\in \BC^{\infty}$ and ${\bf s}:=(s_1,s_2,\ldots)\in \BC^{\infty}$. The inner-product $\ll~,~\gg$, as in (\ref{12'}), is a special instance of $ \la ~,~\ra_{{\bf t},{\bf s}}$ upon picking the special locus 
\be
 \LR=\{({\bf t},{\bf s})\mbox{  with  } t_j=-n \frac{(-a)^j}j \mbox{  and  }  s_j=-(n+1)\frac{a^j }j
 \}.\label{23}\ee
 Indeed, 
 \be
\begin{aligned}
e^{\sum_1^{\infty} t_ju^j}\Bigr|_{\LR}&=(1+au)^n
 ,~~~~
e^{ \sum_1^{\infty} s_ju^{-j}}\Bigr|_{\LR} =\bigl(1-\frac au\bigr)^{  n+1 }.
\end{aligned}
\label{24}\ee
and thus
$$
e^{\sum_{j=1}^{\infty}(t_ju^j-s_ju^{-j})}\Bigr|_\LR=
\frac {(1+au)^n}{ (1-\tfrac au)^{  n+1 }}=  \rho_a(u)
.$$
 Also notice that (see definition (\ref{Idef}))
 \be
e^{\sum_{j=1}^{\infty}(t_ju^j+s_ju^{-j})}\Bigr|_\LR=
  {(1+au)^n}{ (1-\tfrac au)^{  n+1 }}=   \vp_a(2n;u)
.\label{psi}\ee

 \begin{lemma}\label{L31} Given the biorthonormal polynomials $P_{k }$ and $\widehat P_{\ell }$,
 $$
 \frac{1}{2\pi i}\int_{\Gamma_{0,a}}P_{k }(z)\widehat P_{\ell }(z^{-1})\rho_a(z)\frac{dz}{z}  =\dt_{k,\ell},
$$
the following identities hold:
\be\begin{aligned}
P_k(z)&=  \frac{z^k\left(1-\frac{a}{z}\right)^{n+1}H_k^{(1)}(z^{-1})}{\sqrt{H_k(0)H_{k+1}(0)}}
\\
\widehat P_k(w^{-1})&= \frac{w^{-k}(1+aw)^{-n}H_k^{(2)}(w)}{\sqrt{H_k(0)H_{k+1}(0)}} 
\\
\sum_{j=0}^k P_j(z)
 \widehat P_j(w^{-1})&=
 \left(\frac zw\right)^k\frac{ (1-\frac az)^{n+1}}{(1+aw)^{n}(1-\frac wz)}
\frac{H_k^{(3)}(z^{-1},w)}{H_{k+1} (0)},
\end{aligned}
\label{31}\ee
where
\be
\begin{aligned}
H_k^{(1)}(z^{-1})&:= \det(I-K^{(1)}(z^{-1}))_{\ell^2(k,k+1,...)} 
\\
H_k^{(2)}(z)&:= \det(I-K^{(2)}(z))_{\ell^2(k,k+1,...)}
\\
H_k^{(3)}(z^{-1},w)&:= \det(I-K^{(3)}(z^{-1},w))_{\ell^2(k,k+1,...)},
\end{aligned}
\label{29}\ee
with
\begin{equation}
\begin{aligned}
\lefteqn{K^{(1)}_{k,\ell}(z^{-1}) ~~~~~~~~\mbox{for $|z|>|u|,~|v|$}}
\\%
&:=\frac{(-1)^{k+\ell}}{(2\pi i)^2}  \oint_{|u|=\rho<1}du\oint_{ {|v|=\rho^{-1}>1}}dv\frac{u^{\ell}}{v^{k+1}}
 \frac{1}{v-u}
\frac{z-u}{z-v} \frac
 {\vp_a(2n,u)
 }
   {\vp_a(2n,v)
   }  ,
\\
 &= K^{(1)}_{k,\ell}(0)+h^{(1)}_{k}(z^{-1})g^{(1)}_{\ell}(n)
 \end{aligned}\label{25}\end{equation}
 \begin{equation}
\begin{aligned}
\lefteqn{{K^{(2)}_{k,\ell}(w)}~~~~~\mbox{for $|w^{-1}|>|u|,~|v|$}}\\
&:=\frac{(-1)^{k+\ell}}{(2\pi i)^2}  \oint_{|u|=\rho<1}du\oint_{ {|v|=\rho^{-1}>1}}dv\frac{u^{\ell}}{v^{k+1}}
 \frac{1}{v-u}
\Bigl(\frac{w^{-1}-u}{w^{-1}-v}\Bigr)\frac{\vp_a(2n,v^{-1})
}{\vp_a(2n,u^{-1})
}
\\
 &= K^{(2)}_{k,\ell}(0)+h^{(2)}_{k}(w)g^{(2)}_{\ell}(n)
\\ 
 \end{aligned}\label{26}\end{equation}
 \begin{equation}\begin{aligned}
  \lefteqn{K^{(3)}_{k ,\ell }(z^{-1},w) 
   ~~~\mbox{for $a<r_2=|w|<s_2<s_1<r_1=|z|<a^{-1}$}}\\
 &:=\frac{(-1)^{k+\ell}}{(2\pi i)^2}  \oint_{\gamma_{s_2}}du\oint_{ {\gamma_{s_1}}}dv\frac{u^{\ell +1}}{v^{k+2}}
 \frac{1}{v-u}
 \frac{(u-z)(v-w)}{(v-z)(u-w)} 
 \frac {\vp_a(2n,u)
 }
   {\vp_a(2n,v)
   } 
 \\
 &
 = K^{(1)}_{k+1,\ell+1}(0)+(\frac zw-1)h^{(1)}_{k+1}(z^{-1}) h^{(2)}_{\ell+1}(w),
 \end{aligned}\label{25'}\end{equation}
$K^{(1)}(z^{-1})$, $K^{(2)}(w)$ and $K^{(3)}(z^{-1},w)$, being rank $1$ perturbations of $K^{(1)}(0) $, as kernels in $k$ and $\ell$.  
In these formulas,
 \be \begin{aligned}
g^{(1)}_{\ell}(n)&= -\oint_{\Gamma_{0}} \frac{du}{2\pi i}(-u)^\ell \vp_a(2n;u)
 =-\oint_{\Gamma_{0}} \frac{du}{2\pi i}(-u)^{-\ell-2} \vp_a(2n;u^{-1})%
\\
g^{(2)}_{\ell}(n)&= -\oint_{\Gamma_{0,-a}} \frac{du}{2\pi i}\frac{(-u)^{\ell}}{\vp_a(2n,u^{-1})
}
= -\oint_{\Gamma_{0,a}} \frac{du}{2\pi i}\frac{(-u)^{-\ell-2}}{
\vp_a(2n,u)
},
\end{aligned}\label{34}\ee
and\footnote{$\Gamma_{0,a}$ in the first integral can (and will be) be realized as $\gamma_{s }$ with $~a<|v|=s <|z|$. Also $\Gamma_{0}$ in the second integral can be realized as $\gamma_{s }$, with $|v|=s <|w^{-1}|$. Finally $\Gamma_{0,w}$ in the third integral can be realized as $\gamma_{s }$, with $|w|<s =|u|$.  }
$$ \begin{aligned}
h^{(1)}_{k}(z^{-1})&
:=-\oint_{\Gamma_{0,a}} \frac {dv}{2\pi i(v-z)} \frac {1}{(-v)^{k+1}\vp_a(2n,v)}    
 \\&=-\sum_{\alpha=0}^\infty (-z)^\alpha g^{(2)}_{k+\alpha}(n)
+\frac{1
}{(-z)^{k+1}\vp_a(2n,z)}
\end{aligned}
$$
\be \begin{aligned} 
h^{(2)}_{\ell}(w)&:= -\oint_{\Gamma_{0}} \frac{dv}{2\pi i(v-w^{-1})}\frac{\vp_a(2n;v^{-1})}{(-v)^{\ell+1}} 
 \\&  =  -\oint_{\Gamma_{0,w}} \frac {w~du}{2\pi i(u-w)}  {(-u)^{\ell}\vp_a(2n,u)}\\
 &=-\sum_{\alpha=0}^\infty (-w)^{-\alpha} g^{(1)}_{\ell+\alpha}(n)
+(-w)^{\ell+1}\vp_a(2n,w)
~.~~~
\end{aligned}\label{25h}\ee
The unperturbed kernel can then be expressed in terms of the $g^{(i)}$'s:
\be \begin{aligned}
K_{k,\ell}^{(1)}(0)=\sum_{\alpha=0}^{\infty}g^{(1)}_{\ell+\alpha}(n)g^{(2)}_{k+\alpha} (n)=K_{\ell,k}^{(2)}(0)
,\label{35a} \end{aligned}\ee
and therefore $K^{(1) \top}_{}(0)=K^{(2)}_{}(0)$.

 \end{lemma}

 \begin{proof}
It was shown in \cite{AvM97}
  (see also the lecture notes \cite{vM10}) 
that the functions\footnote{ For $\alpha\in \BC$, one defines $[\alpha]=\left(\alpha,\frac{\alpha^2}{2},\frac{\alpha^3}{3},\ldots\right)\in \BC^{\infty}$.}
\begin{equation}\label{tau1}
\begin{aligned}
p_k^{(1)}({\bf t},{\bf s};z)&:=z^k\frac{\tau_k({\bf t}-[z^{-1}],{\bf s})}{\sqrt{\tau_k({\bf t},{\bf s})\tau_{k+1}({\bf t},{\bf s})}}
\\
p_k^{(2)}({\bf t},{\bf s};w^{-1})&:=w^{-k}\frac{\tau_k({\bf t},{\bf s}+[w])}{\sqrt{\tau_k({\bf t},{\bf s})\tau_{k+1}({\bf t},{\bf s})}}
\end{aligned}
\end{equation}
are bi-orthonormal polynomials  with regard to the inner-product (\ref{in-prod3}), i.e., $\la p_k^{(1)},p_\ell^{(2)}\ra_{\bf t,{\bf s}}=\dt_{k,\ell}$ and that
\be
\sum_{j=0}^k 
p_j^{(1)}({\bf t},{\bf s};z)
~p_j^{(2)}({\bf t},{\bf s};w^{-1})
=\left(\frac zw\right)^k\frac{\tau_k({\bf t}-[z^{-1}],{\bf s}+[w])}{\tau_{k+1}({\bf t},{\bf s})}
\label{tau1'}\ee

In the formulae above, the $\tau_n({\bf t},{\bf s})$ are 2-Toda $\tau$-functions and are defined as a Toeplitz determinant, which is also expressible as a Fredholm determinant of the kernel (\ref{20}) below, using the Borodin-Okounkov identity~\cite{BO99}. We obtain two formulas, one obtained from the other by substitution $u\mapsto 1/u$:
%
 \begin{equation}
\begin{aligned}
\tau_p({\bf t},{\bf s})&:=\det\left[\frac{1}{2\pi i}\oint_{S^1}\frac{du}{u} u^{k-\ell}e^{\sum_{j=1}^\infty (t_ju^j-s_ju^{-j})}\right]_{1\leq k,\ell\leq p}\\
&=\det\left[\frac{1}{2\pi i}\oint_{S^1}\frac{du}{u} u^{k-\ell}e^{\sum_{j=1}^\infty (t_ju^{-j}-s_ju^{j})}\right]_{1\leq k,\ell\leq p}\\
&=Z({\bf t},{\bf s})\det\left(\Id-\mathbf{K}({\bf t},{\bf s})\right)_{\ell^2(\{p,p+1,\ldots\})},\quad Z({\bf t},{\bf s}):=e^{-\sum_{j=1}^{\infty}j\,t_js_j},
\end{aligned}
\label{19}\end{equation}
and the kernel $\mathbf{K}({\bf t},{\bf s})$ is given by two seemingly different expressions, upon using the two different expressions for $\tau_p$ above, related by $t\leftrightarrow -s$:
 \begin{equation}
\begin{aligned}
\label{20}
\mathbf{K}({\bf t},{\bf s})_{k,\ell}&=\frac{1}{(2\pi i)^2} \oint_{|u|=\rho<1}du\oint_{\Gamma_{|v|=\rho^{-1}>1}}dv\frac{u^{\ell}}{v^{k+1}}\frac{1}{v-u} ~~\frac{e^{\sum_{j=1}^{\infty}(t_jv^{-j}+s_jv^j)}}{e^{\sum_{j=1}^{\infty}(t_ju^{-j}+s_ju^j)}}.
\\
  &=\frac{1}{(2\pi i)^2} \oint_{|u|=\rho<1}du\oint_{\Gamma_{|v|=\rho^{-1}>1}}dv\frac{u^{\ell}}{v^{k+1}}\frac{1}{v-u} 
  \frac
 {e^{\sum_{j=1}^{\infty}
 ( t_ju^{j}+s_ju^{-j})  
 }}
  {e^{\sum_{j=1}^{\infty}( t_jv^{j} +s_jv^{-j})}}. 
\end{aligned}
\end{equation}
This is valid insofar $\sum_{j=1}^{\infty}
( t_ju^{j}+s_ju^{-j})$ is analytic in the annulus $(\rho,\rho^{-1})$. 
In view of the expressions (\ref{tau1}) and (\ref{tau1'}), involving shifts, we now compute, using $\sum_{j=1}^\infty (v/z)^j/j=-\ln(1-v/z)$ (for $|v/z|<1$) in the second formula (\ref{20}),
\begin{equation}
\begin{aligned}
\lefteqn{\mathbf{K}({\bf t}-[z^{-1}],{\bf s})_{k,\ell},~~\mbox{with $|z|>|u|,~|v|$}}\\
&=\frac{1}{(2\pi i)^2}  \oint_{|u|=\rho<1}du\oint_{ {|v|=\rho^{-1}>1}}dv\frac{u^{\ell}}{v^{k+1}}
 \frac{1}{v-u}
\frac{1-\frac{u}{z}}{1-\frac{v}{z}} \frac
 {e^{\sum_{j=1}^{\infty}(t_ju^{j}+s_ju^{-j})}}
   {e^{\sum_{j=1}^{\infty}(t_jv^{j}+s_jv^{-j})}}  
   \\
 \lefteqn{\mathbf{K}({\bf t}-[z^{-1}],{\bf s}+[w])_{k,\ell}~~~~~\mbox{with $|w|<|u|,~|v|$ and $|u|,~|v|<|z|$}}\\  
 &=\frac{1}{(2\pi i)^2}  \oint_{|u|=\rho<1}du\oint_{ {|v|=\rho^{-1}>1}}dv\frac{u^{\ell}}{v^{k+1}}
 \frac{1}{v-u}
\frac{(1-\frac{u}{z})(1-\frac wv)}{(1-\frac{v}{z})(1-\frac wu)} \frac
 {e^{\sum_{j=1}^{\infty}(t_ju^{j}+s_ju^{-j})}}
   {e^{\sum_{j=1}^{\infty}(t_jv^{j}+s_jv^{-j})}}
 \end{aligned}
 \label{21}\end{equation}
and from the first formula (\ref{20}),
\begin{equation}
\begin{aligned}
\lefteqn{\mathbf{K}({\bf t},{\bf s}+[w])_{k,\ell}
,~~\mbox{with $|w^{-1}|>|u|,~|v|$}}\\
&=\frac{1}{(2\pi i)^2}  \oint_{|u|=\rho<1}du\oint_{ {|v|=\rho^{-1}>1}}dv\frac{u^{\ell}}{v^{k+1}}
 \frac{1}{v-u}
\frac{1-\frac{u}{w^{-1}}}{1-\frac{v}{w^{-1}}}\frac{e^{\sum_{j=1}^{\infty}(t_jv^{-j}+s_jv^j)}}{e^{\sum_{j=1}^{\infty}(t_ju^{-j}+s_ju^j)}}.
 \end{aligned}\label{22}\end{equation}
 We now apply this to the special locus $\LR$. 
Then picking $\rho$ such that $a<\rho<1$ and $\rho^{-1}<|z|, ~|w|^{-1}$, one deduces from (\ref{21}) and (\ref{22}) the three kernels (\ref{25}), (\ref{25'}) and (\ref{26}):
\begin{equation}
\begin{aligned}
  \mathbf{K}({\bf t}-[z^{-1}],{\bf s})_{k,\ell}\Bigr |_{\LR}&=(-1)^{k+\ell}K^{(1)}_{k,\ell}(z^{-1}) \\
   \mathbf{K}({\bf t}-[z^{-1}],{\bf s}+[w])_{k,\ell}\Bigr |_{\LR} &=
   (-1)^{k+\ell}K^{(3)}_{k ,\ell }(z^{-1},w) 
 %
\\
 \mathbf{K}({\bf t},{\bf s}+[w])_{k,\ell}\Bigr |_{\LR} 
&=(-1)^{k+\ell}K^{(2)}_{k,\ell}(w)
 \end{aligned}\label{26'}\end{equation}
%
 In order to compute the polynomials (\ref{tau1}), for $\tau_k$ given by (\ref{19}), one needs to evaluate all the functions involved along the locus $\LR$; in particular, from (\ref{19}) and (\ref{23}),
\be
Z({\bf t},{\bf s})\Bigr|_{\LR}=e^{-\sum_{j=1}^{\infty}j\,t_js_j}\Bigr|_{\LR}=(1+a^2)^{n(n+1)}\neq 0
,\label{27}\ee
and from (\ref{24}), one finds for $a<|z|$ and $|w|<a^{-1}$,
\begin{equation}\begin{aligned}
\frac{Z({\bf t},{\bf s}+[z^{}])}{Z({\bf t},{\bf s})}\Bigr|_{\LR}&=\frac{1}{{Z({\bf t},{\bf s})}}e^{-\sum_{j=1}^{\infty}j\, t_j(s_j +z^{j}/j)}\Bigr|_{\LR}= e^{-\sum_{j=1}^{\infty}t_jz^{j}}\Bigr|_{\LR}= (1+az)^{-n}.
\\
\frac{Z({\bf t}-[z^{-1}],{\bf s})}{Z({\bf t},{\bf s})}\Bigr|_{\LR}&=\frac{1}{{Z({\bf t},{\bf s})}}e^{-\sum_{j=1}^{\infty}j\, (t_j-z^{-j}/j)s_j }\Bigr|_{\LR}=  e^{ \sum_{j=1}^{\infty}s_jz^{-j}}\Bigr|_{\LR}= (1-\frac az)^{n+1}
\\  
\frac{Z({\bf t}-[z^{-1}] ,{\bf s}+[w])}{Z({\bf t},{\bf s})}& \Bigr|_{\LR}=  e^{ \sum_{j=1}^{\infty}(s_jz^{-j}-t_jw^j+z^{-j}w^j/j)}\Bigr|_{\LR}=\frac{ (1-\frac az)^{n+1}}{(1+aw)^{n}(1-\frac wz)}.
\end{aligned}\label{28}\end{equation}
From the definitions of the Fredholm determinants (\ref{29}), we have
%
from the expressions (\ref{19}), (\ref{27}), (\ref{25}) and (\ref{26}) in terms $ \mathbf{K}({\bf t}-[z^{-1}],{\bf s})$ and $\mathbf{K}({\bf t},{\bf s}+[w])$,  that
\be
\tau_k({\bf t},{\bf s})\big\vert_{\LR}=Z(t,s)H_k^{(1)}(0)=Z(t,s)H_k^{(2)}(0)=Z(t,s)H_k^{(3)}(0,0).
\label{30}\ee
Then, setting $H_k(0):=H_k^{(1)}(0)=H_k^{(2)}(0)=H_k^{(3)}(0,0)$, one finds for the orthonormal polynomials (\ref{tau1}) with regard to the inner-product $\ll~,~\gg=\la ~,\ra_{{\bf t},{\bf s}}\Bigr|_{\LR}$, and for the Darboux sum (\ref{tau1'}):
(remember notation (\ref{three}))%
\be\begin{aligned}
\lefteqn{P_k(z)=p_k^{(1)}({\bf t},{\bf s};z)\big\vert_{\LR}}
\\&= z^k\frac{Z({\bf t}-[z^{-1}],{\bf s})}{Z({\bf t},{\bf s})}\frac{\det(\Id -\mathbf{K}({\bf t}-[z^{-1}],{\bf s}))_k}{\sqrt{\det(\Id -\mathbf{K}({\bf t},{\bf s}))_k\det(\Id -\mathbf{K}({\bf t},{\bf s}))_{k+1}}}\Big\vert_{\LR}
\\
\lefteqn{\widehat P_k(w^{-1})=p^{(2)}_k({\bf t},{\bf s};w)\big\vert_{\LR} } \\&= w^{-k}\frac{Z({\bf t} ,{\bf s}+[w^{}])}{Z({\bf t},{\bf s})}\frac{\det(\Id -\mathbf{K}({\bf t},{\bf s}+[w^{}]))_k}{\sqrt{\det(\Id -\mathbf{K}({\bf t},{\bf s}))_k\det(\Id -\mathbf{K}({\bf t},{\bf s}))_{k+1}}}\Big\vert_{\LR}
\\
\lefteqn{\sum_{j=0}^k P_j(z)
 \widehat P_j(w^{-1})=
 \sum_{j=0}^k 
p_j^{(1)}({\bf t},{\bf s};z)
~p_j^{(2)}({\bf t},{\bf s};w^{-1})\big\vert_{\LR}}
 \\
&=
\left(\frac zw\right)^k 
\frac{Z({\bf t}-[z^{-1}],{\bf s}+[w^{}])}{Z({\bf t},{\bf s} )}
 \frac{ \det(\Id -\mathbf{K}({\bf t}-[z^{-1}],{\bf s}+[w^{}]))_k}{\det(\Id -\mathbf{K}({\bf t},{\bf s}))_{k+1}}\Big\vert_{\LR}
%
\end{aligned}
\label{31'}\ee
establishing formulas (\ref{31}), via (\ref{29}), (\ref{26}), (\ref{28}) and (\ref{30}), with the double integral representation (\ref{25}), (\ref{26}) and (\ref{25'}) for $K^{(i)}_{k,\ell}$. 
 Then using
\be 
\begin{aligned}
\frac{1}{v-u}\frac{1-\frac uz}{1-\frac vz}&=\frac{1}{v-u}-\frac{1}{v-z} \\
\frac{1}{v-u}{\frac { \left( u-z \right)  \left( v-w \right) }{ \left( v-z \right) 
 \left( u-w \right) }}&=\frac{1}{v-u}- {\frac {   \left(z -w 
 \right)  }{ \left( v-z \right)  \left( u-w \right) }},
%
\end{aligned}
\label{32}\ee
%
%
one notices that the kernels $K^{(1)}_{k,\ell}(z^{-1})$ and $K^{(2)}_{k,\ell}(w^{})$ are rank one perturbations of the kernels $K^{(1)}_{k,\ell}(0)$ and $K^{(2)}_{k,\ell}(0)$, whereas $K^{(3)}_{k,\ell}(z^{-1},w)$ is a rank two perturbation of $K^{(1)}_{k,\ell}(0)$, as given by the second equalities in (\ref{25}), (\ref{26}) and (\ref{25'}).
 %
The functions $g^{(i)}_{\ell}$, appearing here, have each two different expressions, one obtained from the other by $u\to u^{-1}$.
The functions $h_k^{(i)}$ are contour integrals, for which it is convenient, for later use\footnote{it enables us to expand $1/(v-z)=v^{-1}/(1-\tfrac zv)=\tfrac 1v \sum_1^\infty (\tfrac zv)^j$.}, to include $z$ in the contour, at the expense of introducing a residue; so, one has from the definitions of (\ref{25h}),
\be \begin{aligned}
h^{(1)}_{k}(z^{-1})
&=  \oint_{\Gamma_{0,a,z}}\frac{-dv}{2\pi i(v-z)}\frac{1
 }{(-v)^{k+1}\vp_a(2n;v)}+\frac{1
  }{(-z)^{k+1}\vp_a(2n,z)} 
\\
&=: \bar h_k^{(1)}(z^{-1})+\frac{1
}{(-z)^{k+1}\vp_a(2n,z)}\\
\\
h^{(2)}_{\ell}(w)
&=  \oint_{\Gamma_{0,w^{-1}}}\frac{-dv}{2\pi i(v-w^{-1})}\frac{\vp_a(2n;v^{-1})}{(-v)^{\ell+1}}+(-w)^{\ell+1}\vp_a(2n,w)
\\
&= \bar h_\ell^{(2)}(w)+(-w)^{\ell+1}\vp_a(2n,w)
.\label{35}\end{aligned}\ee
%
%
Then the $\bar h_k^{(i)}$ have the following form in terms of the $g^{(i)}_{\ell}$, upon using the expansion $1/(v-z)=v^{-1}/(1-\tfrac zv)=\tfrac 1v \sum_0^\infty (\tfrac zv)^j$ :
\be \begin{aligned}
\bar h_k^{(1)}(z^{-1})&=- \oint_{\Gamma_{0,a,z}}\frac{ dv}{2\pi i(v-z)}\frac{1
 }{(-v)^{k+1}\vp_a(2n;v)}
 =-\sum_{\alpha=0}^\infty (-z)^\alpha g^{(2)}_{k+\alpha}(n)
\\
\bar h_\ell^{(2)}(w)&=-\oint_{\Gamma_{0,w^{-1}}}\frac{dv}{2\pi i(-v)^{\ell+1}}\frac{\vp_a(2n,v^{-1})
}{v-w^{-1}}
=-\sum_{\alpha=0}^\infty (-w)^{-\alpha} g^{(1)}_{\ell+\alpha}(n),
\end{aligned}\label{36}\ee
establishing formulas (\ref{25h}). 
Upon using $1/(v-u)=v^{-1}\sum_0^{\infty} (u/v)^\alpha$, the kernels $K^{(1)}_{k,\ell}(0) $ and $K^{(2)}_{k,\ell}(0) $ can also be expressed in terms of the $g_\ell^{(i)}$'s, yielding:
\be \begin{aligned}
K^{(1)}_{k,\ell}(0)&=\sum_{\alpha=0} ^{\infty} g^{(1)}_{\ell+\alpha}g^{(2)}_{k+\alpha}=
K^{(2)}_{\ell,k}(0)
\end{aligned}\label{37}\ee
and therefore $K^{(1) \top}_{}(0)=K^{(2)}_{}(0)$. This ends the proof of Lemma \ref{L31}.
\end{proof}

\bigbreak

\noindent  It is convenient to introduce a new function $\Phi(z):=\Phi_{m/n}(z)$, depending on $m/n$ and $a$, namely 
 \be
 e^{\frac n2 \Phi(z)}:= 
(-z)^m \left ((1+az)\Bigl(1-\frac az\Bigr)\right)^{n/2}.
 \label{44}
 \ee
From the decomposition (\ref{25}) and (\ref{26}) of the kernels $K^{(i)}_{k,\ell}$, it follows from (\ref{K2}) and (\ref{29}) that, in particular, 
\be \begin{aligned}
H^{(1)}_{2m+1}(z^{-1})&= H^{ }_{2m+1}(0)\left(1-R^{(1)}(z^{-1})\right)\\
H^{(2)}_{2m+1}(z) &= H^{ }_{2m+1}(0)\left(1-R^{(2)} (z)\right),
\end{aligned}\label{38}\ee
%
where, by (\ref{35}),
\be
\begin{aligned}
R^{(1)}_{ }(z^{-1})&:= \sum_{\ell\geq {2m+1}}Q^{(1)}_{ \ell}h^{(1)}_{\ell}(z^{-1})  
\\
&= \sum_{\ell\geq {2m+1}}Q^{(1)}_{ \ell}\bar h^{(1)}_{\ell}(z^{-1})+
 \frac{\sum^{\iy}_1\frac{Q^{(1)}_{2m+\ell }}{(-z)^{\ell}}}{(-z)^{2m+1}\vp_a(2n;z) 
}\\
&=:S ^{(1)}(z^{-1})+  T ^{(1)}(z^{-1})  
\frac{e^{-n\Phi(z)}}{a-z}
   \\
\\
R^{(2)}_{ }(z)&:= \sum_{\ell\geq {2m+1}}Q^{(2)}_{ \ell}h^{(2)}_{\ell}(z) 
\\
&= \sum_{\ell\geq {2m+1}}Q^{(2)}_{ \ell}\bar h^{(2)}_{\ell}(z)+(-z)^{2m+1} 
\vp_a(2n;z)
 \sum^{\iy}_1 Q^{(2)}_{2m+\ell }(-z)^{\ell}\\
&=:S ^{(2)}(z^{})+T ^{(2)}(z^{})\frac{a-z}{e^{-n\Phi(z)}}
,
\end{aligned}
 \label{39}\ee  
 with $Q^{(i)}_{ \ell}$ a function of $\ell$, with support $[2m+1,\infty)$, defined by
\be
Q^{(i)}_{ \ell}=[(\Id -\raisebox{1mm}{$\chi$}{}_{2m+1} K^{(i)}(0)^{\top}\raisebox{1mm}{$\chi$}{}_{2m+1} )^{-1}  \raisebox{1mm}{$\chi$}{}_{2m+1} g^{(i)}](\ell), 
\label{Qell}\ee
and
\be
\begin{aligned} 
S ^{(1)}(z^{-1})&:=\sum_{\ell\geq {2m+1}}Q^{(1)}_{ \ell}\bar h^{(1)}_{\ell}(z^{-1})  ,~~ S ^{(2)}(z^{}):= \sum_{\ell\geq {2m+1}}Q^{(2)}_{ \ell}\bar h^{(2)}_{\ell}(z)
\\T ^{(1)}(z^{-1})&:=\sum_{\ell\geq 1} \frac {Q^{(1)}_{2m+\ell}}{(-z)^\ell}
,~~~ T ^{(2)}(z^{}) :=
 \sum_{\ell\geq 1}   {Q^{(2)}_{2m+\ell }}{(-z)^\ell}
,\end{aligned}
 \label{40}\ee 
and with (remembering from (\ref{36}))
\be \begin{aligned}
\bar h_k^{(1)}(z^{-1})&=-\sum_{\alpha=0}^\infty (-z)^\alpha g^{(2)}_{k+\alpha}(n)
,~~~
\bar h_k^{(2)}(z) =-\sum_{\alpha=0}^\infty (-z)^{-\alpha} g^{(1)}_{k+\alpha}(n).
\end{aligned}\label{36'}\ee

So we have the following expressions for the orthonormal polynomials $P_{2m+1}$ and $\hat P_{2m+1}$ on the circle, as in (\ref{31}) : 
\be\begin{aligned}
P_{2m+1}(z)&=  \sqrt{\frac{H_{2m+1}(0)}{H_{2m+2}(0)}}z^{2m+1}\left(1-\frac{a}{z}\right)^{n+1}\left(1-R^{(1)}(z^{-1})\right) 
\\
\hat P_{2m+1}(z^{-1})&=  \sqrt{\frac{H_{2m+1}(0)}{H_{{2m+2}}(0)}}z^{-2m-1} (1+az)^{-n}\left(1-R ^{(2)}(z)\right).
\end{aligned}\label{41}\ee 
In particular, the fact that the norm equals $1$ implies an identity, which will be used later; namely by (\ref{12'}) and (\ref{13}), find: 
\be \begin{aligned}
1&=\ll P_{2m+1} ,\hat P_{2m+1}\gg\\
&= \frac{H_{2m+1}(0)}{H_{{2m+2}}(0)}\oint_{\Gamma_{0,a}}\frac{dz}{2\pi i z}\frac{(1+az)^n}{\left(1-\frac{a}{z}\right)^{n+1}}\left(1-\frac{a}{z}\right)^{n+1}(1+az)^{-n}\\
&  \hspace*{3cm}\left(1-R ^{(1)}(z^{-1})\right)\left(1-R ^{(2)}(z)\right)\\
&= \frac{H_{2m+1}(0)}{H_{{2m+2}}(0)}\oint_{\Gamma_{0,a}}\frac{dz}{2\pi i z}\left(1-R ^{(1)}(z^{-1})\right)\left(1-R ^{(2)}(z)\right)
.\end{aligned}\label{42}\ee

 \noindent We end the section with some {\bf Fredholm determinant identities:}
    
\begin{lemma}\label{LA1} (\cite{Joh10}) 
Given a trace-class operator $A$ and a rank one operator $a\otimes b$, the following holds for an arbitrary constant $c$:%
\begin{multline}
\det (I-A+c~a\otimes b)=(1-c)\det(I-A)+c\det(I-A+a\otimes b).
 \end{multline}

\end{lemma}

%
%
%


\begin{lemma} \label{LA2} (\cite{Joh10}) Given two continuous functions $F(z)$ and $G(w)$, vectors $a^z,~b^w$ depending continuously on $z$ and $w$, and a trace class kernel $K $, one has that
 \begin{multline}
\int_{\gamma_{r_1}}\frac{dz}{z}\int_{\gamma_{r_2}}\frac{dw}{w}F(z)G(w)\det (I-K +a^z\otimes b^w)\\
= 
 \det\left( I-K +\left( \int_{\gamma_{r_1}}F(z)a^z\frac{dz}{z}\right)\otimes\left(\int_{\gamma_{r_2}}G(w)b^w\frac{dw}w \right)  \right)  \\
 +\left[ \left(\int_{\gamma_{r_1}}F(z)\frac{dz}{z}\right)    \left(  \int_{\gamma_{r_2}} G(w)\frac{dw}{w}\right)-1\right]\det(I-K )
 \label{A1}\end{multline}

\end{lemma}

%


\section{Representations of the kernel  $\widetilde\BK^{\rm ext}_{n,m}$ for the domino outlier paths
 }
 
 The purpose of this section is to deduce Theorem \ref{main1'} from Proposition \ref{prop}, using the Fredholm determinant representation, {\em either} of the Darboux sum, {\em or} of the bi-orthogonal polynomials themselves, as in (\ref{31}). They will lead to the two very different expressions for the kernel, announced in Theorem \ref{main1'}.

 The following formula, based on contour integration, will be useful in the sequel:
\be \begin{aligned}
  \frac{1}{(2\pi i)^2}\oint_{\Gamma_{0,a}}dz&\oint_{\Gamma_{0,a,z}}\hspace{-0.5em} dw \frac{F(z,w)}{w-z} \\
&=
\frac{1}{(2\pi i)^2}\oint_{\Gamma_{0,a}}dw\oint_{\Gamma_{0,a,w}}\hspace{-0.5em}dz \frac{F(z,w)}{w-z}+\frac{1}{2\pi i}\oint_{\Gamma_{0,a}} F(z,z).
\end{aligned}\label{45}\ee

 \subsection{Representation 
 as a perturbation of the Krawtchouk kernel }
 
 In this section we establish formula (\ref{K1}) of Theorem \ref{main1'}. To do so, we take formula (\ref{14}) for the kernel $\BK_{n,m}$ and the third formula of Darboux type in (\ref{31}) as starting points. 
 Recall some of the definitions given in (\ref{Idef}); the definitions of $a_{x,s}$ and $b_{y,s}$ will turn out to be the same as the one in (\ref{Idef}). Here we have some additional ones:
 \be
\begin{aligned}
 F_{x,s}(z)& = \frac{ (-1)^x  }{2\pi i } z^{x+m+1}\vp_a(n;z)\left(\frac{1+az}{1-\frac az}\right)^{s-\tfrac n 2},~~~F_x(z)=F_{x,  n/2}(z),
\\ 
 G_{y,r}(w) & =
 \frac{ (-1)^y }{2\pi i  } \frac{1}{w^{y+m+1}\vp_a(n;w)}\left(\frac{1+aw}{1-\frac aw}\right)^{\tfrac  n2-r },~~~G_y(w)=G_{y,n/2}(w)
,\end{aligned} \label{FG}\ee
\be 
\begin{aligned}a_{x,s}(k) & =\oint_{\gamma_{r_1}}\frac{dz}{z} F_{x,s}(z) h^{(1)}_k(z^{-1}),~~~~a_x(k) =a_{x,n/2}(k)
 ,~~~
 \\
  b _{y,r}(\ell) &=\oint_{\gamma_{r_2}} \frac{dw}{w} G _{y,r}(w) h^{(2)}_\ell(w),~~~~
  b _{y}(\ell)=b _{y,n/2}(\ell).
 \end{aligned} \label{ab}\ee
We now extend the functions $g_\ell^{(i)}(k)$ of (\ref{34}) as follows:
\begin{equation}\begin{aligned}
g_{\ell}^{(1)}(k;s)&=-\int_{\Gamma_{0,a}}\frac{dz}{2\pi i}(-z)^{\ell}\vp_a(2k;z)\left(\frac{1+az}{1-\frac{a}{z}}\right)^{s} ,~~~g_{\ell}^{(1)}(k;0)=g_{\ell}^{(1)}(k ),
\\
g_{\ell}^{(2)}(k;r)&=-\int_{\Gamma_{0,a}}\frac{dz}{2\pi i}\frac{(-z)^{-\ell -2}}{\vp_a(2k,z)}\left(\frac{1+az}{1-\frac{a}{z}}\right)^{r},~~~g_{\ell}^{(2)}(k;0)=g_{\ell}^{(2)}(k ),
\end{aligned}
\label{E4}\end{equation}
and $\tilde g^{(1)}_{\ell}(k,j)$, $\tilde g^{(2)}_{\ell}(k,j)$ are a small variation \footnote{\label{foot4'}Compared to $  g^{(2)}_k$, the expression $\widetilde g^{(2)}_k$ has an extra term $z-a$ in the numerator. Also the expression $\widetilde g^{(1)}_k$ has an extra term $w-a$ in the denominator; namely
 \be\begin{aligned}
\widetilde g^{(1)}_{\ell}\left(\tfrac n2\right)&= -\oint_{\Gamma_{0}} \frac{dz}{2\pi i}(-z)^\ell \frac{\vp_a( n;z)}{z-a}
~~\mbox{and}~~
\widetilde g^{(2)}_{\ell}\left(\tfrac n2\right) = 
  -\oint_{\Gamma_{0,a}} \frac{dz}{2\pi i}(-z)^{-\ell-2}\frac{ z-a }{
\vp_a( n;z)
 }.
\end{aligned}\label{g}\ee
} of $  g^{(1)}_{\ell}(k,j)$, $  g^{(2)}_{\ell}(k,j)$.

\begin{proposition} \label{Prop4.2} The dual kernel has the expression (\ref{dual}). It is a perturbation of the one-Aztec diamond kernel (\ref{OneAztec})  and involves the resolvent of the kernel (\ref{25})\footnote{$\la f(k), g(k)\ra_{\ell^2 (2m+1,\ldots)}=\sum_{2m+1}^{\infty} f(k)g(k).
$}:
\be
  \begin{aligned}
\lefteqn{(-1)^{x-y}   \widetilde\BK^{\rm ext}_{n,m}(2r,x;2s,y) }\\ 
 &=-\Id_{s<r} (-1)^{x-y}\psi_{2(s-r)}(x,y)  +
  S(2r,x;2s,y) 
  \\
 &~~~~+\left\la(\Id-K_{2m+1}^{(1)}(0))^{-1}a_{-y,s}(k),b_{-x,r}(k) \right\ra
 _{\ell^2 (2m+1,\ldots)}    ,\end{aligned}  %
 \label{dual}\ee
 with \footnote{\label{foot4}Note that the definition of $S(2r,x;2s,y)$ is compatible with the one in (\ref{Idef}). Also it appears as a part of the kernel (\ref{43a}).}
  \be
\begin{aligned}
  \lefteqn{S(2r,x;2s,y)}\\
  &:={\oint_{ \gamma_{r_3}}\frac{ dz}{z} 
  \oint_{ \gamma_{r_2}} \frac{dw}w 
  \frac{F_{-y,s}(z)G _{-x,r}(w)w}{w-z}  }
\\
 &  =\frac{(-1)^{x-y}} {(2\pi i)^2}\!\oint_{\gamma_{r_3}}dz
 \oint_{\gamma_{r_2}}\frac{dw}{w\!-\!z}\frac{ z ^{m-y }}{ w ^{m-x+1 }} \frac{\vp_a(n;z)}{\vp_a(n;w)}
   \left(\tfrac{1+az}{1-\frac{a}{z}}\right)^{s-\tfrac n2}\!\!\left(\tfrac{1+aw}{1-\frac{a}{w}}\right)^{ \tfrac n 2 -r}
\\
&=\sum^{\iy}_{b=0}g^{(1)}_{ -y+m+b }\left(\tfrac{n}{2},s- \tfrac{n}{2}\right)g^{(2)}_{ -x+m+b } \left(\tfrac{n}{2},  \tfrac{n}{2}-r \right)\\
\end{aligned}
\label{FG1}\ee
and with
$$\begin{aligned}
a_{x,s}(k)&=\frac{(-1)^{k+x}}{(2\pi i)^2}\oint_{\gamma_{s_1}}   dv \oint_{\gamma_{r_1}}\frac{ dz}{v-z}\frac{\vp_a(n;z)z^{m+x}}{\vp_a(2n;v)v^{k+1}}   \left(\tfrac{1+az}{1-\frac{a}{z}}\right)^{s-\tfrac n2}
\\
&=(-1)^m\left(\sum_{b\geq 0}g^{(1)}_{ x+m+b }\left(\tfrac{n}{2},s\!-\! \tfrac{n}{2}\right)g^{(2)}_{k+b}(n)+\tilde g^{(2)}_{ -x-m+k }\left(\tfrac{n}{2},s\!-\! \tfrac{n}{2}\right)\right)%
 \end{aligned}$$
%
\\
\be\begin{aligned}
b_{y,r}(\ell)&= \frac{(-1)^{\ell +y+1}}{(2\pi i)^2}\oint_{\gamma_{r_2}} {dw}    \oint_{\gamma_{s_2}}\frac {du} {u-w}\frac{\vp_a(2n;u)u^{\ell}}{\vp_a(n;w)w^{y+m+1}} \left(\tfrac{1+aw}{1-\frac{a}{w}}\right)^{ \tfrac n2 -r}  
\\
&=(-1)^m\left(\sum_{b\geq 0}g^{(1)}_{\ell +b}(n )g^{(2)}_{ m+y+b }\left(\tfrac{n}{2},\tfrac{n}{2}-r  \right) +\tilde g^{(1)}_{-y-m+\ell }
 \left(\tfrac{n}{2},\tfrac{n}{2}-r\right)\right).\\
%
 \end{aligned}\label{ab0}\ee

\end{proposition}




\begin{proof}  Throughout the proof the following radii, in between $a$ and $a^{-1}$, will be used:
\be 
a<r_3<r_2<s_2<s_1<r_1<s_3<a^{-1}.\label{rad}\ee
  The contours $\Gamma_{0,a}$ and $\Gamma_{0,a,z}$ in the kernel (\ref{14}) will be deformed to circles $\gamma_{r_2}$ and $\gamma_{r_1}$ respectively, with $w\leftrightarrow z$ interchanged,
%
\be 
\begin{aligned}
\lefteqn{ (-1)^{x-y} \BK_{n,m}(-y,-x)} 
\\
&=  \frac{ (-1)^{x-y}}{(2\pi i)^2}\oint_{\gamma_{r_2}}\frac{dw}{w}\rho^R_a(w)\oint_{\gamma_{r_1}}\frac{dz}{z}\rho^L_a(z)\frac{z^{x-m}}{w^{y-m}}\sum_{k=0}^{2m}P_k(z)\widehat P_k(w^{-1}).
\end{aligned}
\label{Kyx}\ee
%
%
Using (\ref{31}), (\ref{29}) and (\ref{25'}) combined for the Darboux sum, the formulas (\ref{rho}) for $\rho^R_a(w)$ and $\rho^L_a(z)$, and the formula (\ref{psi}) for $\vp_a(n;z)$, and Lemma \ref{LA1} for the rank $2$ perturbation, one finds for the integrand of kernel (\ref{Kyx}),
$$\begin{aligned}
 \lefteqn{\frac{ (-1)^{x-y}}{(2\pi i)^2}\frac{\rho^R_a(w)}{w}\frac{\rho^L_a(z)}{z}
 \frac{z^{x-m}}{w^{y-m}}\sum_{k=0}^{2m} P_k(z)\widehat P_k(w^{-1}) }
 \\&=\frac{ (-1)^{x-y}}{(2\pi i)^2}\frac{\vp_a(n;z)}{\vp_a(n;w)}\frac{z^{x+m}}{w^{y+m+1}}\frac{1}{z-w}
  \frac{\det \left(\Id -K^{(3)}_{k  ,\ell }(z^{-1},w)
 \right)_{\ell^2 (2m ,\ldots)}}{
 \det \left(\Id -K^{(1)}_{k ,\ell }(0)
 \right)_{\ell^2 (2m+1 ,\ldots)}} \\
&= 
\frac{F_x(z)G_y(w)}{z}
\frac{\det \left(\Id -K^{(1)}_{k ,\ell }(0)+(z-w) h^{(1)}_{k }(z^{-1}) \bigl(-\tfrac{1}{w}h^{(2)}_{\ell }(w))
 \right)_{\ell^2 (2m+1,\ldots)}}{
 (z-w) \det \left(\Id -K^{(1)}_{k,\ell}(0)
 \right)_{\ell^2 (2m+1,\ldots)}}\\
 &=
\frac{F_x(z)G_y(w)}{z} 
\\
&~~~\times \left(\frac{1}{z-w}-1+ \frac{\det \left(\Id -K^{(1)}_{k ,\ell }(0)+  h^{(1)}_{k }(z^{-1}) \bigl( -\tfrac{ 1}{w}h^{(2)}_{\ell }(w))
 \right)_{\ell^2 (2m+1,\ldots)}}
   {\det \left(\Id -K^{(1)}_{k,\ell}(0)
 \right)_{\ell^2 (2m+1,\ldots)}}\right).
 \end{aligned}$$
 Then one performs the $w$-contour integration over $\gamma_{r_2}$ and the $z$-integration over $\gamma_{r_1}$; from (\ref{45}), one has that for any $a<r_3<r_2<r_1,$
 $$\begin{aligned}
\lefteqn{  
 \int_{ \gamma_{r_2}} {dw}\int_{ \gamma_{r_1}}\frac{ dz}{z}   \frac{F_x(z)G _y(w) }{z-w}
 }\\
&=
  \int_{ \gamma_{r_3}}\frac{ dz}{z} 
  \int_{ \gamma_{r_2}}  {dw}  
  \frac{F_x(z)G _y(w) }{z-w}
  +
   {2\pi i} \oint_{ \gamma_{r_2}}\frac{dw}{w } F_x(w)G _y(w) ,
 \end{aligned}  %
 $$
 with
 $$
 {2\pi i} \oint_{ \gamma_{r_2}}\frac{dw}{w } F_x(w)G _y(w) =
  \frac { (-1)^{x-y}}{2\pi i} \oint_{ \gamma_{r_2}}
 \frac{dw}{w } 
   w^{x-y}=  { (-1)^{x-y}} \delta_{x,y},
 $$
 which is used in $\stackrel{** }{=} $ in (\ref{46a}) below. Remembering the definitions (\ref{ab}) of $a_x(k)$ and $b_y(\ell)$, this yields (using the radii (\ref{rad})), thanks to Lemma \ref{LA2} used in equality $\stackrel{* }{=} $ in (\ref{46a}), 
  \be  
\begin{aligned}
\lefteqn{ (-1)^{x-y} \BK_{n,m}(-y,-x)} 
\\  
%
 &=   \int_{\gamma_{r_2}} \frac{dw}w 
 \int_{\gamma_{r_1}}\frac{ dz}{z} \frac{F_x(z)G _y(w)w} {z-w}
 -  \int_{\gamma_{r_2}} \frac{dw}w G _y(w)w
 \int_{\gamma_{r_1}}\frac{ dz}{z} F_x(z)
 \\
 &+   \int_{\gamma_{r_2}} \frac{dw}w 
 \int_{\gamma_{r_1}}\frac{ dz}{z} F_x(z)G _y(w)w
  \frac{\det \left(\Id -K^{(1)}_{k ,\ell }(0)+ h^{(1)}_{k }(z^{-1}) (-\tfrac{1}{w}h^{(2)}_{\ell }(w) )
 \right)_{\ell^2 (2m+1,\ldots)}}{\det \left(\Id -K^{(1)}(0)
 \right)_{\ell^2 (2m+1,\ldots)}} 
 \\
 &\stackrel{* }{=}  \int_{ \gamma_{r_2}} \frac{dw}w 
 \int_{\gamma_{r_1}}\frac{ dz}{z} \frac{F_x(z)G _y(w)w}{z-w}
 -1 +\frac{\det \left(\Id -K^{(1)}(0) -a_x\otimes b_y
 \right)_{\ell^2 (2m+1,\ldots)}}
 {\det \left(\Id -K^{(1)}(0)
 \right)_{\ell^2 (2m+1,\ldots)}} 
 \\
 &\stackrel{** }=\! \int_{ \gamma_{r_3}}\!\!\frac{ dz}{z} 
  \int_{ \gamma_{r_2}} \!\!\frac{dw}w 
  \frac{F_x(z)G _y(w)w}{z-w}
  + { (-1)^{x-y}} \delta_{x,y}  
 \\
 &~~~~ -1 +\frac{\det \left(\Id -K^{(1)}(0) - a_x 
  \otimes   b _y
 \right)_{\ell^2 (2m+1,\ldots)}}
 {\det \left(\Id -K^{(1)}(0)
 \right)_{\ell^2 (2m+1,\ldots)}} .
  \end{aligned} \label{46a}\ee
   Then the dual kernel $\widetilde\BK_{n,m}(-y,-x)$ is given by
 \be
  \begin{aligned}
 (-1)^{x-y}   \widetilde\BK_{n,m}(-y,-x)  &= (-1)^{x-y}(\Id -  \BK_{n,m}(-y,-x))
 \\
  &=
    \int_{ \gamma_{r_3}}\frac{ dz}{z} 
  \int_{ \gamma_{r_2}} \frac{dw}w 
  \frac{F_x(z)G _y(w)w}{w-z}  
  \\
 &~~~~+1 -\frac{\det \left(\Id -K^{(1)}(0) -a_x
  \otimes  b_y
 \right)_{\ell^2 (2m+1,\ldots)}}
 {\det \left(\Id -K^{(1)}(0)
 \right)_{\ell^2 (2m+1,\ldots)}} 
  ,\end{aligned}  %
 \label{dual'}\ee 
 thus establishing formula (\ref{dual}) for $r=s=n/2$. 
 
 The second formula (\ref{FG1}) for $r=s=n/2$ follows immediately from the definitions (\ref{FG}) of $F_x(z)$ and $G_y(w)$, while the third formula is proved exactly as in (\ref{36}). It remains to prove the formulas (\ref{ab0}) for $a_x(k)$ and $b_y(\ell)$. Indeed,
 formula (\ref{45}) permits us to interchange the integration variables in the expression below, so that $|z/v|<1$, enabling us to expand $\frac 1{   1-  z/v  }=  \sum_{b\geq 0}(\tfrac zv)^b$. Thus, using the radii (\ref{rad}), and using the explicit expressions (\ref{FG}) and (\ref{25h}) for $F_x$ and $h^{(1)}_k$, one finds for $a_x(k)$ as in (\ref{ab0}), but for $s=n/2$, 
 $$\begin{aligned}
 a_x(k)
 &=\frac{ (-1)^{x+k}}{(2\pi i)^2}\oint_{\gamma_{s_1}} \frac {dv}{v}
 \oint_{\gamma_{r_1}} \frac {dz}{z}
 \frac{\vp_a(n;z) z^{x+m+1}}{\vp_a(2n;v)v^{k}}\frac{1}{v-z}\\
&=\frac{ (-1)^{x+k}}{(2\pi i)^2}\oint_{\gamma_{s_1}}  {dz} \oint_{\gamma_{s_3}}  {dv} 
 \frac{\vp_a(n;z) z^{x+m }}{\vp_a(2n;v)v^{k+2}}\frac{1}{1-\frac zv}
 \\
 &~~~~  - { (-1)^{x+k}} \oint_{\gamma_{r_1}}  \frac{dz}{2\pi i   } z^{x+m-k-2 }\frac{ z-a  }{\vp_a(n;z) }\\
 &=
    (-1)^{ m }\left(\sum_{b\geq 0} g^{(1)}_{x+m+b} (\tfrac n2) 
 g^{(2)}_{k+b } (n) +\widetilde g^{(2)}_{-x-m+k}(\tfrac n2)\right)
\end{aligned}$$
and, similarly for $b _y(\ell)$, for $r=n/2$,
\be\begin{aligned} 
   b _y(\ell)
  &=-\frac{ (-1)^{y+\ell }}{(2\pi i)^2}\oint_{\gamma_{r_2}} \frac {dw}{w}\oint_{\gamma_{s_2}} \frac {du}{u}
 \frac{\vp_a(2n;u) u^{  \ell+1}}{\vp_a( n;w)w^{y+m}}\frac{1}{u-w}
 \\
   &= \frac{ (-1)^{y+\ell }}{(2\pi i)^2}\oint_{\gamma_{r_3}}   {du} \oint_{\gamma_{r_2}}  {dw}
 \frac{\vp_a(2n;u) u^{  \ell }}{\vp_a( n;w)w^{y+m+2}}\frac{1}{1-\tfrac uw}
\\
 &~~~~- { (-1)^{y+\ell }}{ }\oint_{\gamma_{r_2}}  \frac{dw}{2\pi i  } w^{-y -m+\ell}\frac{\vp_a( n;w)}{w-a}
 \\
 &= (-1)^m \left( \sum_{b\geq 0} g^{(1)}_{\ell+b}(n) g^{(2)}_{y+m+b}(\tfrac n2)+\widetilde g^{(1)}_{-y-m+\ell }(\tfrac n2)\right). \end{aligned}\label{ab1}\ee
  This proves Proposition \ref{Prop4.2} for $r=s=n/2$.
%
 
 %

  We now proceed to compute the extended kernel for general $r$ and $s$, using the recipe (\ref{E1}) in Proposition \ref{prop} applied to the non-extended kernel (\ref{dual'}). Thus,
$$\begin{aligned}
\lefteqn{ (-1)^{x-y}\tilde {\mathbbm K} ^{{\rm ext}}_{n,m}(2r,x;2s,y)}\\
&= -\Id_{s<r} (-1)^{x-y}\psi_{2 s-2r}(x,y)    
\\
&\quad +(-1)^{x-y}\psi_{n-2r}(x,\cdot)\ast  \tilde {\mathbbm K} ^{{\rm ext}}_{n,m}(n,\cdot ;n,\tc) \ast\psi_{ 2s-n }(\tc,y) 
\\
 & =-\Id_{s<r} (-1)^{x-y}\psi_{2 s-2r}(x,y) + (-1)^{x-y} \int_{\gamma_{r_3}}\frac{dz}{z} \int_{\gamma_{r_2}}\frac{dw}{w-z}
 \\
 &\quad   { \bigl(\psi_{n-2r}(x,u)\ast G_{-u}(w)(-1)^u \bigr) \bigl( F_{ -v   }(z)(-1)^v\ast \psi _{2s-n}(v,y)\bigr)} 
 \\
 &\quad +(-1)^{x-y}\left\la   (\Id-K_{2m+1}^{(1)}(0))^{-1} a_{ -v   }(k) (-1)^v\ast_{_{\!\!v}} \psi _{2s-n}(v,y), \right.
 \\
 &~~~\left. \psi_{n-2r}(x,u)\ast_{_{\!\!u}} b_{-u}(\ell)(-1)^u\right\ra_{\ell^2(2m+1,\ldots )}
 \end{aligned}
 $$
 \be\begin{aligned}
  &=-\Id_{s<r} \psi_{2s-2r)}(x,y)(-1)^{x-y} +S(2r,x;2s,y)\\&~~+\left\la   (\Id-K_{2m+1}^{(1)}(0))^{-1} a_{ -y,s   }(k)  ,b_{-x,r}(\ell) \right\ra_{\ell^2(2m+1,\ldots )} 
\end{aligned}
\label{E2}\ee
We now show the last identity in (\ref{E2}). The useful tool here will be
the following formula for the $\ast$-product (with respect to $x$) holds:\be z^{\pm x}\ast_{_{\!\!x}} \oint \frac{dw}{2\pi iw} w^{\mp x}F(w)=\oint \frac {dw}{2\pi i w}\sum_{u\in \BZ} \left(\frac zw\right)^u F(w)=F(z),\label{astpower}\ee
from which it follows that (see (\ref{Idef}))
\be
z^{-v} \ast_{_{\!\!v}}\psi_{2s-n}(v,y)=z^{-y}
\left(\tfrac{1+az}{1-\frac{a}{z}}\right)^{s-\tfrac n 2},~~
\psi_{n-2r}(x,u)\ast_{_{\!\!u}}w^u=\left(\tfrac{1+aw}{1-\frac{a}{w}}\right)^{\frac n 2-r}  w^x .
\label{astpower'}\ee
We thus need to compute  the following $\ast$-product, with regard to $v$, also using formula (\ref{astpower'}), 
\be
\begin{aligned}
(-1)^yF_{-y,s}(z) &={F_{-v}(z)(-1)^v\ast_{_{\!\!v}}\psi_{2s-n}(v,y)}   
  \\
  &
=\frac{\vp_a(n;z)}{2\pi i}z^{-v+m+1}\ast_{_{\!\!v}}\psi_{2s-n}(v,y)
\\
&=\frac{\vp_a(n;z)}{2\pi i}z^{-y+m +1} \left(\tfrac{1+az}{1-\frac{a}{z}}\right)^{s-n/2} 
 =(-1)^yF_{-y}(z)\left(\tfrac{1+az}{1-\frac{a}{z}}\right)^{s-n/2},
\end{aligned}
\label{Fext}\ee
from which one deduces the following $\ast$-product, also in $v$, remembering the definitions (\ref{25h}) of the $h_k^{(i)}$, and using the definitions (\ref{FG}) and (\ref{ab}), %
\be \begin{aligned}
(-1)^y a_{-y,s}(k)&:=a_{-v}(k)(-1)^v\ast_{_{\!\!v}}\psi_{2s-n}(v,y)  
\\
&=\oint_{\gamma_{r_1}}\frac{dz}{z}  h^{(1)}_k(z^{-1})F_{-v}(z)(-1)^v\ast\psi_{2s-n}(v,y)  
\\
&=\oint_{\gamma_{r_1}}\frac{dz}{z}  F_{-y,s}(z)(-1)^y  h_{k}^{(1)}(z^{-1})
\end{aligned}\label{aext}\ee
Similarly, one checks, using (\ref{FG}) and formula (\ref{astpower'}), 
\be
\begin{aligned}
(-1)^x G_{-x,r}(w)&:={\psi_{n-2r}(x,u)\ast_{_{\!\!u}} G_{-u}(w)(-1)^u}  
\\
&= \psi_{n-2r}(x,u)\ast_{_{\!\!u}}\frac{1}{2\pi i}\frac{w^{u-m-1}}{\vp_a(n;w)}  
\\
&=\frac{1}{2\pi i}\frac{w^{x-m-1}}{\vp_a(n,w)} \left(\frac{1+aw}{1-\frac{a}{w}}\right)^{\tfrac n2 -r}
\!=(-1)^xG_{-x}(w)\!\left(\frac{1\!+\!aw}{1-\frac{a}{w}}\right)^{\tfrac n2 -r} ,
\end{aligned}
\label{Gext}\ee
from which one deduces, using (\ref{FG}) and  (\ref{ab}) 
\be  \begin{aligned}
 {(-1)^xb_{-x,r}(\ell)} 
&:=\psi_{n-2r}(x,u)\ast_{_{\!\!u}} b_{-u}(\ell)(-1)^u  
\\
&=\oint_{\gamma_{r_2}}\frac{dw}{w}  
 \psi_{n-2r}(x,u)\ast G_{-u}(w)(-1)^u h^{(2)}_k(w) 
\\
&=\oint_{\gamma_{r_2}}\frac{dw}{w}  G_{-x,r}(w)(-1)^x  h_{\ell}^{(2)}(w)  
.\end{aligned}
\label{bext}\ee
The last identity in (\ref{E2}) now follows from (\ref{Fext}), (\ref{bext}) and (\ref{FG1}). The formulas (\ref{FG1}) and (\ref{ab0}), involving the $g_k^{(i)}(\tfrac n2,.)$ are obtained exactly as in the computation for $r=s=n/2$. 

The identification of the first line of the equality  (\ref{dual}) with the one-Aztec diamond kernel (\ref{OneAztec}) is immediate, after setting $z\mapsto -z^{-1}$ in the integral $\psi_{2(s-r)}$. This concludes the proof of Proposition \ref{Prop4.2}.
\end{proof}


\subsection{Representation of $\widetilde\BK^{\rm ext}_{n,m}$ as a double integral}

Here we show formula (\ref{K2}) for $\widetilde\BK^{\rm ext}_{n,m}$ in Theorem \ref{main1'} and also, along the way, an alternative representation (\ref{43a}), which will be useful in taking saddle point limits. The starting point here is again formula (\ref{14}) for the kernel $\BK_{n,m}$ as in previous section; but instead we will now use the Christoffel-Darboux formula and the two first formulas of (\ref{31}). The methods and the notation in this section are very close to those used in \cite{AFvM12}. 
Remembering the definition (\ref{44}) of $\Phi(z)$, define the functions
\be\begin{aligned}
E_1(z,w)&:=e^{\frac n2 (\Phi(z)-\Phi(w))}
\left(1-S_{ }^{(1)}(z^{-1})\right)\left(1-S_{ }^{(2)}(w)\right)
\\
E_2(z,w)&:=
e^{\frac n2 (\Phi(z)+\Phi(w))} (w-a)\left(1-S_{ }^{(1)}(z^{-1})\right) T_{ }^{(2)}(w) 
\\
E_3(z,w)&:=e^{-\frac n2 (\Phi(z)+\Phi(w))} \frac{1}{z-a}T_{ }^{(1)}(z^{-1})   \left(1-S_{ }^{(2)}(w)\right) 
\\
E_4(z,w)&:=-e^{ \frac n2 (\Phi(z)-\Phi(w))}T_{ }^{(1)}(w^{-1})  T_{ }^{(2)}(z)
\end{aligned}\label{Ei}\ee
and
\be
\begin{aligned}
C(u;x) :=&~ C_1(u;x)+2C_2(u;x)\\
=& ~~  \frac{\dt_{x\neq 0}}{2\pi i}
 \oint_{\Gamma_{0,a}}\frac{dz}{(-z)^{x +1}}
   (1-R^{(1)}_{ }(z^{-1}))(1-R^{(2)}_{}(z))
   \left(\tfrac{1+az}{1-\frac az}\right)^{u}
\\
&~~-\frac 2{2\pi i} \oint _{\Gamma_{0,a}} \frac{dz}{ (-z)^{ x+1}}T^{(1)}(z^{-1})T^{(2)}(z)
\left(\tfrac{1+az}{1-\frac az}\right)^{u}.\label{55}
\end{aligned}
\ee
Remembering the form of the kernel $\BK_{n,m}(-x,-y)$, obtained in (\ref{14}), one has the following Proposition:

\begin{proposition}\label{prop4.1}
The dual kernel $\widetilde\BK^{\rm ext}_{n,m}(2r,-x;2s,-y)$ takes on two different forms
\newline (i) the form (\ref{K2}), as stated in Theorem \ref{main1'}. 
\newline (ii) or another form, useful for taking saddle point limits (see section \ref{sect7.2}):
\be
\begin{aligned}
\lefteqn{(-1)^{x-y} \tfrac{H_{2m+2}(0)}{H_{2m+1}(0)}
\widetilde\BK^{\rm ext}_{n,m}(2r,-x;2s,-y) 
= C(s-r;x-y) }\\
& ~~-\Id_{s<r} (-1)^{x-y}\tfrac{H_{2m+2}(0)}{H_{2m+1}(0)}\psi_{2(s-r)}(-x,-y)   + 
 \oint_{\Gamma_{0,a}}\frac{dz}{(2\pi i)^2}\oint_{\Gamma_{0,a,z}}dw 
 \frac{\sum_1^4 E_i(z,w)}{z-w}
\\& 
 ~~~~~~~~\times \left(\frac{(-w)^{y-1}}{(-z)^{x}}\frac{\left(\tfrac{1+az}{1-\frac az}\right)^{\tfrac n2-r}}{\left(\tfrac{1+aw}{1-\frac aw}\right)^{ \tfrac n2 -s}}+
\frac{(-z)^y}{(-w)^{x+1} }
\frac{\left(\tfrac{1+aw}{1-\frac aw}\right)^{\tfrac n2-r}}{\left(\tfrac{1+az}{1-\frac az}\right)^{ \tfrac n2 -s}}
\frac{1-\!\tfrac az}{1-\!\tfrac aw}\right).
\end{aligned} \label{43a}\ee
\end{proposition}
\begin{proof} As a first step, we prove the case $r=s=n/2$. 
The kernel $\BK_{n,m}(x,y)$ in (\ref{14}) can be expressed as 
\begin{equation}
\begin{aligned}
  \BK_{n,m}&( x,y)  
 = \frac{1}{(2\pi i)^2}\oint_{\Gamma_{0,a}}\frac{dz}{z}\rho^R_a(z)\oint_{\Gamma_{0,a,z}}\frac{dw}{w}\rho^L_a(w)\frac{w^{-y-m}}{z^{-x-m}}\frac{1}{1-\tfrac w z}\\
\\
&  \hspace*{0cm}\times\left( \left(\tfrac{w}{z}\right)^{2m+1}P_{2m+1}(z) \hat P_{2m+1}(w^{-1})-\hat P_{2m+1}(z^{-1})P_{2m+1}(w)\right) ,
\end{aligned}
\label{16}\end{equation}
 using the Christoffel-Darboux formula\footnote{This can be shown by generalizing an argument of B. Simon in \cite{Sim04}, first proof of Theorem 2.2.7.} for bi-orthonormal polynomials on the circle, (for $M=2m+1$)
\be
\sum^{M-1}_{k=0}\hat P_k(z^{-1})P_k(w)=\frac{z^{-M}P_M(z)w^M\hat P_M(w^{-1})-\hat P_M(z^{-1})P_M(w)}{1-\frac{w}{z}}.
\label{15}\ee
 One then uses the representation (\ref{31}) of the polynomials in terms of the $H_k^{(i)}$'s, yielding: 
\begin{equation}
\begin{aligned}
\lefteqn{\BK_{n,m}(x,y)} 
\\
&= \oint_{\Gamma_{0,a}}\frac{dz}{2\pi i} \oint_{\Gamma_{0,a,z}}\frac{dw}{2\pi i}\left( 
\frac{(1+az)\left(1-\frac{a}{z}\right)}{(1+aw)\left(1-\frac{a}{w}\right)} \right)^{n/2}\frac{w^{-y-m-1}}{z^{-x-m}}\frac{1}{z-w}\\
&~~~~\hspace*{7cm}\frac{H^{(1)}_{2m+1}(z^{-1})
H^{(2)}_{2m+1}(w)}{H_{2m+1}(0)H_{2m+2}(0)}\\
\\
&  -\oint_{\Gamma_{0,a}}\frac{dz}{2\pi i} \oint_{\Gamma_{0,a,z}}\frac{dw}{2\pi i}\left( 
\frac{(1+aw)\left(1-\frac{a}{w}\right)}{(1+az)\left(1-\frac{a}{z}\right)} \right)^{n/2}\left(\frac{1-\frac{a}{w}}{1-\frac{a}{z}}\right) \frac{w^{-y+m}}{z^{-x+m+1}}\frac{1}{z-w}\\
&~~~~\hspace*{7cm}\frac{H^{(1)}_{2m+1}(w^{-1})H^{(2)}_{2m+1}(z)}{H_{2m+1}(0)H_{2m+2}(0)}.
 \end{aligned} \label{43}\end{equation}
Using (\ref{45}), 
 one finds, using (\ref{38}) and (\ref{39}) for the Fredholm determinants and the function $\Phi(z)$ as in (\ref{44}), and setting $x\mapsto -x,~y\mapsto -y$, 
\be
\begin{aligned}
\lefteqn{
\frac{H_{2m+2}(0)}{H_{2m+1}(0)}\BK_{n,m}(-x,-y)}
\\&=
 \Bigl ( \frac{1}{2\pi i} \Bigr)^2
\oint_{\Gamma_{0,a}}dz 
\oint_{\Gamma_{0,a,z}}  \frac{dw}{z-w} ~\frac{w^{y -1}}{z^{x  }}~  e^{\tfrac n2 (\Phi(z)-\Phi(w))}
\\  &~~~~~~~~~~~~~~~~~~~~~~~~~~~~~~~~~~~~~~~~~~~~\times 
 (1-R^{(1)} (z^{-1}))(1-R^{(2)}_{ }(w))
 \\
  &+
 \Bigl ( \frac{1}{2\pi i} \Bigr)^2
\oint_{\Gamma_{0,a}}dw 
\oint_{\Gamma_{0,a,w}}  \frac{dz}{w-z} ~\frac{z^{-x -1}}{w^{-y}}~
 \left(\frac{1-\frac aw}{1-\frac az}\right)e^{\tfrac n2 (\Phi(w)-\Phi(z))}
 \\   &~~~~~~~~~~~~~~~~~~~~~~~~~~~~~~~~~~~~~~~~~~~~~~\times 
 (1-R^{(1)}_{ }(w^{-1}))(1-R^{(2)}_{ }(z))
 \\  
 &+ \frac{1}{2\pi i}
 \oint_{\Gamma_{0,a}}\frac{dz}{z^{x-y+1}}
   (1-R^{(1)}_{ }(z^{-1}))(1-R^{(2)}_{ }(z)).
 \end{aligned}
 \label{46}\ee
 %
 %
 Using the identity (\ref{42}) applied to the last term of $\BK_{n,m}(-x,-y)$, the dual kernel is then given by:
\be
\begin{aligned}
\lefteqn{\frac{H_{2m+2}(0)}{H_{2m+1}(0)}
\widetilde\BK_{n,m}(-x,-y)=\frac{H_{2m+2}(0)}{H_{2m+1}(0)}\left(\delta_{x,y}-\BK_{n,m}(-x,-y)\right)}%
\\&=
  \Bigl ( \frac{1}{2\pi i} \Bigr)^2
\!\oint_{\Gamma_{0,a}}\!dz 
\oint_{\Gamma_{0,a,z}} \! \frac{dw}{w\!-\!z} ~\frac{w^{y -1}}{z^{x }}~
e^{\tfrac n2 (\Phi(z)-\Phi(w))}
(1-R^{(1)}_{ }(z^{-1}))(1-R^{(2)}_{ }(w))    
 \\
  & 
+ \Bigl ( \frac{1}{2\pi i} \Bigr)^2
\oint_{\Gamma_{0,a}}dw 
\oint_{\Gamma_{0,a,w}}  \frac{dz}{z-w} ~\frac{z^{-x -1}}{w^{-y }}~
e^{\tfrac n2 (\Phi(w)-\Phi(z))}
  \left(\frac{1-\frac aw}{1-\frac az}\right)
 \\ 
&~~~~~~~~~~~~~~~~~~~~~~~~~~~~~~~~~~~~~~~~~~~~~~\times 
 (1-R^{(1)}_{ }(w^{-1}))(1-R^{(2)}_{ }(z))
 \\  
 &
 - \frac{\dt_{x\neq y}}{2\pi i}
 \oint_{\Gamma_{0,a}}\frac{dz}{z^{x-y+1}}
   (1-R^{(1)}_{ }(z^{-1}))(1-R^{(2)}(z))
.\end{aligned}
 \label{47}\ee
 Multiplying the kernel (\ref{47}) above by $(-1)^{x-y}$, interchanging $z\leftrightarrow w$ in the second double integration and then combining the two first integrals, one finds that the kernel $\widetilde\BK_{n,m}(-x,-y) $ consists of two parts, a double integral ${\cal K} (x,y)$and a single integral $C_1(0;x-y)$, defined in (\ref{55}),
 \be
\begin{aligned}
{(-1)^{x-y}\frac{H_{2m+2}(0)}{H_{2m+1}(0)}
\widetilde\BK_{n,m}(-x,-y)={\cal K} (x,y)+ C_1(0;x-y)}  
%
   ,
\end{aligned}
 \label{48}\ee
 where (recall $R^{(i)}$ is given by (\ref{39}))
 \be\begin{aligned}
 {\cal K} (x,y)&:=\Bigl ( \frac{1}{2\pi i} \Bigr)^2
\oint_{\Gamma_{0,a}}dz 
\oint_{\Gamma_{0,a,z}}  \frac{dw}{z-w} ~\left(\frac{(-w)^{y-1}}{(-z)^{x}}+
\frac{(-z)^y}{(-w)^{x+1} }\frac{1-\tfrac az}{1-\tfrac aw}\right)
  \\ \\&~~~~~~~~~~~~~~~~~~~~~~\times ~
  e^{\frac n2 (\Phi(z)-\Phi(w))}
 (1-R^{(1)}_{ }(z^{-1}))(1-R^{(2)}_{ }(w))
 \label{calK}\end{aligned}
 \ee
 and, confirming (\ref{55}), 
  $$\begin{aligned}
  C_1(0,x-y)=\frac{\dt_{x\neq y}}{2\pi i}
 \oint_{\Gamma_{0,a}}\frac{dz}{(-z)^{x-y+1}}
   (1-R^{(1)}_{ }(z^{-1}))(1-R^{(2)}_{}(z))
 . \end{aligned}
 $$
  Notice this establishes formula (\ref{K2}) of Theorem \ref{main1'} for $r=s=n/2$. 
   
 Using (\ref{Ei}), the double integral-part ${\cal K} $ of the kernel (\ref{48}) reads, upon multiplying out $(1-R^{(1)}_{ }(z^{-1}))(1-R^{(2)}_{ }(w))$, 
  $$
 \begin{aligned}
\lefteqn{{\cal K} (x,y)}    
\\
&  = \frac{1}{(2\pi i)^2}
 \oint_{\Gamma_{0,a}}dz\oint_{\Gamma_{0,a,z}}dw 
 \frac{1}{z-w}
\left(\frac{(-w)^{y-1}}{(-z)^{x}}+
\frac{(-z)^y}{(-w)^{x+1} }\frac{1-\tfrac az}{1-\tfrac aw}\right) e^{\frac n2 (\Phi(z)-\Phi(w))}
\\& ~~~~
\left(1-S_{ }^{(1)}(z^{-1})- \frac{e^{-n\Phi(z)}}{a-z}T_{ }^{(1)}(z^{-1})\right)
  \left( 1- S_{ }^{(2)}(w^{})-\frac {a-w}{   e^{-n\Phi(w)} }  T_{ }^{(2)}(w^{})\right)
\end{aligned}
$$
\be \begin{aligned}
&= \frac{1}{(2\pi i)^2}
 \oint_{\Gamma_{0,a}}dz\oint_{\Gamma_{0,a,z}}dw 
 \frac{\sum_1^3 E_i(z,w)-\frac wz\frac{1-\frac a w}{1- \frac a z}E_4(w,z)}{z-w}
 \\
 &\hspace*{6cm}\times\left(\frac{(-w)^{y-1}}{(-z)^{x}}+
\frac{(-z)^y}{(-w)^{x+1} }\frac{1-\tfrac az}{1-\tfrac aw}\right).
 \\
%
  \end{aligned}\label{52}\ee 
 %
%
%
 The part of the double integral, involving $E_4(w,z)$, is not in a usable form, in view of the saddle point method and the topology of the contours; the integrations have to be interchanged, at the expense of a new residue term. So, using (\ref{45}), the double integral ${\cal K} (x,y)$ involving $E_4(w,z)$ becomes, after a further interchange $z\leftrightarrow w$:
 \be
  \begin{aligned}
 =&   \frac{1}{(2\pi i)^2}
 \oint_{\Gamma_{0,a}}\!dz\oint_{\Gamma_{0,a,z}}\!   
 \frac{dw}{z-w}
 \left(\frac{(-z)^{y-1}}{(-w)^{x}}+
\frac{(-w)^y}{(-z)^{x+1} }\frac{1-\tfrac aw}{1-\tfrac az}\right)
 \frac zw\frac{1-\frac a z}{1- \frac a w}E_4(z,w)
\\
& -2\frac 1{2\pi i} \oint _{\Gamma_{0,a}} \frac{dz}{ (-z)^{ x-y+1}}T^{(1)}(z^{-1})T^{(2)}(z)
.\label{53} \end{aligned}
 \ee
 Defining ${\cal L}(x,y)$ as the double integral part of ${\cal K} (x,y)$, namely
 \be\begin{aligned}
{\cal L}(x,y):=&
\frac{1}{(2\pi i)^2}
 \oint_{\Gamma_{0,a}}\!dz\oint_{\Gamma_{0,a,z}}\!\!\!dw 
 \frac{\sum_1^4 E_i(z,w)}{z-w}
 \left(\frac{(-w)^{y-1}}{(-z)^{x}}+
\frac{(-z)^y}{(-w)^{x+1} }\frac{1\!-\!\tfrac az}{1\!-\!\tfrac aw}\right),
\end{aligned}\label{calL}\ee
and remembering the definition (\ref{55}) of $C_2(0,x)$, formulas (\ref{52}) and  (\ref{53}) show that 
$$ {\cal K} (x,y)={\cal L} (x,y)+2C_2(0;x-y)
 .$$
 Combining this formula with (\ref{48}), together with (\ref{55}), yields
  \be \begin{aligned}
 (-1)^{x-y}\frac{H_{2m+2}(0)}{H_{2m+1}(0)}
\widetilde\BK_{n,m}(-x,-y)  &={\cal K} (x,y)+C_1(0;x-y) \\
&={\cal L} (x,y)+2C_2(0;x-y)+C_1(0;x-y)
\\
&={\cal L} (x,y)+C(0;x-y)
,
\label{54}\end{aligned}\ee
thus establishing formula (\ref{43a}) of Proposition \ref{prop4.1}, with the $E_i$ and $C(0;x)$ as in (\ref{Ei}) and (\ref{55}), in the case $r=s=n/2$. 

The proof of the extended case follows the recipe (\ref{E1}) in Proposition \ref{prop} applied to the non-extended kernel $\widetilde\BK_{n,m}(x,y)$, namely:
$$\begin{aligned}
\lefteqn{\hspace*{-1cm}(-1)^{x-y}\frac{H_{2m+2}(0)}{H_{2m+1}(0) }\widetilde {\mathbbm K}_{n,m}^{{\rm ext}}(2r, x;2s, y)   } 
\\
=&-\Id_{s<r} (-1)^{x-y}\psi_{2(s-r)}( x, y) \frac{H_{2m+2}(0)}{H_{2m+1}(0)}
\\&+\psi_{n-2r}( x,\cdot)\ast(-1)^{x-y}\frac{H_{2m+2}(0)}{H_{2m+1}(0)}
\widetilde\BK_{n,m}(\cdot ~;\tc) 
 \ast\psi_{ 2s-n }(\tc~, y).
\end{aligned}
 $$
 Notice that the $\ast$-product only acts on the exponents of $w$ and $z$ in the expressions (\ref{calK}), (\ref{calL}) and (\ref{55}) for ${\cal K}$, ${\cal L}$ and $C_i(0;x)$. To do so, one first removes the $-$signs in $-z$ and $-w$ and one changes $x\mapsto -x$ and $y\mapsto -y$ and ones use the identities (\ref{astpower'}) below:
$$\begin{aligned}
\lefteqn{\psi_{n-2r}(x,u) \ast_{_{\!\tiny{u}}} \left(\frac{w^{-v-1}}{z^{-u}}+
\frac{z^{-v}}{w^{-u+1} }
\frac{1-\tfrac az}{1-\tfrac aw}\right)
\ast_{_{\!\tiny{v}}}
\psi_{2s-n}(v,y) 
}
\\
&\hspace*{3cm}= \left(\frac{w^{-y-1}}{z^{-x}}\frac{\left(\tfrac{1+az}{1-\frac az}\right)^{\tfrac n2-r}}{\left(\tfrac{1+aw}{1-\frac aw}\right)^{ \tfrac n2 -s}}+
\frac{z^{-y}}{w^{-x+1} }
\frac{\left(\tfrac{1+aw}{1-\frac aw}\right)^{\tfrac n2-r}}{\left(\tfrac{1+az}{1-\frac az}\right)^{ \tfrac n2 -s}}
\frac{1-\tfrac az}{1-\tfrac aw}\right)
\end{aligned}$$
and
\be\begin{aligned}
 {\psi_{n-2r}(x,u) \ast_{_{\!\tiny{u}}} \left(\frac{z^u}{z^v}\right)
\ast_{_{\!\tiny{v}}}
\psi_{2s-n}(v,y) 
}
 =z^{x-y}\left(\frac{1+ az}{1-\tfrac az}\right)^{s-r}.
\end{aligned}\label{109'}\ee
One now applies these formulas to the two expressions (\ref{48}) and (\ref{54}) for the non-extended kernel $\widetilde\BK_{n,m}( x, y)$:
\be \begin{aligned}\lefteqn{\hspace*{-.5cm}(- 1)^{x-y}\frac{H_{2m+2}(0)}{H_{2m+1}(0)}
\widetilde\BK_{n,m}( x, y)}\\
&   ={\cal K} (-x,-y)+C_1(0;-x+y) ={\cal L} (-x,-y)+C(0;-x+y)
,\end{aligned}
\label{110} 
\ee
where ${\cal K}$ and ${\cal L}$ are defined in (\ref{calK}) and (\ref{calL}). 
 %
 From (\ref{109'}), the $\ast$-multiplication amounts to inserting some elementary fractions. Thus the first representation (\ref{110}) leads to   \be
\begin{aligned}
\lefteqn{(-1)^{x-y}\frac{H_{2m+2}(0)}{H_{2m+1}(0)}
\widetilde\BK^{\rm ext}_{n,m}(2r, x;2s, y)} 
%
\\&=
-\Id_{s<r} (-1)^{x-y}\psi_{2(s-r)}( x, y) \frac{H_{2m+2}(0)}{H_{2m+1}(0)}
 \\
 & + \frac{\dt_{x\neq y}}{2\pi i}
 \oint_{\Gamma_{0,a}}\frac{dz}{(-z)^{-x+y+1}}
   (1-R^{(1)}_{ }(z^{-1}))(1-R^{(2)}_{}(z))
   \left(\frac{1+az}{1-\frac az}\right)^{s-r}
   \\
&+  \Bigl ( \frac{1}{2\pi i} \Bigr)^2
\oint_{\Gamma_{0,a}}dz 
\oint_{\Gamma_{0,a,z}}  \frac{dw}{z-w} ~(1-R^{(1)}_{ }(z^{-1}))(1-R^{(2)}_{ }(w))
e^{\frac n2 (\Phi(z)-\Phi(w))}
\\
&
~~~~~~\times\left(\frac{(-w)^{-y-1}}{(-z)^{-x}}\frac{\left(\tfrac{1+az}{1-\frac az}\right)^{\tfrac n2-r}}{\left(\tfrac{1+aw}{1-\frac aw}\right)^{ \tfrac n2 -s}}+
\frac{(-z)^{-y}}{(-w)^{-x+1} }
\frac{\left(\tfrac{1+aw}{1-\frac aw}\right)^{\tfrac n2-r}}{\left(\tfrac{1+az}{1-\frac az}\right)^{ \tfrac n2 -s}}
\frac{1-\tfrac az}{1-\tfrac aw}\right),
  \end{aligned}
 \label{48'}\ee
 which is expression {\em (i)}  for the extended kernel (formula (\ref{K2})), upon using the explicit expression (\ref{44}) for $e^{n \Phi(z)/2}$. Expression {\em (ii)} (formula (\ref{43a})) is an immediate consequence of the second representation (\ref{110}) for $\BK_{n,m}(-x,-y)$, with ${\cal L}$ as in (\ref{calL}). 
This ends the proof of Proposition \ref{prop4.1}, and also the proof of Theorem \ref{main1'}.\end{proof}



%
%



 
 \section{Scaling limits of various functions
}

\subsection{Limit of the functions $g_k^{(i)}$}

Remember the quantities $v_0,A,\rho,~\theta$ defined in (\ref{4}) in terms of the weight $0<a<1$ on vertical dominoes. Also the reader is reminded of the scaling (\ref{2}):
\be
\begin{aligned}
n&=2t , ~~~~~~~~~~~~~~~~~~~~~~m =
  \frac{2t}{a+a^{-1}}+  \sigma  \rho
  t^{1/3}
\\
x&=2a^2\theta \tau  t^{2/3}+\xi  \rho t^{1/3},~~
s  =t+(1+a^2)\theta\tau  t^{2/3}
.
\end{aligned}
\label{2'}\ee
In the expression (\ref{I2}), the extended kernel $\widetilde\BK^{\rm ext}_{n,m}$ involves discrete variables $k$ and $\ell$ and involves integration variables $u$ and $v$ in $K^{(i)}(z^{-1})$, which we rescale as follows:
 \be
\begin{aligned}
 k&  
 = \left[\frac{4t}{a+a^{-1}}+\kappa \rho (2t)^{1/3}+1\right] 
 ,~~~\ell  =\left[\frac{4t}{a+a^{-1}}+\lambda \rho
  (2t)^{1/3}+1\right] 
\\
 u&= v_0+\frac{U}{A (2t)^{1/3}}
 ,~~~~~~~~~~~~~~~~~~~~v =v_0+\frac{V}{A (2t)^{1/3}}
  \\
 z&=v_0+\frac{\zeta}{At^{1/3}},~~~~~~~~~~~~~~~~~~~~~~~~
 w =v_0+\frac{\omega}{At^{1/3}}
 .
\end{aligned}
\label{3'}\ee
The variables $z$ and $w$, with the rescaling above, will appear much later in Section \ref{Kintrepr}. Also note that $k,\ell\geq 2m+1$ implies $\kappa,\lambda \geq \tilde \sigma:=2^{2/3} \sigma$.

 Define the function $F(v)$ and $G(v)$, together with their Taylor series around the saddle point 
 $v=
v_0:= -\frac{1-a}{1+a},
$ and the Taylor series of $\log(-v )$ about $v_0$,
 \be
\begin{aligned}F(v)&:=\log (1+av)+\log (1-\tfrac av)+\frac{2}{a+a^{-1}}\log (-v)
\\
& =F(v_0)+\frac {A^3}3(v-v_0)^3+{\cal O}(v-v_0)^4
 \\
 G(v)&:= (1 +a^2)\theta \log \frac{  1+av}{1-\tfrac av }-2a^2\theta\log (-v)  
 \\  
&=G(v_0)- A^2 (v-v_0)^2+{\cal O}  (v-v_0)^3 ,
\\
 \log(-v)&=\log(-v_0)+\tfrac 1 {v_0}(v-v_0)+{\cal O}(v-v_0)^2,
 \end{aligned}\label{58}\ee
with
\be
G(v_0)=(1-a^2)\theta \log (-v_0).
\ee
 Given parameters $\tau$ and $\lambda$, set for future use: 
 \be\begin{aligned}
 M^{\tau}_{\lb}(t)&:=2tF(v_0)+\tau G(v_0)(2t)^{2/3}+  \lambda    \rho\log(-v_0)(2t)^{1/3} \\
 &=2tF(v_0)+\left((1-a^2)\theta \tau (2t)^{2/3}+\rho\lambda (2t)^{1/3}\right)\log(-v_0)
 \\
 M^{ }_{\lb}(t)&:=M^{\tau}_{\lb}(t)\Bigr|_{\tau=0}
.\end{aligned} \label{59}\ee
Note that, upon introducing the scaling (\ref{3'}) for $u$ and $v$, the function $\Phi(z)$ as defined in (\ref{44}), 
 has the following Taylor series about $v_0$, which then can be scaled using the scaling (\ref{3'}) for $v-v_0=\frac{V}{A(2t)^{1/3}}$,
 \be \begin{aligned}
 n \Phi(v)&=2tF(v)+  {\tilde \sigma}  \rho (2t)^{1/3}\log (-v)
\\
&= M_{\tilde \sigma}(t) +\tfrac 13 V^3-  { \tilde \sigma}    V +{\cal O}(t^{-1/3}),\end{aligned}\label{60}\ee
where $F$ and $M_{\tilde \sigma}$ are defined in (\ref{58}) and (\ref{59}). Combining this with $\log (-v)$ and, using the scaling (\ref{3'}) and (\ref{2'}) for $\ell -2m = (\lambda -\tilde \sigma) \rho (2t)^{1/3}+1$, and $v-v_0$, yields:
  \be
\begin{aligned}
 n\Phi(v)+(\ell-2m)\log(-v)&=2tF(v)+\lambda   \rho(2t)^{1/3}\log (-v)+\log(-v)
 \\&=M_\lambda (t)+\log (-v_0)+\tfrac 13 V^3-\lambda V +{\cal O}(t^{-1/3}).
  \end{aligned}\label{61}\ee
We shall need the following expressions in terms of the scaling (\ref{2'}) and (\ref{3'}), together with the summation variable $b=\beta \rho (2t)^{1/3}$~:  
 \be\begin{aligned}
 k+b&=\frac{4t}{a+a^{-1}}+(\kappa+\beta)\rho (2t)^{1/3}+1\\
 -x+m+b&= \frac{2t}{a+a^{-1}}
 -2a^2\theta \tau t^{2/3} 
 +(\sigma-\xi+2^{1/3}\beta)\rho t^{1/3}
 \\
 x-m+\ell & =\frac{2t}{a+a^{-1}}+ 2a^2\theta\tau t^{2/3}+(-\sigma+\xi+2^{1/3} \lambda) \rho t^{1/3}+1.%
\end{aligned}\label{62s}\ee
%
To do the asymptotics, the following Lemma from Borodin-Ferrari \cite{BF} comes in useful:

\begin{lemma}\label{BF}
Given a contour integral of the form 
\begin{equation}\begin{aligned}
I_t^{\pm}=\frac{1}{2\pi i}\oint_{\ga_{\{0,p_1,\ldots,p_n\}}}dz~e^{\pm [t f_0(z)+t^{2/3}f_1(z)+t^{1/3}f_2(z)+f_3(z)]},
\end{aligned}
\label{D85'}\end{equation}
where $\ga_{\{0,p_1,\ldots,p_n\}}$ is a closed contour, having as poles the points $0,p_1,\ldots,p_n$ only.
Assume that $\Re f_0(z)$ has a real saddle point $v_0$ and a steepest descent path $\ga '$ about $v_0$; assume that $\ga_{\{0,p_1,\ldots,p_n\}}$ can be deformed to $\ga '$, without passing through the points  $\{0,p_1,...,p_n\}$. Let $f_0(z)$ experience a double critical point at $v_0$, and $f_1(z)$ a single critical point:
$$
\begin{aligned}
f'_0(v_0)=f ''_0(v_0)=0, f '''_0(v_0)>0, f'_1(v_0)=0.
\end{aligned}
$$
Then, remembering the definition (\ref{E6}) of $\mbox{Ai}^{(s)}(x)$,
\begin{equation}\begin{aligned}
\lim_{t\rg +\iy}(t\kappa_0)^{1/3}e^{\mp[tf_0(v_0)+t^{2/3}f_1(v_0)+t^{1/3}f_2(v_0)+f_3(v_0)]}I_t^{\pm}
={ \rm sgn}(v_0) {\rm Ai}^{(\pm \kappa_1)}(\kappa_2),
\end{aligned}
\label{E8}\end{equation}
%
where the convergence is uniform  for $\kappa_1$, $\kappa_2$ in a bounded set, and with
%
%
\be
\begin{aligned}
 \kappa _0=\frac{f'''_0(v_0)}{2},\quad\quad 
 \kappa _1=\frac{f''_1(v_0)}{2\kappa_0^{2/3}},\quad\quad  \kappa _2=-\frac{f'_2(v_0)}{\kappa_0^{1/3}}.
\end{aligned}
\label{D85''}\ee
For any choice of $L>0$,
\begin{equation}\begin{aligned}
| t^{1/3}e^{\mp[tf_0(v_0)+t^{2/3}f_1(v_0)+t^{1/3}f_2(v_0)]}I_t^{\pm}|\leq c_0e^{-c_1\kappa_2}
\end{aligned}
\label{E9}\end{equation}
uniformly for $\kappa_2>-L$ and $\kappa_1$ in a bounded set; $c_0$ and $c_1$ are positive constants independent of $t$.
\end{lemma}

The next Lemma deals with the scaling limit of the $g_k^{(i)}$'s, given by (\ref{34}) and (\ref{E4}) and the $\tilde g_k^{(i)}$'s, given by (\ref{g}); also remember expression (\ref{59}) for $M_\lambda (t)$.

\begin{lemma}\label{L6.1}{\em (Limits of  $g_k^{(i)}$)}
 Given the scaling (\ref{3'}) for $\ell$, together with $b=\beta\rho (2t)^{1/3}$, 
  the following limits
  hold\footnote{$\Ai(\lambda)=\frac 1{2\pi i}\int _{\nearrow \atop \nwarrow } dV e^{\tfrac 13 V^3-\lambda V}$} : 
\be\begin{aligned}
\lim_{t\to \infty} (2t)^{1/3}\frac A {v_0-a} e^{- M_\lambda (t)}g_{\ell}^{(1)}(2t)&=
-\Ai (\lambda )\\
\lim_{t\to \infty} (2t)^{1/3} A {(v_0-a)}v_0^2 e^{ M_\lambda (t)}g_{\ell}^{(2)}(2t)&=
-\Ai (\lambda ),\\
 \end{aligned}\label{62a}\ee
 where the limit is uniform for $\lambda $ in a bounded set. Also, for $\ell$ as above and for any choice of $L>0$,
 \be
 \left|  (2t)^{1/3}e^{\mp  M_\lambda(t)} g_\ell^{({1\atop 2})}(2t) \right|<c_0e^{-c_1 \lambda}
 ,\label{D87'}\ee
 uniformly for $\lambda>-L$, with positive constants $c_0$ and $c_1$ independent of $t$. 
 Also, using the scaling (\ref{2'}) and (\ref{3'}), one obtains
 \be\begin{aligned}
\lim_{t\to \infty} t^{1/3}\frac {Av_0} {v_0-a} e^{- M^{\tau}_{\sigma-\xi+2^{1/3} \beta} (t/2)}g_{-x+m+b}^{(1)}( t;s-t)&=
 \Ai ^{(-\tau)}(\sigma-\xi+2^{1/3} \beta )\\
\lim_{t\to \infty} t^{1/3} Av_0 {(v_0-a)} e^{ M^{\tau}_{\sigma-\xi+2^{1/3} \beta} (t/2)}g_{-x+m+b}^{(2)}( t;t-s)&=
 \Ai^{ (\tau)} (\sigma -\xi+2^{1/3} \beta ), \\
 \end{aligned}\label{62'}\ee
  where the limit is uniform for $\xi,\tau,~\beta$ in  bounded sets . Also, for any choice of $L>0$,
 \be
 \left|
 t^{1/3}      e^{\mp M^{\tau}_{ \sigma-\xi+2^{1/3} \beta} (t/2)}  g_{-x+m+b  }^{{(1)}\atop{(2)}}( t;\pm(s-t))
 \right|
 <c_0e^{-c_1 (2^{1/3}\beta-\xi)}
 \label{D88'}\ee
 uniformly for $-\xi +2^{1/3} \beta>-L$, with $c_0$ and $c_1$ positive $t$-independent constants. Furthermore, with the scaling (\ref{2'}) and (\ref{3'}),
 \be\begin{aligned}
\lim_{t\to \infty} t^{1/3}  A   e^{- M^{-\tau}_{-\sigma+\xi+2^{1/3} \lambda} (t/2)}\widetilde g_{x-m+\ell  }^{(1)}( t;t-s)&=
-\Ai^{(\tau)} (-\sigma+\xi+2^{1/3} \lambda )\\
\lim_{t\to \infty} t^{1/3} Av_0 ^2  e^{ M^{-\tau}_{-\sigma+\xi+2^{1/3} \lambda} (t/2)} \widetilde g_{ x-m+\ell }^{(2)}( t;s-t)&= 
-\Ai^{(-\tau)} (-\sigma+\xi+2^{1/3} \lambda ),\\
 \end{aligned}\label{62''}\ee
 where the limit is uniform for $\lambda$ in a bounded set and for any choice of $L>0$,
 \be
 \left|
 t^{1/3}      e^{\mp M^{\tau}_{-\sigma+\xi+2^{1/3} \lambda} (t/2)}\widetilde g_{x-m+\ell  }^{{(1)}\atop{(2)}}( t;\pm(t-s))
 \right|
 <c_0e^{-c_1(\xi+2^{1/3}\lambda)}
\label{D89'} \ee
 uniformly for $\xi +2^{1/3} \lambda>-L$, for $c_0$ and $c_1$ positive $t$-independent constants.
\end{lemma}

\begin{proof}Indeed, in view of a saddle point argument, consider the function $\Re F(x+iy)$, for $F(v)$ as in (\ref{58}); it tends to $-\infty$ near $(x,y)=(a,0)$ and $(x,y)=(-1/a,0)$ and tends to $\infty$ near $(x,y)=(0,0)$ and at $\infty$ on the Riemann sphere; near its saddle point $v=
v_0:= -\frac{1-a}{1+a}<0,
$ the function $\Re F(x+iy)$ looks like $\Re (x+iy)^3$, with the three separatrices through $v=v_0$ in the upper-half plane $y>0$  corresponding to the three separatrices of the $\Re (x+iy)^3$ profile in $ y>0$; they are indicated by dotted lines in Figure 18. 
So, the levels profile consists of (topological) circles about $0,~a$, $-a^{-1}$ and $\infty$, separated by the separatrices. A path $\cal C$ of steep descent departing from the saddle point will be given by, say, extending a small straight line making an angle $ \pi/6<\theta<\pi/2$, as indicated in figure 18, 
 until it hits a level curve in that quadrangle, which one then follows until it hits the $x$-axis.  
 Then one takes the mirror-image with respect to the $x$-axis. The contour thus obtained winds around $0$, but not around $a$. So, for the function $g_\ell^{(1)}$, one is in the condition of Lemma \ref{BF}.

The function $g_\ell^{(1)}$, involving integration about a contour $\Gamma_0$ (about $0$ and not about $a$), takes on the following form, using the first equality of (\ref{61}),
$$
g_\ell^{(1)}(2t)
=\oint_{\Gamma_0} \frac{-dv}{2\pi i}
(1-\tfrac av )e^{2t\Phi(v)+(\ell-2m)\log(-v)} 
= \oint_{\Gamma_0}\frac{-dv}{2\pi i}
e^{2tf_0(v)+  (2t)^{1/3}f_2(v)+f_3(v)}, 
$$
with 
$$
f_0(v)=F(v),~~f_2(v)=\lambda \rho\log (-v),~~  f_3(v)=\log(-v)+\log(1-\tfrac av).
$$
with
$$
f_0'(v_0)=f_0''(v_0)=0,~~\tfrac 12 f_0'''(v_0)= A^3>0, ~~f_2'(v_0)=-A\lambda <0,
$$
satisfying the requirements of Lemma \ref{BF}.

Then, by Lemma \ref{BF}, one has that for any $L>0$, keeping in mind the definition (\ref{59}) of $M_\lambda(t)$,
\be\left|  (2t)^{1/3} g_\ell^{(1)}(2t) e^{-2tf_0(v_0)- (2t)^{1/3}f_2(v_0)}\right|
=\left|  (2t)^{1/3} g_\ell^{(1)}(2t) e^{- M_\lambda(t)}\right|<c_0 e^{-c_1\lambda},\label{62u}\ee
uniformly for $\lambda>-L$, with positive constants $c_0,~c_1$, independent of $t$. Then from (\ref{61}) and using the usual saddle argument, it follows that for $\lambda \in K$, a bounded set,
$$
(2t)^{ 1/3}e^{- M_\lambda(t)}g_\ell^{(1)}(2t)
= \frac{a-v_0}{A}
\int _{\nearrow \atop \nwarrow }  \frac{dV}{2\pi i}e^{\frac 13 V^3-\lambda V}(1+{\cal O}(t^{-1/3}))
,$$
with ${\cal O}(t^{-1/3})$ just depending on the set $K$. The uniformity follows from Lemma \ref{BF}. Indeed, in the limit $t\to \infty$, the only contribution of the integral comes from a small neighborhood of the path $\cal C$ of steepest descent, near the saddle point $v_0$ in $v$-coordinates; this yields two rays emanating from $0$ in $V$-coordinates 
 with an angle $ \pi/6<\theta<\pi/2$. This proves the uniform limit (\ref{62a}) for $g_\ell^{(1)}(2t)$.

The limit  (\ref{62a}) for $$
g^{(2)}_{\ell}(2t)=   \oint_{\Gamma_{0,a}} \frac{-dv}{2\pi i}\frac{ e^{-n\Phi(v)-(\ell-2m)\log(-v)}}{v^2(1-\tfrac av)}
  $$ is similar, except that, since $n\Phi(v)+(\ell-2m)\log(-v)$ appears with a minus sign in the exponential, one looks for a path of steepest ascent. 
 Such a path consists of a 
 loop about $0$ and $a$, passing through the saddle: take a line in the direction $ \pi/2<\theta<2\pi/3$, until it hits a level curve, then follow this level around $a$; then, take the mirror image with respect to the $x$-axis. 
 So, this leads to a steepest ascent contour about $0$ and $a$, as depicted in Figure 18. Then, by Lemma \ref{BF}, or more directly upon taking the limit, when $t\to \infty$, the only contribution comes again from a very small neighborhood of $v_0$. One also has a similar uniform  estimate by the arguments of Lemma \ref{BF}, for any $L>0,~\lambda>-L$,
 \be\left|  (2t)^{1/3} g_\ell^{(2)}(2t) e^{ 2tf_0(v_0)+\lambda (2t)^{1/3}f_2(v_0)}\right|
=\left|  (2t)^{1/3} g_\ell^{(2)}(2t) e^{  M_\lambda(t)}\right|<c_0 e^{-c_1\lambda},\label{62v}\ee
with $t$-independent positive constants $c_i$.

 \newpage
 \vspace*{-2cm}
 
  \hspace*{-1cm} 
 {\includegraphics[width=150mm,height=150mm]{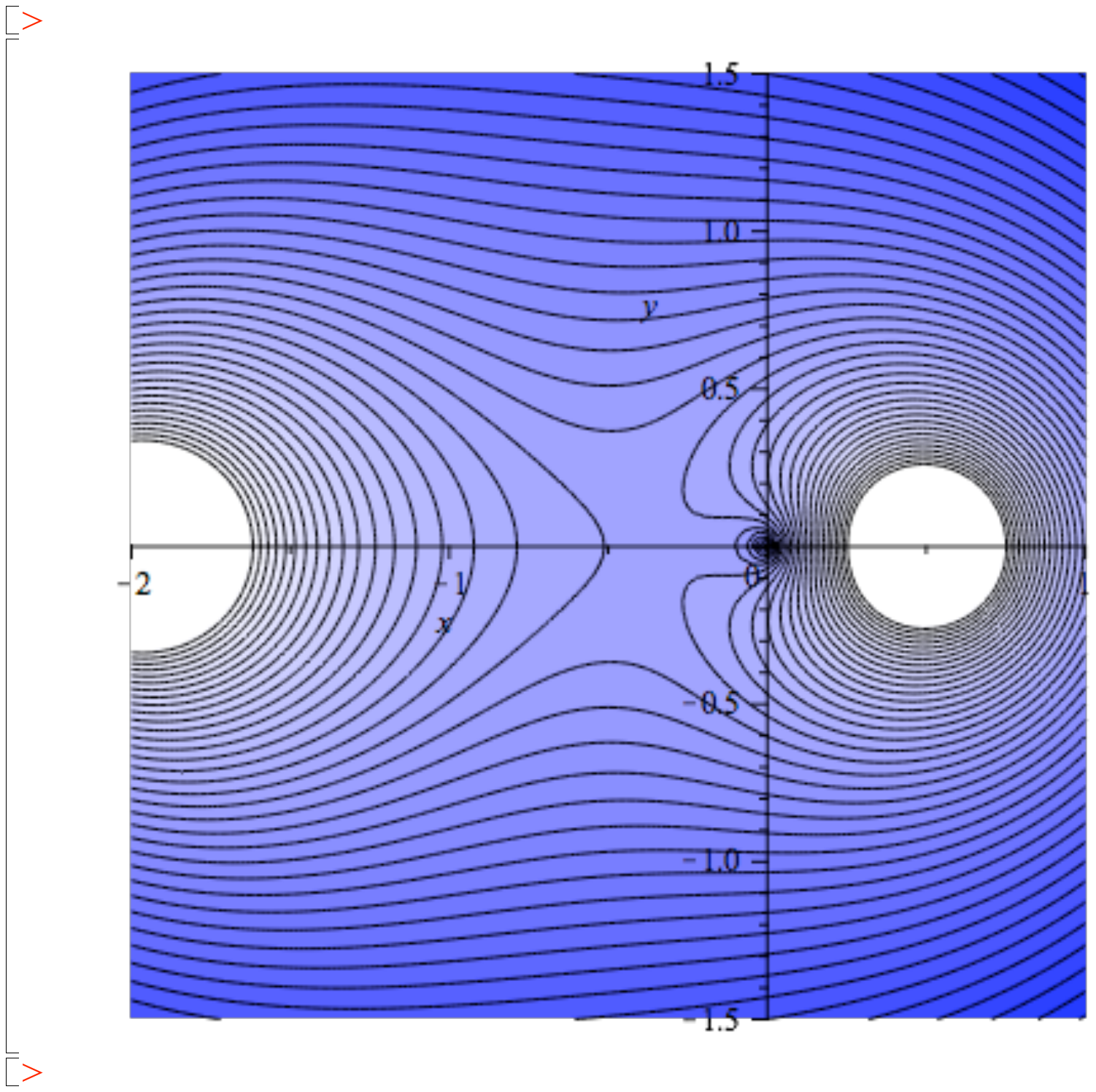}} 
 
  \vspace*{-10.55cm}


 \vspace{.26cm}
 
  \hspace*{5cm}
  \setlength{\unitlength}{0.015in}\begin{picture}(0,0)
%
  \thicklines
    \put(45,3){\line(-1,3){8}}
    \put(45,3){\line(-1,-3){7.5}}
     
          \put(65,38){\vector(-3,-1){1}}
        \put(69,-32){\vector( 3,-1){1}}
       
          \put(140,39){\vector(-4,3){1}}
        \put(140,-30){\vector( 4, 3){1}}
        
            \put(39,21){\vector(  1,-3){1}}
              \put(39.5,-14){\vector(-1,-3){1}}
    
      \put(45,3){\line( 2, 1){41}}
    \put(45,3){\line(2,-1){41}}
     \put(70,15){\vector(-2,-1){1}}
      \put(67,-8){\vector(2,-1){1}}
       \put(80,10){\vector(1,5){1}}
       \put(80,-10){\vector(1,5){1}}

    \put(45,3){\makebox(0,0) {\footnotesize{$\bullet$}}}
    \multiput(45,3)( 0,1){20}{\line( 0,1){.1}}
 \multiput(45,3)( 0,-1){20}{\line( 0,-1){.1}}
 
    \multiput(45,3)( 6,1){4}{\line( 6,1){3}}
 \multiput(45,3)( 6,-1){4}{\line( 6,-1){3}}

    \multiput(45,3)( -6,4){7}{\line( -6,4){2}}
 \multiput(45,3)( -6,-4){7}{\line( -6,-4){2}}
 
 \put(113,3){\makebox(0,0) {\tiny{$\bullet$}}}
  \put(110,6){\makebox(0,0) {\footnotesize{$a$}}}
  \put(-88,3){\makebox(0,0) {\tiny{$\bullet$}}}
   \put(-83,9){\makebox(0,0) {\footnotesize{$ -1/a$}}}

  \put(67,9){\makebox(0,0) {\footnotesize{$ z$}}}
    \put(50,15){\makebox(0,0) {\footnotesize{$ w$}}}
   
%
\end{picture}
 
 \vspace*{ 3cm}


 \hspace*{-1cm} {\includegraphics[width=150mm,height=150mm]{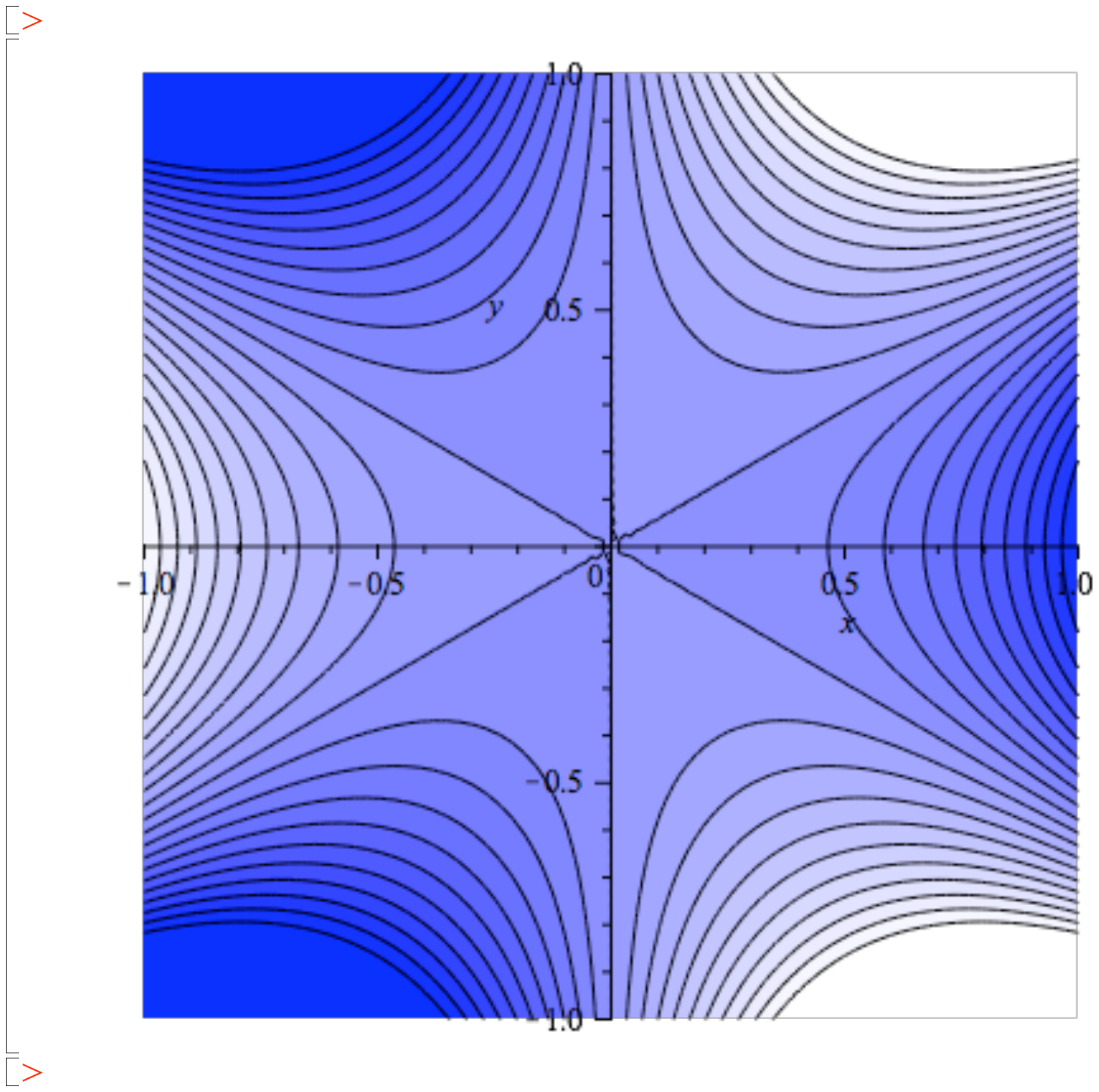}}


 \vspace*{ -6cm}
 
Figure 18.~~Plot of level curves\footnote{Near very dark regions the function tends to $\infty$ and near clear regions to $-\infty$}  for $
\Re F(x+iy)$, for $F(v)$ as in (\ref{58}) with $a=1/2$ and saddle at $v_0=-\tfrac{1-a}{1+a}=-1/3.$ The six dotted curves (departing from $v_0$) are the separatrices of the level profile of $
\Re F(x+iy)$ and correspond in the lower picture to the six separatrix curves of $ \Re (x+iy)^3$ departing from the origin.


\bigbreak

\newpage



 To prove (\ref{62'}), we invoke Lemma \ref{BF}, applied to the following functions, with the scaling (\ref{2'}) and (\ref{3'}), together with $b=\beta\rho (2t)^{1/3}$. So we must prove that
 $$
\begin{aligned}
 g^{(1)}_{ -x +m+b }(t;s-t) =-I^+(t) \mbox{  and  }
 g^{(2)}_{ -x +m+b }(t;t-s)=-I^-(t)
\end{aligned}
$$
with $I^{\pm}_t$ of the form (\ref{D85'}). Using the explicit expression (\ref{E4}) of $g_\ell^{(i)}(t;s-t) $, using expression (\ref{60}) for $\Phi(z)$, formula (\ref{58}) for $G(z)$, and the scaling (\ref{62s}) for $-x+m+b$, one checks : 
$$\begin{aligned}
 \lefteqn{g^{(1)}_{ -x +m+b }(t;s-t)}
 \\
 &=-\oint_{\Gamma_{0,a}}
 \frac{dz}{2\pi i}  
 e^{\tfrac n2\Phi(z)+(b-x )\log(-z)+(s-t)\log \tfrac{1+az}{1-\tfrac az}+\log (1-\tfrac az)} 
 \\
 &=-\oint_{\Gamma_{0,a}}
 \frac{dz}{2\pi i}  
 e^{tF(z)+ G(z)\tau t^{2/3}+(\sigma-\xi+\beta 2^{1/3})\log(-z)\rho t^{1/3}+\log (1-\tfrac az)},
\end{aligned} $$
and similarly for $g^{(2)}_{ -x +m+b }(t;s-t)$. So, for both cases: 
$$
\begin{aligned}
  f_0(z)&=F(z), ~~f_1(z)= \tau G(z),~~f_2(z)=(\sg -\xi +2^{1/3}\beta)\rho\log (-z)
\\
f_3(z)&=\left\{\begin{aligned}&\log (1-\tfrac{a}{z}),\mbox{   for  } I_t^+\\
             &\log ((1-\tfrac{a}{z})z^2)
             ,\mbox{   for  } I_t^-
             \end{aligned}\right.
,\end{aligned}
$$
for which one checks from (\ref{58}) and (\ref{4})  that
$$\kappa_0 =A^3,~~~\kappa_1=-\tau,~~~\kappa_2=\sg -\xi +2^{1/3}\beta,
$$
and consequently, from (\ref{59}),
$$
tf_0(v_0)+t^{2/3}f_1(v_0)+t^{1/3}f_2(v_0)+f_3(v_0)=M^{\tau}_{ \sg-\xi +2^{1/3}\beta }(\tfrac{t}{2})
+\left\{\begin{aligned}
&\log \tfrac{v_0-a}{v_0}\\
&\log   (v_0\!-\!a)v_0
\end{aligned}
\right.  .
$$
This and $\mbox{sgn} ~v_0=-1$ leads to the limit (\ref{62'}), together with the bound (\ref{D88'}).

\bigbreak

To see (\ref{62''}), one shows that
$$
\begin{aligned}  
 \tilde g^{(1)}_{ x-m+\ell }(t;s-t)=I^{+}_t
 \mbox{  and  }
 \tilde g^{(2)}_{ x-m+\ell }(t,s-t)=I^{-}_t
\end{aligned}
$$
with $I^{\pm}_t$ as in Lemma \ref{BF}. Using (\ref{62s}), one checks, for both cases,
$$
\begin{aligned}
 f_0(z)&=F(z), f_1(z)=- \tau G(z),~~f_2(z)=(-\sg +\xi +2^{1/3}\lb)\rho\log (-z)  
\\
f_3(z)&=\left\{\begin{aligned}&0,\mbox{   for  } I_t^+\\
             &\log z^2
             ,\mbox{   for  } I_t^-
             \end{aligned}\right.
, \end{aligned}
$$
for which one checks that
$$\kappa_0 =A^3,~~~\kappa_1= \tau,~~~\kappa_2=-\sg +\xi +2^{1/3}\lambda.
$$
and consequently
$$
tf_0(v_0)+t^{2/3}f_1(v_0)+t^{1/3}f_2(v_0)+f_3(v_0)=M^{-\tau}_{ -\sg+\xi +2^{1/3}\lambda }(\tfrac{t}{2})
+\left\{\begin{aligned}
&0\\
& \log ( v_0^2)
\end{aligned}
\right.
$$
This leads to the limit (\ref{62''}) and the bound (\ref{D89'}), ending the proof of Lemma \ref{L6.1}.\end{proof}

\subsection{Limit of  $a_{-y,s}(k)$, $b_{-x,r}(\ell)$, $K^{(i)}_{k\ell} $ and $S(2r,x;2s,y)$
}

In the usual scaling (\ref{2'}) and (\ref{3'}) one replaces $x$ by two variables $x$ and $y$, and also $s$ by $r$ and $s$, leading to new variables $\xi_1,~\xi_2$ and $\tau_1,~\tau_2$ ; so the scaling is now:
\be
\begin{aligned}
&\mbox{scaling  (\ref{2'}) and (\ref{3'}), together  with }\\
 &x=2a^2\theta \tau_1 t^{2/3}+\xi_1 \rho t^{1/3},~~~~~y=2a^2\theta \tau_2 t^{2/3}+\xi_2 \rho t^{1/3}
\\
&r =t+(1+a^2)\theta\tau_1 t^{2/3},~~~~~~~s =t+(1+a^2)\theta\tau_2 t^{2/3}
.\\
\end{aligned}
\label{2''}\ee
The reader is reminded of the definitions (\ref{E6}) for $ {\cal A}^{(\tau)}_\xi(\kappa) $ and $K_{\Ai}^{(\tau_1,-\tau_2)}(u_1,u_2)$.

\vspace{.1cm}

\begin{lemma}\label{L6.3} {\em (Scaling limit of $a_{-y,s}(k)$, $b_{-x,r}(\ell)$ and $K^{(1)}_{k\ell}(0)$)} Given the scaling of $x,y$ and $k,\ell,r,s$ as in (\ref{2''}), one obtains for $a_{-y,s}(k)$ and $b_{-x,r}(\ell)$, 
\be\begin{aligned}  \lim_{t\to \infty} 
 {(-1)^m} (\rho v_0t^{1/3}) ~e^{ M^{-\tau_2}_{\xi_2-\sigma+2^{1/3} \kappa}(  t /2)} a_{-y,s}(k)&={\cal A} ^{-\tau_2}_{\xi_2-\sg}(\kappa) 
 \\
 \lim_{t\to \infty} 
 {(-1)^m} (- At^{1/3})~ e^{ -M^{-\tau_1}_{\xi_1-\sigma+2^{1/3} \lambda}(t /2)} b_{-x,r}(\ell)&={\cal A} ^{ \tau_1}_{\xi_1-\sg}(\lambda) 
 \\
%
\end{aligned}\label{64c}\ee
with the limit uniform for $\xi_1,~\xi_2,\tau_1,~\tau_2$ in a bounded set. Also for any $L>0$,
\be
\begin{aligned}
\left| t^{1/3} e^{ M^{-\tau_2}_{-\sigma+\xi_2+2^{1/3} \kappa}(  t/ 2)} a_{-y,s}(k)\right|
&<c_0 e^{-c_1(\xi_2+2^{1/3}\kappa)}
\\
\left|\
t^{1/3} e^{ -M^{-\tau_1}_{-\sigma+\xi_1+2^{1/3} \lambda}( t /2)} b_{-x,r}(\ell) \right|
&<c_0 e^{-c_1(\xi_1+2^{1/3}\lambda)}
,\end{aligned}
\label{D96'}\ee
uniformly for $\xi_2+2^{1/3}\kappa$ and $\xi_1+2^{1/3}\lambda> -L$ for $c_0$, $c_1$ positive $t$-independent constants. 
Finally,
 \be \begin{aligned}
 \lim_{t\to \infty}(2t)^{1/3}e^{ M_{\kappa}(t)}K^{(1)}_{k\ell}(0)e^{- M_{\lb}(t)} 
&=
   \lim_{t\to \infty}  (2t)^{1/3}e^{ M_{\lb}(t)}K^{(2)}_{k\ell}(0)e^{- M_{\kappa}(t)} \\
  &=  \int_0^{\infty} d\beta ~\Ai(\kappa+\beta)\Ai(\lambda+\beta)= K_\Ai(\kappa,\lambda)  . \end{aligned}\label{62K}\ee 
  where the convergence is uniform for $\kappa,~\lambda $ in bounded sets. One has that
\be
  \left| (2t)^{1/3}e^{ M_{\lb}(t)}K^{(2)}_{k\ell}(0)e^{- M_{\kappa}(t)} \right|<c_0e^{-c_1(\kappa+\lambda)}
\label{D90'}\ee
uniformly for $\kappa,~\lambda>-L$, with $c_0$ and $c_1$ positive $t$-independent constants.
\end{lemma}

\begin{proof}
In the limits below, the $M$'s defined in (\ref{59}) satisfy the following identity, useful later on:
\be\begin{aligned}
M^{\tau}_{\sigma-\xi+2^{1/3} \beta} (t/2) 
  &+M^{-\tau}_{-\sigma+\xi+2^{1/3} \kappa} (t/2)
=M_{\kappa+\beta}(t)
.\end{aligned}\label{64m} \ee
 Setting $n=2t$, as usual, and using (\ref{62a}), (\ref{62'}) and (\ref{64m}), one checks, by substituting the scaling (\ref{62s}) for $-x+m+b$ and $k+b$, with $b=\rho\beta(2t)^{1/3}$, that
\be\begin{aligned}
& 
\lim_{t\to \infty} \rho v_0  t^{1/3}~e^{  M^{-\tau_2}_{ -\sigma+\xi_2+2^{1/3} \kappa} (t/2)}\sum_{b=0}^\infty 
g^{(1)}_{-x_2+m+b}(t; s-t)
g_{k+b}^{(2)}(2t)  
\\
=&-\lim_{t\to \infty}
 \sum_{\beta\in \rho^{-1}(2t)^{-1/3}{\mathbb N}}   \frac{(\Dt b)}{\rho (2t)^{1/3}}
  \left(t^{1/3}\frac {Av_0} {v_0\!-\!a} e^{- M^{\tau_2}_{\sigma-\xi_2+2^{1/3} \beta} (t/2)}
g^{(1)}_{-x_2+m+b}(t;s\!-\!t)\right)
\\
&
 ~~~~~~~~~~~~~~~~~~~~~~~~~~~~~~~~~~~~~~~
\left( (2t)^{1/3} A {(a-v_0)}v_0^2 e^{ M_{\kappa+\beta} (t)}g_{k+b}^{(2)}(2t)\right)%
\\
&=- \int_0^\infty \Ai^{(-\tau_2)} (\sigma-\xi_2+2^{1/3} \beta )
\Ai (\kappa+ \beta )d\beta .\end{aligned}
\label{64'} \ee
The last limit holds by dominated convergence, using the estimates (\ref{D88'}) and (\ref{D87'}); the sum above does not exceed in absolute value the following sum
$$
(2t)^{-1/3} c_0\sum_{\beta\in \rho^{-1}(2t)^{-1/3}{\mathbb N}}  e^{-c_1(\sigma-\xi_2+2^{1/3}\beta)}e^{-c_1 (\kappa+\beta)}
 \leq c'_0 e^{-c_1( -\xi_2+\kappa)}.
 $$
%
%
Also, from (\ref{62''}),  
\be\begin{aligned}
&\lim_{t\to \infty}(\rho v_0 t^{1/3})~
e^{ M^{-\tau_2}_{-\sigma+\xi_2+2^{1/3} \kappa}(t/2)}
\widetilde g^{(2)}_{ x_2-m+k }(t;s-t)  
\\
&= \Ai^{(-\tau_2)} (-\sigma+\xi_2+2^{1/3} \kappa)
. \label{64''}\end{aligned}\ee
Combining (\ref{64'}), (\ref{64''}) and (\ref{ab0}), one concludes the first formula of (\ref{64c}). 
 From (\ref{62a}), (\ref{62'}) and (\ref{64m}), it follows in a similar fashion that
\be\begin{aligned}
& 
\lim_{t\to \infty} (-At^{1/3})~e^{ - M^{-\tau_1}_{ -\sigma+\xi_1+2^{1/3} \lambda} (t/2)}\sum_{b=0}^\infty 
g_{\ell+b}^{(1)}(2t)  
g^{(2)}_{-x_1+m+b }(t;t-r)
\\
&=-\lim_{t\to \infty}
 \sum_{b=0}^\infty \frac{(\Dt b)}{\rho(2t)^{1/3}}
 \left((2t)^{1/3}  \frac A {a-v_0} e^{- M_{\lambda+ \beta} (t )}
g^{(1)}_{\ell+b }(2t)\right)
\\&\hspace*{4cm}\times \left( t^{1/3} A {(v_0-a)}v_0  e^{ M^{\tau_1} _{\sigma-\xi_1+2^{1/3}\beta} (t/2)}g_{-x_1+m+b }^{(2)}( t;t-r)\right)%
\\
&=- \int_0^\infty \Ai^{(\tau_1)}  (\sigma-\xi_1+2^{1/3} \beta )
\Ai(  \beta+\lambda )d\beta .\end{aligned}
\label{64a} \ee
From (\ref{62''}), one obtains 
\be\begin{aligned}
&\lim_{t\to \infty}  (-At^{1/3})~
e^{- M^{-\tau_1}_{-\sigma+\xi_1+2^{1/3}  \lambda}(t/2)}
\widetilde g^{(1)}_{x_1-m+\ell }(t;t-r)
 %
 =   \Ai^{(\tau_1)}  (-\sigma+\xi_1+2^{1/3} \lambda)
.\end{aligned}
\label{64b} \ee
Assembling (\ref{64a}) and (\ref{64b}), one deduces the second formula (\ref{64c}) from (\ref{ab0}). Next one uses (\ref{D87'}), (\ref{D88'}), (\ref{D89'}) and an argument analogous to the previous argument showing the estimate for dominated convergence; one deduces (\ref{D96'}). Finally, $a_{-y,s}(k)$ and $b_{-x,r}(\ell)$ tend to their respective limits uniformly on bounded sets by the ``uniform convergence on bounded sets"-statement in Lemma \ref{BF}.

In order to take the scaling limit of the kernels $K_{k\ell}^{(i)}$, using the scaling (\ref{2''}), one sets $b=\rho\beta (2t)^{1/3}$, and uses the infinitesimal elements
\be\begin{aligned} \frac {\Dt \ell}{\rho(2t)^{1/3}}=d\lambda 
~~\mbox{and}~~ \frac {\Dt b}{\rho(2t)^{1/3}}=d\beta \end{aligned}
 ,\label{63}\ee
 where $\Dt\ell=1$ and $\Dt b=1$; using these, one obtains from (\ref{35a}) and (\ref{62a}),
$$\begin{aligned}
\lefteqn{\lim_{t\to \infty}e^{ M_{\kappa}(t)}K^{(1)}_{k\ell}(0)e^{- M_{\lb}(t)}\Dt \ell}\\
&\stackrel{*}{ =}\lim_{t\to \infty}\frac {\Dt \ell}{(2t)^{1/3}\rho}
\left[\sum_{\beta\in\rho^{-1}(2t)^{-1/3}{\mathbb N}} 
\frac {\Dt b}{(2t)^{1/3}\rho}\right.
\\
& ~~~~~~~\times \Bigl(  
(2t)^{1/3}\frac A{v_0-a} e^{- M_{\lb+\beta} (t)} g^{(1)}_{\bigl[\frac{4t}{a+a^{-1}}+(\lambda+\beta)\rho (2t)^{1/3}+1\bigr]}(2t)
\Bigr) 
\\ &~~~~~~~\times \left.
\Bigl((2t)^{1/3}A{(v_0-a)}v_0^2 e^{ M_{\kappa+\beta} (t)} g^{(2)}_{\bigl[\frac{4t}{a+a^{-1}}+(\kappa+\beta)\rho (2t)^{1/3}+1\bigr]}(2t)\Bigr)
\begin{array}{c}  \\ \\ \\     \end{array} \right] 
\end{aligned}$$
\be \begin{aligned}&=d\lambda \int_0^{\infty} d\beta ~\Ai(\kappa+\beta)\Ai(\lambda+\beta)= K_\Ai(\kappa,\lambda) d\lambda.
 \end{aligned}
\label{64} \ee
Note that the sum in $\stackrel{*}{ =}$ can be viewed as an integral of a piecewise continuous function. Then $\lim_{t\to \infty}$ and integration can be exchanged, by using dominated convergence. Indeed, using (\ref{62u}) and (\ref{62v}), the sum in $\stackrel{*}{ =}$, before taking the limit, denoted (say) by  ${\cal K}^{(1)}_{t}(\kappa,\lambda)$, can be bounded in absolute value by
\be
|{\cal K}^{(1)}_{t}(\kappa,\lambda)|\leq c_0\sum_{\beta\in \rho^{-1}(2t)^{-1/3}{\mathbb N}}
\frac 1{(2t)^{1/3}} 
e^{-c_1(\lambda+\beta)}e^{-c_1(\kappa+\beta)}
\leq c_2 e^{-c_1 ( \lambda+\kappa)}
,\label{D94'}\ee
for $\lambda,\kappa>-L$ and with positive constants $c_i$ independent of $t$, yielding (\ref{D90'}). This justifies the existence of the uniform limit to the Airy kernel for $\kappa$ and $\lambda$ in a bounded set. 
Moreover, since 
$$
(2t)^{1/3}e^{ M_{\lb}(t)}K^{(2)}_{k\ell}(0)e^{- M_{\kappa}(t)} =(2t)^{1/3} e^{ M_{ \lb}(t)}K^{(1)}_{\ell k}(0)e^{- M_{ \kappa}(t)} ,
$$
one deduces from the above
\be
 \lim_{t\to \infty}  (2t)^{1/3}e^{ M_{\lb}(t)}K^{(2)}_{k\ell}(0)e^{- M_{\kappa}(t)}
= \lim_{t\to \infty}
(2t)^{1/3}e^{ M_{ \lb}(t)}K^{(1)}_{\ell k}(0)e^{- M_{ \kappa}(t)}
=
K_\Ai(\lb,\kappa)  .
\label{65}\ee
This concludes the proof of lemma \ref{L6.3}.
\end{proof} 


\begin{lemma}\label{L6.4}
{\rm (Scaling limit of $S(2r,x;2s,y)$)}. One obtains for the scaling of $x,y,r,s$ as in (\ref{2''}),
\begin{equation}\begin{aligned}
 {\lim_{t\to \infty}\rho t^{1/3}
(-v_0)^{y-x+r-s}
S(2r,x ;2s,y) } 
%
=K_{\Ai}^{(\tau_1,-\tau_2)}(\sg -\xi_1,\sg-\xi_2),
\end{aligned}
\label{E127'}\end{equation}
with uniform convergence for $\xi_1,\xi_2,\tau_1,\tau_2$ in bounded sets.
\end{lemma}

\proof  
At first, one checks that the scaling (\ref{2''}) and the definition (\ref{59}) of $M^{-\tau}_{\xi}$ imply
\be\begin{aligned}
\left(y-x+r-s\right)\log(-v_0)&=
\left((1-a^2)\theta (\tau_1-\tau_2)t^{2/3}+\rho (\xi_2-\xi_1) t^{1/3}\right)\log(-v_0)
\\
&=
M_{\xi_2}^{-\tau_2}(\tfrac t2)
-M_{\xi_1}^{-\tau_1}(\tfrac t2)
\\
&=M^{\tau_1}_{\sg-\xi_1+\beta}(\tfrac{t}{2})-M^{\tau_2}_{\sg-\xi_2+\beta}(\tfrac{t}{2}).
\end{aligned}\label{pref}\ee
Set, as usual, the scaling (\ref{2''}), with $n=2t$, in the third expression (\ref{FG1}) for $S(2r,x ;2s,y)$. However, we now set $b=\beta\rho t^{1/3}$ instead of $b=\beta\rho (2t)^{1/3}$; i.e., replace $2^{1/3}\beta$ by $\beta$. %
Then, one checks, using the limits (\ref{62'}),
$$
\begin{aligned}
&\lim_{t\to \infty}\rho t^{1/3}e^{ (M^{-\tau_2}_{ \xi_2}(t/2)- M^{-\tau_1}_{ \xi_1} (t/2)) } S(2r,x ;2s,y)
\\
&=\lim_{t\to \infty}\sum_{\beta\in\frac{{\mathbb N}}{\rho t^{1/3}}}\frac{\Dt b}{\rho t^{1/3}}\left( t^{1/3}\frac{Av_0}{v_0-a} e^{-M^{\tau_2}_{\sg -\xi_2+\beta}(t/2)}g^{(1)}_{-x_2+m+b}(t;s-t)\right)\\
&\hspace*{3cm}\left( t^{1/3}Av_0(v_0-a) e^{M^{\tau_1}_{\sg -\xi_1+\beta}(t/2)}g^{(2)}_{-x_1+m+b}(t,t-r)\right)
\\
&=\int_0^{\iy}\Ai^{(-\tau_2)}(\sg -\xi_2 +\beta)\Ai^{(\tau_1)}(\sg-\xi_1+\beta)d\beta =K_{\Ai}^{(\tau_1,-\tau_2)}(\sg-\xi_1,\sg-\xi_2),
\end{aligned}
$$
where we have used the exponential estimates (\ref{D88'}) to get dominated convergence in the above limit to the integral in the style of (\ref{D94'}). The convergence is uniform on bounded sets by uniform convergence on bounded sets in Lemma \ref{BF} and dominated convergence in the integral, concluding the proof of Lemma \ref{L6.4}.

\begin{lemma}\label{L6.5}
Upon using  the scaling of (\ref{2''}), with $\tau_1-\tau_2 >0$, the following limit holds for the function $\psi_{2(s-r)}(x,y) $ in (\ref{Idef}), %
$$
\begin{aligned}
\lim_{t\to \infty}\rho t^{1/3}(-v_0)^{y-x+r-s} (-1)^{x-y}\psi_{2(s-r)}(x,y) 
 =\frac{e^{-\frac{(\xi_1-\xi_2)^2}{4(\tau_1-\tau_2)}}}{\sqrt{4\pi (\tau_1-\tau_2)}} ,
\end{aligned}
$$
with uniform convergence for $\xi_1,\xi_2,\tau_1,\tau_2$ in bounded sets.
\end{lemma}

\begin{proof} Using the scaling (\ref{2''}) in (\ref{Idef}) and (\ref{58}),
$$
\begin{aligned}
I_t :=(-1)^{x-y}\psi_{2(s-r)}(x,y)&=\int_{\Gamma_{0,a}}\frac{dz}{2\pi iz}(-z)^{x-y} \left(\frac{1+az}{1-\frac{a}{z}}\right)^{s-r}  
\\
&=-\int_{\Gamma_{0 }}\frac{dz}{2\pi i}e^{(t^{2/3}f_0(z)+t^{1/3}f_1(z)+f_2(z))},
\end{aligned}
$$
where we replaced $\Gamma_{0,a}$ by $\Gamma_0$, since $s-r<0$, and with
$$
\begin{aligned}
f_0(z)=(\tau_2-\tau_1)G(z),\quad f_1(z)=(\xi_1-\xi_2)\rho\ln (-z),\quad f_2(z)=-\ln(-z).
\end{aligned}
$$
By (\ref{58}), $G(z)$ experiences a simple critical point at $v_0$, which suggests a saddle point analysis about $v_0$; therefore, set
$$
\begin{aligned}
Z=\sqrt{2(\tau_1-\tau_2)}~A(z-v_0)t^{1/3},
\end{aligned}
$$
and picking a path of steepest descent, 
one finds that
$$
\begin{aligned}
I_t&= \frac{e^{ t^{2/3}(\tau_2-\tau_1)  G(v_0)+t^{1/3}\rho(\xi_1-\xi_2)\ln (-v_0))}}{-Av_0\sqrt{2(\tau_1-\tau_2)}~t^{1/3}}
 \left(\int_{-i\infty}^{i\infty}\frac{dZ}{2\pi i}e^{\frac{1}{2}(Z^2-\frac{2(\xi_1-\xi_2)}{\sqrt{2(\tau_1-\tau_2)}}Z)}+{\cal O}(t^{-1/3})\right)
 \\
&=\frac{e^{   (M^{-\tau_1}_{ \xi_1}(\tfrac{t}{2})-M^{-\tau_2}_{ \xi_2}(\tfrac{t}{2}) )}}{\rho t^{1/3} } 
 \left(\frac{e^{-\frac{1}{2}(\frac{\xi_1-\xi_2}{\sqrt{2(\tau_1-\tau_2)}})^2}}{\sqrt{4\pi (\tau_1-\tau_2)}}+{\cal O} (t^{-1/3} )\right),
\end{aligned}
$$
with ${\cal O}(t^{-1/3})$ uniform for $\xi_1,\xi_2,\tau_1,\tau_2$ in bounded sets. The proof is completed upon exhibiting a path of steep descent through $v_0$ for the function $G(z)$, or equivalently $\hat G(z)$. 
 $$ \begin{aligned} \hat G(z)&:=\frac{G(z)}{\theta(1+a^2)}-\log(-a)
 \\&=\log(z+a^{-1})-\log (z-a)+\left(\frac{a^{-1}-a}{a^{-1}+a}\right)\log(-z)
 \end{aligned}
 $$
 and note that $\Re \hat G(z)$ at $(-a^{-1},0,a,\iy)$ is $(-\iy,-\iy,+\iy,+\iy)$. To construct a steepest ascent path, through $v_0$, start with a small line segment through $v_0$ making a right angle with the real axis and join it to a level curve in the upper half plane which it meets and follow the level curve until it hits the $x$-axis to the right of $a$; then reflect this entire path in the $x$-axis to complete the path. The construction of such level curves is described in the proof of Lemma \ref{L6.1}. This concludes the proof of Lemma \ref{L6.5}.
\end{proof}


\subsection{Limit of the ingredients of the integral representation of $\widetilde\BK^{\rm ext}_{n,m}$}
\label{Kintrepr}

Recall the scaling (\ref{2'}) and (\ref{3'}), in particular the one of $z$ and $w$. Also remember the following expressions from (\ref{57}), (\ref{Qell}) and also define the kernel ${\cal K}^{(2)}_{t}(\kappa,\lambda)$ :
$$ {\cal Q}(\kappa)  =
 \left[(\Id - \chi_{\tilde \sigma}K_{\Ai}\chi_{\tilde \sigma})^{-1}
\chi_{\tilde \sigma}  \Ai\right]
(\kappa)
$$
 $$
Q^{(i)}_{ \ell}=[(\un -\raisebox{1mm}{$\chi$}{}_{2m+1} K^{(i)}(0)^{\top}\raisebox{1mm}{$\chi$}{}_{2m+1} )^{-1}  \raisebox{1mm}{$\chi$}{}_{2m+1} g^{(i)}](\ell),
$$
$$
{\cal K}^{(2)}_{t}(\kappa,\lambda):=(2t)^{1/3}e^{-M_{\kappa}(t)}K_{k\ell}^{(2)}(0)e^{M_{\lambda}(t)}
\Bigr|_{\tiny scaling ~(\ref{3'})}.$$
One shows the following Lemma:
\begin{lemma}\label{L6.6'} The following limits hold:
\be\lim_{t\to \infty} (2t)^{1/3} \frac{A}{v_0-a}e^{- M_{\kappa }(t)}
Q^{(1)}_{ k} 
=\lim_{t\to \infty} (2t)^{1/3}  A  (v_0-a) v_0^2e^{ M_{\kappa }(t)}
Q^{(2)}_{ k} 
=-{\cal Q}(\kappa)
\label{Q}\ee
\be\begin{aligned}
\lim_{t\to \infty}(v_0-a)v_0e^{ M_{\kappa}(t)}\bar h^{(1)}_k (z^{-1})&=
-\int_0^{\infty} e^{-2^{1/3}\zeta \beta} \Ai(\kappa+\beta)d\beta
,
\\
\lim_{t\to \infty} {((v_0-a)v_0)^{-1}}{e^{- M_{\kappa}(t)}}\bar h^{(2)}_k (z^{})
&=
-\int_0^{\infty} e^{2^{1/3}\zeta \beta} \Ai(\kappa+\beta)d\beta
,\end{aligned}\label{69}\ee
\be \begin{aligned}\lim_{t\to \infty}{(v_0-a)^{-1}}e^{- M_{\tilde \sigma}(t)} T^{(1)} (z^{-1})  
  &=
 -e^{ -\tilde \sigma 2^{1/3}\zeta}
 \int_{\tilde \sigma}^{\infty}{\cal Q}(\kappa)e^{\kappa 2^{1/3}\zeta} d\kappa= -e^{ -2 \sigma  \zeta}\hat {\cal Q}(\zeta).
\\
 {\lim_{t\to \infty} {(v_0-a)}e^{  M_{\tilde \sigma}(t)} T^{(2)} (z^{ })}
&= -
 e^{ - \tilde \sigma 2^{1/3}\zeta}
 \int_{\tilde \sigma}^{\infty}{\cal Q}(\kappa)e^{-\kappa 2^{1/3}\zeta} d\kappa= -e^{ 2 \sigma  \zeta}\hat {\cal Q}(-\zeta).
\end{aligned}
\label{71}\ee
\be
\lim_{t\to \infty}  S^{(1)}(z^{-1}) =\hat {\cal P}(\zeta)
,~~~\lim_{t\to \infty}  S^{(2)}(z^{}) =
\hat {\cal P}(-\zeta)
\label{73}\ee
where
$$
\hat {\cal P}(\zeta)=- \int_{\tilde \sigma}^{\infty}{\cal Q}(\kappa)
 d\kappa
\int_0^{\infty} e^{-2^{1/3}\zeta \beta} \Ai(\kappa+\beta)d\beta ,
$$
with uniform convergence for $\kappa$ and $\zeta$ in bounded sets.
\end{lemma}

\begin{proof}
\noindent{\bf Limits of  $Q_{k}^{(i)}$ }:~
The proof of (\ref{Q}) proceeds along similar lines as in \cite{AFvM12}, namely first one notices that the operator norm of the Airy kernel $|\! | \chi_{\tilde \sigma} K_\Ai(\lb,\kappa) \chi_{\tilde \sigma} |\! |< 1$, and since $\chi_{\tilde \sigma}{\cal K}^{(2)}_{t}(\kappa,\lambda)\chi_{\tilde \sigma}$ tends pointwise to $\chi_{\tilde \sigma} K_\Ai(\lb,\kappa) \chi_{\tilde \sigma}$ (with exponential domination as in (\ref{D94'})), one also has $|\! | \chi_{\tilde \sigma}{\cal K}^{(2)}_{t}(\kappa,\lambda)\chi_{\tilde \sigma} |\! |<1$ for large enough $t$ from which it follows that $|\! | (\Id -\chi_{\tilde \sigma}{\cal K}^{(2)}_{t}(\kappa,\lambda)\chi_{\tilde \sigma})^{-1}|\! | <C$ uniformly for $t$ large enough. Because of the estimate (\ref{62u}) involving $g_\ell({1})(2t)$,  one checks that  
$${\cal Q}_t(\kappa):=\left(\Id -\chi_{\tilde \sigma}{\cal K}^{(2)}_{t}(\kappa,\lambda)\chi_{\tilde \sigma} \right)^{-1}   \frac{A (2t)^{1/3}}{a-v_0}
  e^{- M_{\lb }(t)}g_{\frac{4t}{a+a^{-1}}+\lambda  \rho (2t)^{1/3}+1}^{(1)}(2t)
$$
satisfies 
$$
|{\cal Q}_t(\kappa)|\leq c e^{-\theta \kappa}
$$ for any $\theta>0$ and some constant $c>0$.


We have, using (\ref{62K}), (\ref{62a}) and the above discussion,%
\be\begin{aligned}
\lefteqn{\lim_{t\to \infty} (2t)^{1/3} \frac{A}{v_0-a}e^{- M_{\kappa }(t)}
Q^{(1)}_{ k}=-\lim_{t\to \infty} {\cal Q}_t(\kappa)}\\
&=\lim_{t\to \infty} 
\Bigl(\Id -  e^{- M_{\kappa }(t)}\chi_{_{2m+1}}K^{(2)}(0)_{k,\ell}
\chi_{_{2m+1}}
 e^{ M_{\lb }(t)}\Bigr)^{-1}   \frac{A (2t)^{1/3}}{v_0-a}
  e^{- M_{\lb }(t)}g_\ell^{(1)}(2t)
\\
&=
\left[(\Id - \chi_{\tilde \sigma}K_{\Ai}\chi_{\tilde \sigma})^{-1}
\chi_{\tilde \sigma}  \Ai\right]
(\kappa)=-{\cal Q}(\kappa),
\end{aligned}
 \label{66}\ee
uniformly $\kappa$ in a bounded set. Similarly
\be\begin{aligned}
 {\lim_{t\to \infty} (2t)^{1/3}  {A}{(v_0-a)}v_0^2e^{ M_{\kappa }(t)}
Q^{(2)}_{ k}} 
=-{\cal Q}(\kappa).
\end{aligned}
\label{67}\ee


\noindent{\bf Limits of  $h_k^{(i)}$ }:~
With the scaling $b=\rho\beta (2t)^{1/3}$, using the infinitesimal element $\Dt b=1$ as in (\ref{63}), and using the scaling (\ref{3'}) for $z$ and $k$, one checks from (\ref{25h}), (\ref{62a}) and (\ref{D87'}) that 
\be\begin{aligned}\lefteqn{\lim_{t\to \infty}(v_0-a)v_0e^{ M_{\kappa}(t)}\bar h^{(1)}_k (z^{-1})}\\
&=
-\lim_{t\to \infty}(v_0-a)v_0e^{ M_{\kappa}(t)}\sum_{b=0}^{\infty} (-z)^{b} g_{k+b}^{(2)}\Dt b\\
&=-\lim_{t\to \infty} (v_0-a)v_0\sum_{b=0}^{\infty} \Bigl(1-\frac {2^{1/3}\zeta }{ \rho(2t)^{1/3}}\Bigr)^{\rho\beta (2t)^{1/3}}  e^{ \beta \rho\log(-v_0) (2t)^{1/3}}e^{ M_\kappa(t)}\\
&\hspace*{9cm}
g_{k+b}^{(2)}  (2t)^{1/3}\rho d\beta
\\&
= \lim_{t\to \infty}\int_0^{\infty} e^{\rho\beta (2t)^{1/3}\log (1-\frac {2^{1/3}\zeta }{ \rho(2t)^{1/3}})}
[e^{ M_{\kappa+\beta}(t)} g^{(2)}_{k+b}
 (2t)^{1/3}v_0^2(v_0-a)A]d\beta 
\\
&=-\int_0^{\infty} e^{-2^{1/3}\zeta \beta} \Ai(\kappa+\beta)d\beta
,\end{aligned}
\label{68}\ee
using dominated convergence from (\ref{D87'}). Similarly,
\be
\lim_{t\to \infty}\frac{e^{- M_{\kappa}(t)}}{(v_0-a)v_0}\bar h^{(2)}_k (z^{})
=-
\int_0^{\infty} e^{2^{1/3}\zeta \beta} \Ai(\kappa+\beta)d\beta
.\label{69}\ee



\noindent{\bf Limits of  $T^{(i)}$ and $S^{(i)}$}:~
Setting $\Dt k=d\kappa ~ (2t)^{1/3}\rho $, with $\Dt k=1$, one finds, using the scaling (\ref{2'}) and (\ref{3'}) for $z,k,m$, with $k-2m-1=(\kappa-\tilde \sigma)\rho (2t)^{1/3}$,
\be\begin{aligned}
\lefteqn{\lim_{t\to \infty}\frac 1{v_0-a}e^{- M_{\tilde \sigma}(t)} T^{(1)} (z^{-1})}
\\&=\lim_{t\to \infty}\frac 1{v_0-a}
e^{- M_{\tilde \sigma}(t)}\sum_{k\geq 2m+1}\frac {Q^{(1)}_{k}}{(-z)^{k-2m}}   \Dt k
\\&=\lim_{t\to \infty}\frac 1{v_0-a}
e^{- M_{\tilde \sigma}(t)}
\sum_{k\geq 2m+1} {Q^{(1)}_{ k}} 
\left(1-\frac{\zeta}{ \rho t^{1/3}}\right)^{-k+2m}(-v_0)^{-k+2m}\\
&~~~~~~~~~~~~~~~~~~~~~~~~~~~~~~~~~~~~~~[\rho(2t)^{1/3}  d\kappa ]
\\&=\lim_{t\to \infty}\frac A{v_0-a}
\sum_{k\geq 2m+1} {Q^{(1)}_{ k}} 
\left(1-\frac{\zeta}{ \rho t^{1/3} }\right)^{-k+2m}
e^{- M_{\tilde \sigma}(t)}(-v_0)^{-k+2m+1}\\
&~~~~~~~~~~~~~~~~~~~~~~~~~~~~~~~~~~~~~~~~~~~~~~~~~~~~~~~~~~~~~~~~~~~~[(2t)^{1/3}  d\kappa ]
\\&=\lim_{t\to \infty}
\sum_{k\geq 2m+1}
\frac {A(2t)^{1/3}}{v_0-a} e^{- M_{\kappa}(t)}{Q^{(1)}_{ k}} 
 d\kappa  e^{(\kappa-\tilde \sigma)2^{1/3}\zeta}
 \\&=-
 e^{ -\tilde \sigma 2^{1/3}\zeta}
 \int_{\tilde \sigma}^{\infty}{\cal Q}(\kappa)e^{\kappa 2^{1/3}\zeta} d\kappa= -e^{ -2 \sigma  \zeta}\hat {\cal Q}(\zeta),\label{70}
 \end{aligned}
\ee
using (\ref{Q}), the estimate for ${\cal Q}_t(\kappa)$ and dominated convergence. In the same way
\be\begin{aligned}
 {\lim_{t\to \infty} {(v_0-a)}e^{  M_{\tilde \sigma}(t)} T^{(2)} (z^{ })}
=
 -e^{  \tilde \sigma 2^{1/3}\zeta}
 \int_{\tilde \sigma}^{\infty}{\cal Q}(\kappa)e^{-\kappa 2^{1/3}\zeta} d\kappa=- e^{ 2 \sigma  \zeta}\hat {\cal Q}(-\zeta).
\end{aligned}
\label{71}\ee
Then for the $S^{(i)}$'s, one checks, using  (\ref{66}) and (\ref{68}),  the following limit:
\be\begin{aligned}
{\lim_{t\to \infty}  S^{(1)}(z^{-1}) }
 &=-\!\!\sum_{k\geq 2m+1}\!\!\left(\frac{A (2t)^{1/3}}{v_0-a}e^{- M_{\kappa}(t)}Q^{(1)}_{ k}\right)
\left( (v_0\!-\!a)v_0 e^{ M_{\kappa}(t)}\bar h^{(1)}_{k}(z^{-1})\right)
 d\kappa   
 \\
&=- \int_{\tilde \sigma}^{\infty}{\cal Q}(\kappa)
 d\kappa
\int_0^{\infty} e^{-2^{1/3}\zeta \beta} \Ai(\kappa+\beta)d\beta=
\hat {\cal P}(\zeta)
\end{aligned}
\label{72}\ee
and
\be\begin{aligned}
 {\lim_{t\to \infty}  S^{(2)}(z^{}) }
&=- \int_{\tilde \sigma}^{\infty}{\cal Q}(\kappa)
 d\kappa
\int_0^{\infty} e^{ 2^{1/3}\zeta \beta} \Ai(\kappa+\beta)d\beta=\hat {\cal P}(-\zeta).
\end{aligned}
\label{73}\ee
To justify the limit, one uses the exponential estimate for ${\cal Q}_t(\kappa)$; also one uses an exponential estimate obtained from (\ref{68}), (\ref{69}) and (\ref{D87'}) in order to apply dominated convergence in the above integral. This ends the proof of Lemma \ref{L6.6'}.
\end{proof}


\section{Scaling limits of the kernel $\widetilde\BK^{\rm ext}_{n,m}$}

\subsection{Scaling limit as a perturbation of the Airy kernel}

In this section we prove the first part of the main Theorem \ref{th:main}, namely formula (\ref{E127}).

\begin{theorem}\label{Th7.1}  The following scaling limit holds under (\ref{2''}):
\begin{equation}\begin{aligned}
\lefteqn{\hspace*{-1cm}\lim_{t\to \infty} 
(-v_0)^{y-x+r-s} (-1)^{x-y} ~\widetilde \BK^{\rm ext}_{n,m}(2r,x ;2s,y) \rho t^{1/3} }  
\\
= & -\Id_{\tau_1>\tau_2}
\frac{e^{-\frac{(\xi_1-\xi_2)^2}{4(\tau_1-\tau_2)}}}{\sqrt{4\pi (\tau_1-\tau_2)}}+ K_{\Ai}^{(\tau_1,-\tau_2)}(\sg -\xi_1,\sg-\xi_2)
\\
&+
 2^{1/3} \int_{\tilde\sg}^{\iy}\left((\Id- \raisebox{1mm}{$\chi$}{}_{\tilde\sg}K_{\Ai}\raisebox{1mm}{$\chi$}{}_{\tilde\sg})^{-1}
 \AR_{\xi_1-\sg}^{ \tau_1}\right)(\lb)
\AR_{\xi_2-\sg}^{ -\tau_2 }(\lb)d\lb 
    ,
\end{aligned}
\label{E127'}\end{equation}
where the first line on the right hand is the extended Airy kernel, as in (\ref{AiryP}), except for some conjugation; see (\ref{AiryP1}).
\end{theorem}
\begin{proof}  At first note that the $y$-scaling implies $dy=\rho t^{1/3}d\xi_2$. The kernel $(-1)^{x-y} \widetilde \BK^{\rm ext}_{n,m}(2r,x ;2s,y)  $, as in (\ref{K1}), is a sum of three parts. Multiplying each one with $(-v_0)^{y-x+r-s}\rho t^{1/3}$, the limit of the first part is taken care of by Lemma \ref{L6.5} and the limit of the second part by Lemma \ref{L6.4}. So, it remains to show the following uniform limit (uniform on $\xi_1,\xi_2,\tau_1,\tau_2$ in bounded sets), taking into account identity (\ref{pref}) and the symmetry of the Airy kernel:
\begin{equation}
\begin{aligned}
  \lim_{t\to \infty}\rho t^{1/3}&e^{ (M^{-\tau_2}_{ \xi_2}(\tfrac{t}{2})- M^{-\tau_1}_{ \xi_1} (\tfrac{t}{2})) } \Big\la (\Id-K_{2m+1}^{(1)}(0))^{-1}a_{-y,s},b_{-x,r}\Big\ra _{\ell^2 {(2m+1,...)}}  \\
&=2^{1/3}\int_{\tilde\sg}^{\iy}\left(\Id- \raisebox{1mm}{$\chi$}{}_{\tilde\sg}K_{\Ai}\raisebox{1mm}{$\chi$}{}_{\tilde\sg})^{-1}\AR_{\xi_2-\sg}^{ -\tau_2 }\right)(\lb)
\AR_{\xi_1-\sg}^{  \tau_1 }(\lb)d\lb,
\end{aligned}
\label{E128}\end{equation}
remembering the notation (\ref{three}) for the restriction of the kernel $K(0)$ to $[2m+1,\infty)$. 
%
%
%
Using the identity
$$
\begin{aligned}
 \left(M^{-\tau_2}_{\xi_2} (\tfrac{t}{2})-\!  M^{-\tau_1}_{\xi_1} (\tfrac{t}{2})\right)+\left(M_{\kappa}(t) \!-M_{\lb}(t)\right) =M^{-\tau_2}_{ \xi_2-\sg+2^{1/3}\kappa} (\tfrac{t}{2})-M^{-\tau_1}_{   \xi_1-\sg +2^{1/3}\lb} (\tfrac{t}{2})  ,
\end{aligned}
$$
together with Lemma \ref{L6.3}, one checks the following limit:
%
\begin{equation}
\begin{aligned}
\lefteqn{\lim_{t\to \infty}\rho t^{1/3}e^{ (M^{-\tau_2}_{ \xi_2}(\tfrac{t}{2})- M^{-\tau_1}_{ \xi_1} (\tfrac{t}{2})) } \Big\la (\Id -K_{2m+1}^{(1)}(0))^{-1}a_{-y,s},b_{-x,r}\Big\ra _{\ell^2 {(2m+1,...)}}}
\\
&=\lim_{t\rg\iy} \frac{2^{1/3}}{\rho (2t)^{1/3}} \sum_{\ell\geq 2m+1} \Delta \ell 
\\
&\left[(\Id-e^{M_{\lb}(t)}K_{2m+1}^{(1)}(0)_{\ell k}e^{-M_{\kappa}(t)} )^{-1}
((-1)^m( t^{1/3}\rho v_0 e^{M^{-\tau_2}_{ \xi_2-\sg +2^{1/3}\kappa} (\tfrac{t}{2}) }a_{-y,s}(k))\right]
\\
&\hspace*{4cm}\times \left[(-1)^m(-t^{1/3}A e^{-M^{-\tau_1}_{\xi_1-\sg +2^{1/3}\lb} (\tfrac{t}{2}) }b_{-x,r}(\ell))\right] 
\\
&=2^{1/3}\int_{\tilde\sg}d\lb \left( (\Id-\raisebox{1mm}{$\chi$}{}_{\tilde\sg}K_{\Ai}\raisebox{1mm}{$\chi$}{}_{\tilde\sg})^{-1}\AR_{\xi_2-\sg}^{ -\tau_2 }\right)(\lb)
\AR_{\xi_1-\sg}^{ \tau_1 }(\lb),
\end{aligned}
\label{E129}\end{equation}
where in the last equality, one uses the fact that $\tfrac{\Dt\ell}{ \rho (2t)^{1/3} }=d\lambda $, with $\Dt \ell=1$, and the fact, already pointed out just after (\ref{3'}), that $\ell \geq 2m+1$ implies $\lambda \geq \tilde \sigma=2^{2/3}\sigma$.

  To justify the above, one needs uniformity of convergence. From the estimates 
  (\ref{D96'}), and using arguments, as in the derivation of the exponential bound for  $|{\cal Q}_t(\kappa)|$ in the proof of Lemma \ref{L6.6'}, we have that
  $$
  \left\{\mbox{product of the expressions in the first bracket in (\ref{E129})}\right\}\leq  c_0e^{-c_1\lb}
  $$
  for $\xi_i,\tau_i$ in a bounded set, and for $\lb\geq\tilde\sg$, with $c_0,c_1$ positive time-independent constants, while from (\ref{D96'}) one has a similar estimate for the expression in the second bracket in (\ref{E129}). Thus by dominated convergence we have the convergence to the integral in (\ref{E129}), with uniform convergence for the parameters $\xi_i,\tau_i$ in bounded sets, concluding the proof of Theorem \ref{Th7.1}, except for the identification with the Airy process kernel, which follows from (\ref{E6}).
\end{proof}

  

\subsection{Scaling limit as a sum of double integrals}\label{sect7.2}
Here we indicate the proof of the second half of Theorem \ref{th:main}, namely formula (\ref{main1}) for $r=s=n/2$. To do so, we take the scaling limit of $\widetilde \BK_{n,m}(x,y)$, given by formula (\ref{43a}) in Proposition \ref{prop4.1} for $r=s=n/2$; or what is the same, formula (\ref{54}) containing $\LR$ given by (\ref{calL}). It proceeds along similar lines as in \cite{AFvM12}. We then merely sketch how to pass to the limit for the extended kernel $\widetilde \BK^{\rm ext}_{n,m}(2r,x;2s,y)$. 

From (\ref{60}), one has, using the scaling (\ref{2'}), (\ref{3'}) and (\ref{2''}) for  for $n,z$,
\be \begin{aligned}
 \tfrac n2 \Phi(z)&=tF(z)+\tfrac {\tilde \sigma}2 \rho (2t)^{1/3}\log (-z)
\\
&= \tfrac 12 M_{\tilde \sigma}(t) +\tfrac 13 \zeta^3-  {   \sigma} \zeta+{\cal O}(t^{-1/3}).\end{aligned}
\label{74}\ee
With the scaling (\ref{3'}) and (\ref{2''}) for $w,z,x,y$, with $\tau_i=0$, one finds the following limits: 
$$
\lim_{t\to \infty}  (-v_0)^{x-y }\left(\frac{(-w)^{y-1}}{(-z)^x}
+\frac{(-z)^y}{(-w)^{x+1}}\frac{1-\tfrac az}{1-\tfrac aw}\right)=\frac{1}{-v_0}\left(\frac{e^{ \zeta \xi}}{e^{ \omega \eta}} +\frac{e^{- \zeta \eta}}{e^{ -\omega \xi}}\right)
$$
\be
\lim_{t\to \infty}  (-v_0)^{x-y }\left(\frac{(-z)^{y-1}}{(-w)^x}
+\frac{(-w)^y}{(-z)^{x+1}}\frac{1-\tfrac aw}{1-\tfrac az}\right)=\frac{1}{-v_0}\left(\frac{e^{ \omega \xi}}{e^{ \zeta \eta}} +\frac{e^{- \omega \eta}}{e^{ -\zeta \xi}}\right)
\label{75}\ee
and
\be
\lim_{t\to \infty} 
\frac{dz dw}{z-w} dy
=-v_0 \frac {d\zeta d\omega}{\zeta-\omega} d\eta.
\label{76}\ee
{\bf The ${\cal L}(x,y)$-part of the kernel} (\ref{54}),
\be (-1)^{x-y}\frac{H_{2m+2}(0)}{H_{2m+1}(0)}
\widetilde\BK_{n,m}(-x,-y)={\cal L}(x,y)+C(0;x-y),
\label{167}\ee
consists of the four double integrals appearing in (\ref{calL}), with the $E_i$ given in (\ref{Ei}). Each of them will be examined separately, using the same saddle point arguments as given around Figure 18 in the proof of Lemma \ref{L6.1}. 
\medbreak

\noindent{\bf 1st double integral:} From (\ref{74}) and (\ref{73}), it follows that
$$
\lim_{t\to \infty} 
 {e^{\tfrac n2( \Phi(z)-\Phi(w))}} =
\frac{e^{\tfrac 13 \zeta^3-\sigma \zeta}}
{e^{\tfrac 13 \omega^3-\sigma \omega}}
$$
$$\lim_{t\to \infty} S^{(1)}(z^{-1})=\hat {\cal P}(\zeta)\mbox{   and   }
\lim_{t\to \infty} S^{(2)}(w)=\hat {\cal P}(-\omega),
$$
and thus
 \be
\lim_{t\to \infty} 
 {e^{\tfrac n2( \Phi(z)-\Phi(w))}}
 (1-S^{(1)}(z^{-1}))(1-S^{(2)}(w))=\frac{e^{\tfrac 13 \zeta^3-\sigma \zeta}}
{e^{\tfrac 13 \omega^3-\sigma \omega}}
(1-\hat {\cal P}(\zeta))(1-\hat {\cal P}(-\omega)). 
\label{77}\ee

\noindent {\bf 2nd double integral:} Since from (\ref{74}), (\ref{72}) and (\ref{73}), 
$$
\lim_{t\to \infty} 
e^{- M_{\tilde \sigma}(t)} e^{\tfrac n2 (\Phi(z)+\Phi(w))}  =
\frac{e^{\tfrac 13 \zeta^3-\sigma \zeta}}
{e^{-\tfrac 13 \omega^3+\sigma \omega}}
,~~ {\lim_{t\to \infty} {(w-a)}e^{ M_{\tilde \sigma}(t)} T^{(2)} (w^{ })}
= -e^{ 2 \sigma  \omega}\hat {\cal Q}(-\omega),
$$
one has
\be
\lim_{t\to \infty} {(w-a)}e^{\tfrac n2 (\Phi(z)+\Phi(w))} (1-S^{(1)}(z^{-1}))T^{(2)} (w^{ })
= -  \frac{e^{\tfrac 13 \zeta^3-\sigma \zeta}}
 {e^{-\tfrac 13 \omega^3-\sigma \omega}}
(1-\hat {\cal P}(\zeta)) \hat {\cal Q}(-\omega).
\label{78}\ee

\noindent{\bf 3rd double integral:} From (\ref{74}), (\ref{71}) and (\ref{73})), one obtains:
$$
\lim_{t\to \infty} 
e^{ M_{\tilde \sigma}(t)} e^{-\tfrac n2 (\Phi(z)+\Phi(w))}  =
\frac{e^{-\tfrac 13 \zeta^3+\sigma \zeta}}
{e^{\tfrac 13 \omega^3-\sigma \omega}}
$$
$$\begin{aligned}
 {\lim_{t\to \infty}\frac 1{z-a}e^{-M_{\tilde \sigma}(t)}} T^{(1)} (z^{-1}) =- e^{ -2 \sigma  \zeta}\hat {\cal Q}(\zeta)\end{aligned}
.$$
Therefore
\be
\lim_{t\to \infty} \frac{1}{z-a}
e^{-\tfrac n2 (\Phi(z)+\Phi(w))}  T^{(1)} (z^{-1})(1-S^{(2)}(w)) 
 =
-\frac{e^{-\tfrac 13 \zeta^3-\sigma \zeta}}
{e^{\tfrac 13 \omega^3-\sigma \omega}}
\hat {\cal Q}(\zeta)(1-\hat {\cal P}(-\omega))
.\label{79}\ee
{\bf 4th double integral:} Since
$$
\lim_{t\to \infty} 
 {e^{-\tfrac n2( \Phi(w)-\Phi(z))}} =
\frac{e^{ \tfrac 13 \zeta^3-\sigma \zeta}}
{e^{\tfrac 13 \omega^3-\sigma \omega}}
,$$
one has from (\ref{71})
\be \begin{aligned}
 {\lim_{t\to \infty}\frac {a-z}{a-w} {e^{-\tfrac n2( \Phi(w)-\Phi(z))}} T^{(1)} (w^{-1})} 
 T^{(2)} (z^{ })
 =  
  \frac{e^{ \tfrac 13 \zeta^3+\sigma \zeta}}
{e^{\tfrac 13 \omega^3+\sigma \omega}}\hat {\cal Q}(-\zeta)
 \hat {\cal Q}(\omega).
\end{aligned}
\label{80}\ee

\bigbreak


\noindent
{\bf Assembling the four double integrals,} one finds the following limit, upon using the contours in \cite{AFvM12}, in particular Figures 4 and 5:
%
{\footnotesize\be \begin{aligned}\label{81}
\lefteqn{\lim_{t\to \infty}(-v_0)^{x-y}\rho t^{1/3}~{\cal L}(x,y)}\\
&=\frac{1}{(2\pi \I)^2}\left\{\begin{aligned}
 &\int_{\delta +\I\BR}\hspace{-1em}d\zeta\int_{-\delta +\I\BR}\hspace{-1em}d\omega\,
\frac{e^{\frac{\zeta^3}3-\sigma \zeta}}{e^{\frac{\omega^3}3-\sigma  \omega}}
 \left(\frac{e^{ \zeta \xi}}{e^{ \omega \eta}} +\frac{e^{- \zeta \eta}}{e^{ -\omega \xi}}\right)
  \frac{(1-\hat{\cal P}(\zeta))(1-\hat{\cal P}(-\omega))}{\zeta-\omega}\\
&- \int_{2\delta +\I\BR} \hspace{-1em}d\zeta \int_{\delta +\I\BR}\hspace{-1em}d\omega\,
\frac{e^{\frac{\zeta^3}3-\sigma  \zeta}}{e^{-\frac{\omega^3}3-\sigma  \omega}}
 \left(\frac{e^{ \zeta \xi}}{e^{ \omega \eta}} +\frac{e^{- \zeta \eta}}{e^{ -\omega \xi}}\right)
\frac{(1-\hat{\cal P}(\zeta)) \hat{\cal Q}(-\omega)}{\zeta-\omega} 
\\&- \int_{-\delta +\I\BR}\hspace{-1em}d\zeta\int_{-2\delta +\I\BR}\hspace{-1em}d\omega\,
\frac{e^{-\frac{\zeta^3}3-\sigma \zeta}}{e^{\frac{\omega^3}3-\sigma \omega}} 
 \left(\frac{e^{ \zeta \xi}}{e^{ \omega \eta}} +\frac{e^{- \zeta \eta}}{e^{ -\omega \xi}}\right)
\frac{(1-\hat{\cal P}(-\omega)) \hat{\cal Q}(\zeta)}{\zeta-\omega}\\
& -\int_{ \delta +\I\BR}\hspace{-1em}d\zeta \int_{-\delta +\I\BR}\hspace{-1em}d\omega\,
\frac{e^{ \frac{\zeta^3}3+\sigma \zeta}}{e^{ \frac{\omega^3}3+\sigma \omega}}
\left(\frac{e^{ \omega \xi}}{e^{ \zeta \eta}} +\frac{e^{- \omega \eta}}{e^{ -\zeta \xi}}\right)
\frac{ \hat{\cal Q}(-\zeta)  \hat{\cal Q}( \omega)}{\zeta-\omega}.
\end{aligned}\right\}\end{aligned}
\ee}
This leads to the double integral representation (\ref{main1}) in Theorem \ref{th:main}, for the non-extended kernel $r=s=n/2$, 
upon interchanging $\omega$ and $\zeta$ in the last integral.   

 \bigbreak


 In order to take the scaling limit of the {\bf $C(0;x-y)$-part of the kernel} (\ref{167}), one needs to express $C(0;x)$ in a more convenient form, as is given by formula (\ref{84}) in Proposition \ref{prop8.1} of the Appendix. Using this form, one now proves the following:

 \begin{proposition}\label{Prop7.3}
 Under the scaling (\ref{2''}) for $x,y$ with $\tau_1=\tau_2=0$ and $\xi_1\mapsto \xi$ and $\eta_1\mapsto \eta$, one finds
 $$\begin{aligned}
 \lefteqn{\lim_{t\to \infty}(-v_0)^{x-y}\rho t^{1/3}~ C(0;x-y)} \\
 &=2^{-1/3} 
 \\
 &\left\{
 \begin{aligned}
 &\int^{\iy}_{\tilde\sigma}d\kappa~{\cal Q}(\kappa)\left(\Ai (\kappa+(\xi -\eta)2^{-1/3})+\Ai (\kappa-(\xi -\eta)2^{-1/3})\right) \\
&+ \int^{\iy}_{\tilde\sigma}d\kappa     \int^{\iy}_{\tilde\sigma}d\lambda~{\cal Q}(\kappa){\cal Q}(\lambda)
\\
&~~~~\times ~\left(
K_{\mbox{\tiny \rm Ai}}(\kappa,\lambda -(\xi -\eta)2^{-1/3}) 
 +K_{\mbox{\tiny  \rm Ai}}(\kappa +(\xi -\eta)2^{-1/3},\lambda)\right)      
 \end{aligned}\right\}
 \end{aligned}$$
 \end{proposition}

\proof  Setting $1=\Delta y =\rho t^{1/3}d\eta$ and  $1=\Delta k =\rho t^{1/3}d\kappa$, one checks in the first expression (\ref{84}), using (\ref{62a}) and (\ref{Q}), that:
\be \begin{aligned} 
\lefteqn{\lim_{t\rg\iy}(-v_0)^{x-y}\Dt y\sum_{k\geq 2m+1}Q_k^{(1)}g^{(2)}_{k+x-y}  \Dt k} 
\\
&=\lim_{t\rg\iy}2^{-1/3}e^{\rho\log(-v_0)(\xi -\eta)t^{1/3}}\rho (2t)^{1/3}d\eta\sum_{k\geq 2m+1}Q_k^{(1)}g^{(2)}_{k+x-y}\rho (2t)^{1/3}d\kappa\\
&=\lim_{t\rg\iy}2^{-1/3}d\eta~e^{\rho\log(-v_0)(\xi -\eta)t^{1/3}} e^{ M_k(t)}e^{- M_{\kappa+(\xi-\eta)2^{-1/3}}(t)}
\\
& \sum_{k\geq 2m+1}\left((2t)^{1/3}\frac{A}{v_0-a}e^{- M_{\kappa}(t)}Q_k^{(1)}\right)\left((2t)^{1/3}A(v_0-a)v^2_0 e^{M_{\kappa+(\xi-\eta)2^{-1/3}}}g^{(2)}_{k+x-y}\right)d\kappa
\\
&=2^{-1/3}d\eta\int^{\iy}_{\tilde\sigma}d\kappa~{\cal Q}(\kappa)\mbox{Ai}(\kappa+(\xi -\eta)2^{-1/3}),
\end{aligned} \ee
using dominated convergence in the integral. Similarly
$$
\lim_{t\rg\iy}(-v_0)^{x-y}  \Delta y\sum_{k\geq 2m+1}{  Q}_k^{(2)}g^{(1)}_{k-x+y}=2^{-1/3}d\eta  \int^{\iy}_{\tilde\sigma}d\kappa~{\cal Q}(\kappa)\mbox{Ai}(\kappa-(\xi -\eta)2^{-1/3}) 
$$
The limit of the last two expressions in (\ref{84}) are the following:
$$ \begin{aligned} 
\lefteqn{\lim_{t\rg\iy}(-v_0)^{x-y}\Delta y   \sum_{k,\ell\geq 2m+1}Q_k^{(1)}Q^{(2)}_{\ell}K^{(1)}_{k,\ell -x+y}(0)}
\\
&=\lim_{t\rg\iy}(-v_0)^{x-y}\Delta y  \sum_{k\geq 2m+1} \sum_{\ell\geq 2m+1}Q^{(1)}_kQ^{(2)}_{\ell}K^{(1)}_{k,\ell -x+y}(0)\Delta k~\Delta\ell 
\\
& =2^{-1/3}d  \eta   \int^{\iy}_{\tilde\sigma}d\kappa     \int^{\iy}_{\tilde\sigma}d\lambda~{\cal Q}(\kappa){\cal Q}(\lambda)K_{\mbox{\tiny Ai}}(\kappa,\lambda -(\xi -\eta)2^{-1/3})d\kappa~d\lambda
\end{aligned} $$
and
$$ \begin{aligned} 
\lefteqn{\lim_{t\rg\iy}(-v_0)^{x-y}\Delta y   \sum_{k,\ell\geq 2m+1}Q_k^{(1)}Q^{(2)}_{\ell}K^{(1)}_{k+x-y,\ell}(0)} 
\\
& =2^{-1/3}d  \eta   \int^{\iy}_{\tilde\sigma}d\kappa     \int^{\iy}_{\tilde\sigma}d\lambda~{\cal Q}(\kappa){\cal Q}(\lambda)K_{\mbox{\tiny Ai}}(\kappa +(\xi -\eta)2^{-1/3},\lambda),
\end{aligned} $$
thus establishing Proposition \ref{Prop7.3}.\qed

\section{The tacnode kernel and non-colliding Brownian motions}

In this section, we prove expression (\ref{E127K}) for the kernel ${\mathbb K}^{\rm tac}$, as given in (\ref{E127L}), and we also show that ${\mathbb K}^{\rm tac}={\mathbb K}^{\rm tac}_{\rm br}$. To do so, introduce the function
$$
S^\tau_{\xi}(\kappa)= \Ai^{(\tau)} ( \xi-\sigma+2^{1/3} \kappa)
$$
and the operator $T$ on $L^2(\tilde\sg,\iy)$, with kernel
$$
T(\kappa,\beta)=\Ai (\kappa+\beta-\tilde\sigma).
$$
  Note that  
$K_{\Ai}$, as an operator on $L^2(\tilde\sg,\iy)$, is given by $$
\begin{aligned}
\int^{\iy}_{\tilde\sg}\Ai(\kappa+\lb-\tilde\sg)\Ai(\beta+\lb-\tilde\sg)d\lb =K_{\Ai}(\kappa,\beta),
\end{aligned}
$$
and that
\be
\begin{aligned}
K_{\Ai}=T^2.\end{aligned}
\ee
  It therefore follows that the resolvent of the Airy kernel can be expressed as
 \be
\begin{aligned}
 \Id+R=(\Id -K_{\Ai})^{-1}=(\Id-T^2)^{-1}=\sum^{\iy}_{r=0}T^{2r } .
\end{aligned}
\label{K3}\ee
With this notation, 
 expression (\ref{E6}) for ${\cal A}^{\tau}_\xi(\kappa) $ can be written 
 \be
 {\cal A}^{\tau}_{\xi-\sg}(\kappa)=
 S_\xi^{\tau} (\kappa) -TS_{-\xi}^{\tau} (\kappa)
. \label{K4}\ee
 Setting 
 $$\la f(\kappa),g(\kappa)\ra=\int_{\tilde \sigma}^{\infty}
 f(\kappa) g(\kappa) d\kappa,$$
 one checks that
 \be
  K_{\Ai}^{(\tau_1,-\tau_2)}(\sg -\xi_1,\sg-\xi_2)=
  2^{1/3}\la S_{-\xi_1}^{\tau_1}, S_{-\xi_2}^{-\tau_2}\ra
  \label{K7}\ee
  and, using (\ref{K3}),
 \be\begin{aligned}
\lefteqn{ \int_{\tilde\sg}^{\iy}\left((\Id- \raisebox{1mm}{$\chi$}{}_{\tilde\sg}K_{\Ai}\raisebox{1mm}{$\chi$}{}_{\tilde\sg})^{-1}
 \AR_{\xi_1-\sg}^{ \tau_1}\right)(\lb)
\AR_{\xi_2-\sg}^{ -\tau_2 }(\lb)d\lb}
\\&=
\left\la (\Id+R)(S_{\xi_1}^{\tau_1} 
  -TS_{ -\xi_1}^{\tau_1} 
~,  
 ~S_{\xi_2}^{-\tau_2} 
  -TS_{-\xi_2}^{-\tau_2}
  \right\ra
  \\
  &=
  \bigl\la  \sum_0^\infty T^{2r}
   S_{\xi_1}^{\tau_1} 
  ,  
 ~S_{\xi_2}^{-\tau_2} 
  \bigr\ra
  + \bigl\la  \sum_0^\infty T^{2r+2}
   S_{-\xi_1}^{\tau_1} 
  ,  
 ~S_{-\xi_2}^{-\tau_2} 
  \bigr\ra
 \\
 & -
   \bigl\la  \sum_0^\infty T^{2r+1}
   S_{-\xi_1}^{\tau_1} 
  ,  
 ~S_{ \xi_2}^{-\tau_2} 
  \bigr\ra
   -
   \bigl\la  \sum_0^\infty T^{2r+1}
   S_{ \xi_1}^{\tau_1} 
  ,  
 ~S_{ -\xi_2}^{-\tau_2} 
  \bigr\ra
  \end{aligned}
  \label{K8}\ee
To conclude, by adding the two contributions (\ref{K7}) and (\ref{K8}), the kernel, as in (\ref{E127}) takes on the following form:
$$\begin{aligned}
\lefteqn{{\mathbb K}^{\rm tac}(\tau_1,\xi_2;\tau_2,\xi_2)
}\\
&=-\Id_{\tau_1>\tau_2}
p(\tau_1-\tau_2;\xi_1,\xi_2)
 \\
 &+
  2^{1/3}\bigl\la  \sum_0^\infty T^{2r}
   S_{\xi_1}^{\tau_1} 
  ,  
 ~S_{\xi_2}^{-\tau_2} 
  \bigr\ra
  +  2^{1/3}\bigl\la  \sum_0^\infty T^{2r}
   S_{-\xi_1}^{\tau_1} 
  ,  
 ~S_{-\xi_2}^{-\tau_2} 
  \bigr\ra
 \\
 & -
   2^{1/3} \bigl\la  \sum_0^\infty T^{2r+1}
   S_{-\xi_1}^{\tau_1} 
  ,  
 ~S_{ \xi_2}^{-\tau_2} 
  \bigr\ra
   -
   2^{1/3} \bigl\la  \sum_0^\infty T^{2r+1}
   S_{ \xi_1}^{\tau_1} 
  ,  
 ~S_{ -\xi_2}^{-\tau_2} 
  \bigr\ra
 \end{aligned}$$
 \be \begin{aligned} 
  =&-\Id_{\tau_1>\tau_2}
p(\tau_1-\tau_2;\xi_1,\xi_2) 
 \\
 &+
  2^{1/3}  \int_{\tilde \sg}^{\infty}  d\lambda
  \left[\begin{aligned}&\left((\Id - K_{\Ai})_{\tilde\sg}^{-1}
   S_{\xi_1}^{\tau_1}\right)(\lb) 
 ~S_{\xi_2}^{-\tau_2} (\lb)
  \\&-\left((\Id - K_{\Ai})_{\tilde\sg}^{-1}TS_{-\xi_1}^{\tau_1} 
 \right)(\lb)~S_{ \xi_2}^{-\tau_2} (\lb)\end{aligned}+\{\xi_i\leftrightarrow -\xi_i\}\right],
  \end{aligned}
  \label{K6}\ee
  which yields formula (\ref{E127K}), upon using formula (\ref{K4}) for $ {\cal A}^{\tau_1}_{\xi_1-\sg}(\lb)$.

 \bigbreak

 \noindent {\em Proof of Theorem \ref{th:main3}: }  The kernel ${\mathbb K}_{br}^{\rm tac}$ as in (\ref{E127L}) can be written, using the notations (\ref{K3}) and (\ref{K4}) as follows
 %

$$\begin{aligned}
{\mathbb K}_{\rm br}^{\rm tac}
=&-\Id_{\tau_1>\tau_2}
p(\tau_1-\tau_2;\xi_1,\xi_2)\\
&+2^{1/3}\left\{\begin{aligned}
 &\la S_{-\xi_1}^{\tau_1}, S_{-\xi_2}^{-\tau_2}\ra
+\la S_{ \xi_1}^{\tau_1}, S_{ \xi_2}^{-\tau_2}\ra
\\
&
-\left\la (\Id+R)(S_{\xi_1}^{-\tau_2} 
  -TS_{ -\xi_2}^{-\tau_2} 
~,  
 ~TS_{-\xi_1}^{ \tau_1}
  \right\ra
\\&-\left\la (\Id+R)(S_{-\xi_2}^{-\tau_2} 
  -TS_{  \xi_2}^{-\tau_2} 
~,  
 ~TS_{ \xi_1}^{ \tau_1}
  \right\ra
\end{aligned}\right\}
\\
&=-\Id_{\tau_1>\tau_2}
p(\tau_1-\tau_2;\xi_1,\xi_2)\\
 &+ 2^{1/3}\bigl\la  \sum_0^\infty T^{2r}
   S_{\xi_1}^{\tau_1} 
  ,  
 ~S_{\xi_2}^{-\tau_2} 
  \bigr\ra
  +  2^{1/3}\bigl\la  \sum_0^\infty T^{2r}
   S_{-\xi_1}^{\tau_1} 
  ,  
 ~S_{-\xi_2}^{-\tau_2} 
  \bigr\ra
 \\
 & -
   2^{1/3} \bigl\la  \sum_0^\infty T^{2r+1}
   S_{-\xi_1}^{\tau_1} 
  ,  
 ~S_{ \xi_2}^{-\tau_2} 
  \bigr\ra
   -
   2^{1/3} \bigl\la  \sum_0^\infty T^{2r+1}
   S_{ \xi_1}^{\tau_1} 
  ,  
 ~S_{ -\xi_2}^{-\tau_2} 
  \bigr\ra ={\mathbb K}^{\rm tac},
  \end{aligned}
  $$
which yields the formula in (\ref{K6}), ending the proof of Theorem \ref{th:main3}.\qed

\section{Appendix: An expression for the  integral $C(0;x)$ in the kernel
  $\widetilde\BK_{n,m} $}

 As was pointed out in Proposition \ref{Prop7.3},   in order to take scaling limits, one needs to express $C(0;x)$, defined in (\ref{55}), in a more convenient form. This is done below. In this section, let $C(x):=C(0;x)$ and $C_i(x):=C_i(0;x)$.
\begin{proposition}: \label{prop8.1} The following holds:
 \be
 \begin{aligned}
C(0;x) =&
 \sum_{k\geq 2m+1} \left(Q^{(1)}_k g_{k+x}^{(2)}
 +  Q^{(2)}_k g_{k-x}^{(1)}\right)
 \\
 &+  \sum_{k\geq 2m+1} Q^{(1)}_kQ^{(2)}_\ell \left(K^{(1)}_{k,\ell-x}(0)+
 K^{(1)}_{k+x,\ell}(0)\right).
 \end{aligned}
 \label{84}\ee

 \end{proposition}
 
 \proof 
 In this proof one uses over and over the expressions (\ref{39}) for $R^{(1)}(z^{-1})$ and $R^{(2)}(w)$, expressions (\ref{40}) for $S^{(i)} $ and $T^{(i)}$, the series expansions (\ref{36}) for $\bar h_k^{(i)}$ and the integral representation (\ref{34}) for $g_\ell^{(i)}=g_\ell^{(i)}(n)$. 
%
%
Using
$$
\oint_{\Gamma_{0}} \frac{dz}{2\pi i}\frac 1{(-z)^{x+1}}=\dt_{x,0}
$$
 one checks by substituting the expressions above, and recognizing the integrals for $g^{(i)}_{\ell}$,
 \be\begin{aligned}
\lefteqn{- \oint_{\Gamma_{0,a}} \frac{dz}{2\pi i}\frac 1{(-z)^{x+1}} 
 (R^{(1)}(z^{-1})+R^{(2)}(z))}\\
 &~~~~~~~=
 \dt_{x<0}\sum_{k\geq 2m+1} Q^{(1)}_k g_{k+x}^{(2)}
 +\dt_{x>0}\sum_{k\geq 2m+1} Q^{(2)}_k g_{k-x}^{(1
 )}
 \end{aligned}\label{86}\ee
 and
  \be\begin{aligned}
 \lefteqn{ \oint_{\Gamma_{0,a}} \frac{dz}{2\pi i}\frac 1{(-z)^{x+1}} 
  R^{(1)}(z^{-1}) R^{(2)}(z) }
\\
 &=-\sum_{{k,\ell\geq 2m+1}\atop {\alpha,\beta \geq 0}}
 Q^{(1)}_kQ^{(2)}_\ell
 g_{k+\alpha}^{(2)}g_{\ell+\beta}^{(1)}
\dt_{x,\alpha-\beta}-
\sum_{{k,\ell\geq 2m+1}}
 Q^{(1)}_kQ^{(2)}_\ell
\dt_{x,\ell-k}
 \\&
 ~~~~~+\sum_{{k,\ell\geq 2m+1}\atop {\alpha  \geq 0}}
 Q^{(1)}_kQ^{(2)}_\ell
 \left(g_{k+\alpha}^{(2)}g_{\ell+\alpha-x}^{(1)}
 +
 g_{\ell+\alpha}^{(1)}g_{k+\alpha+x}^{(2)}\right).
  \end{aligned}\label{87}\ee
 So, $C_1(x)$ has the following form:
 \be
 C_1(x)=\dt_{x\neq 0}C^{\ast}_1(x)\label{88}\ee
 with (adding $\dt_{x=0}$ in the expression below is harmless, since $C^{\ast}_1(x)$ gets multiplied with $\dt_{x\neq 0}$),
  \be\begin{aligned}
  C^{\ast}_1(x)&:=
   (\dt_{x<0}+\tfrac 12 \dt_{x=0})\sum_{k\geq 2m+1} Q^{(1)}_k g_{k+x}^{(2)}
 +(\dt_{x>0}+\tfrac 12 \dt_{x=0})\sum_{k\geq 2m+1} Q^{(2)}_k g_{k-x}^{(1
 )}
\\
&-\sum_{{k,\ell\geq 2m+1}\atop {\alpha,\beta \geq 0}}
 Q^{(1)}_kQ^{(2)}_\ell
 g_{k+\alpha}^{(2)}g_{\ell+\beta}^{(1)}
\dt_{x,\alpha-\beta}-
\sum_{{k,\ell\geq 2m+1}}
 Q^{(1)}_kQ^{(2)}_\ell
\dt_{x,\ell-k}
 \\&
 ~~~~~+\sum_{{k,\ell\geq 2m+1}\atop {\alpha  \geq 0}}
 Q^{(1)}_kQ^{(2)}_\ell
 \left(g_{k+\alpha}^{(2)}g_{\ell+\alpha-x}^{(1)}
 +
 g_{\ell+\alpha}^{(1)}g_{k+\alpha+x}^{(2)}\right).
  \end{aligned}\label{89}\ee
Also, by straightforward substitution, one finds:
\be
C_2(x)=  \sum_{{k,\ell\geq 2m+1}}
 Q^{(1)}_kQ^{(2)}_\ell
\dt_{x,\ell-k}.
\label{90}\ee
In the Lemma below, it will be shown that $C_1^*(0)=0$ and thus $C_1 (x)=C_1^*(x)$. Therefore, adding (\ref{89}) and (\ref{90}), one checks, using (\ref{35a}) and the identity (\ref{93}) of the Lemma below:
 \be\begin{aligned}
C(x)&=C_1(x)+2C_2(x)=C_1^*(x)+2C_2(x)
\\
&=\left(\dt_{x<0}+\frac{1}{2}\dt_{x=0}\right)\sum_{k\geq 2m+1}Q_k^{(1)}g^{(2)}_{k+x}+\left(\dt_{x>0}+\frac{1}{2}\dt_{x=0}\right)\sum_{k\geq 2m+1}Q_k^{(2)}g^{(1)}_{k-x} 
\\
&~ +\sum_{k,\ell\geq 2m+1}Q_k^{(1)}Q_{\ell}^{(2)}\left(K^{(1)}_{k,\ell -x}(0)+K^{(1)}_{k+x,\ell}(0)\right) 
\\
&~ -\sum_{{k,\ell\geq 2m+1}\atop{\al,\beta\geq 0}}Q^{(1)}_kQ^{(2)}_{\ell}g^{(2)}_{k+\al}g^{(1)}_{\ell +\beta}\dt_{x,\al -\beta} 
  +\sum_{k,\ell\geq 2m+1}Q_k^{(1)}Q_{\ell}^{(2)}\dt_{x,\ell -k}
\\
 &=
 \sum_{k\geq 2m+1} \left(Q^{(1)}_k g_{k+x}^{(2)}
 +  Q^{(2)}_k g_{k-x}^{(1)}\right)
 \\&~+  \sum_{k\geq 2m+1} Q^{(1)}_kQ^{(2)}_\ell \left(K^{(1)}_{k,\ell-x}(0)+
 K^{(1)}_{k+x,\ell}(0)\right).
\label{91}\end{aligned}
\ee
This establishes Proposition \ref{prop8.1}.\qed

\begin{lemma} \label{L7.2} The following holds
\be\begin{aligned}
  C^{\ast}_1(0)&=0\\
  \end{aligned}\label{92}\ee
  and
  \be \begin{aligned}
& -\sum_{{k,\ell\geq 2m+1}\atop {\alpha,\beta \geq 0}}
 Q^{(1)}_kQ^{(2)}_\ell
 g_{k+\alpha}^{(2)}g_{\ell+\beta}^{(1)}
\dt_{x,\alpha-\beta}+
\sum_{{k,\ell\geq 2m+1}}
 Q^{(1)}_kQ^{(2)}_\ell
\dt_{x,\ell-k}\\
&=
 (\dt_{x>0}+\frac 12 \dt_{x=0})\sum_{k\geq 2m+1} Q^{(1)}_k g_{k+x}^{(2)}
 +(\dt_{x<0}+\frac 12 \dt_{x=0})\sum_{k\geq 2m+1} Q^{(2)}_k g_{k-x}^{(1
 )}
  \end{aligned}\label{93}\ee
  
  \end{lemma}
  
\proof
({\em i}) First we prove that
\be
\sum_{k\geq M}Q^{(2)}_kg_k^{(1)}=\sum_{k\geq M}Q_k^{(1)}g_k^{(2)}
\ee
Since, by definition, upon setting $P^{(i)}:=(\Id-\raisebox{1mm}{$\chi$}{}_{M}K^{(i)}(0)\raisebox{1mm}{$\chi$}{}_{M})^{-1}$, (set  $2m+1=M$)
$$\begin{aligned}
Q_k^{(1)}&=  \left((\Id- \raisebox{1mm}{$\chi$}{}_{M}K^{(2)}(0)\raisebox{1mm}{$\chi$}{}_{M})^{-1}\raisebox{1mm}{$\chi$}{}_{M}g^{(1)}\right)(k)=:P^{(2)}(\raisebox{1mm}{$\chi$}{}_{M}g^{(1)})(k)\\
\\
Q_k^{(2)}&= \left((\Id- \raisebox{1mm}{$\chi$}{}_{M}K^{(1)}(0)\raisebox{1mm}{$\chi$}{}_{M})^{-1}\raisebox{1mm}{$\chi$}{}_{M}g^{(2)}\right)(k)=:P^{(1)}(\raisebox{1mm}{$\chi$}{}_{M}g^{(2)})(k)
,\end{aligned}
$$
one has, using $P^{(1)\top}=P^{(2)}$,
$$\begin{aligned}
\sum_{k\geq M}Q_k^{(2)}g_k^{(1)}&= \la P^{(1)}( \raisebox{1mm}{$\chi$}{}_{M}g^{(2)}),~\raisebox{1mm}{$\chi$}{}_{M}g^{(1)}\ra  
\\
&= \la \raisebox{1mm}{$\chi$}{}_{M}g^{(2)},P^{(2)} (\raisebox{1mm}{$\chi$}{}_{M}g^{(1)})\ra = \sum_{k\geq M}Q_k^{(1)}g_k^{(2)}
.\end{aligned}
$$
({\em ii}) Therefore, by (\ref{89}),
\be\begin{aligned}
C_1^*(0)&= \sum_{k\geq M}Q_k^{(1)}g_k^{(2)}+\sum_{k,\ell\geq M}Q_k^{(1)}Q_{\ell}^{(2)}K^{(1)}_{k ,\ell}(0)-\sum_{k\geq M}Q_k^{(1)}Q_k^{(2)} 
\\
&= \la P^{(2)}(\raisebox{1mm}{$\chi$}{}_{M}g^{(1)}), \raisebox{1mm}{$\chi$}{}_{M}g^{(2)}\ra  
\\
& ~~ +\la P^{(2)}(\raisebox{1mm}{$\chi$}{}_{M}g^{(1)}), \raisebox{1mm}{$\chi$}{}_{M}K^{(1)}(0)  \raisebox{1mm}{$\chi$}{}_{M}P^{(1)} 
(\raisebox{1mm}{$\chi$}{}_{M}g^{(2)}) \ra  
\\
& ~~ -\la P^{(2)}(\raisebox{1mm}{$\chi$}{}_{M}g^{(1)}),P^{(1)}(\raisebox{1mm}{$\chi$}{}_{M}g^{(2)})\ra 
\\
&= \la P^{(2)}(\raisebox{1mm}{$\chi$}{}_{M}g^{(1)}),\raisebox{1mm}{$\chi$}{}_{M}g^{(2)}\ra  
\\
& ~~ -\la P^{(2)}(\raisebox{1mm}{$\chi$}{}_{M}g^{(1)}), (\Id-\raisebox{1mm}{$\chi$}{}_{M}K^{(1)}(0)\raisebox{1mm}{$\chi$}{}_{M})P^{(1)}
\raisebox{1mm}{$\chi$}{}_{M}g^{(2)}\ra  
= 0.
\end{aligned}\ee
Then $C_1^*(0)=0$ implies 
$ 
C_1(x)=C_1^*(x), 
$
establishing identity (\ref{92}).
 \bigbreak

({\em iii})  Using\footnote{for $P=(\Id-K)^{-1}$, one has $P=\Id+KP$}, for $k\geq M$,
\be\begin{aligned}
Q_k^{(1)}&=&  g_k^{(1)}+\sum_{\ell\geq M}K_{k\ell}^{(2)}(0)Q_{\ell}^{(1)} 
\\
Q_k^{(2)}&=&  g_k^{(2)}+\sum_{\ell\geq M}K_{k\ell}^{(1)}(0)Q_{\ell}^{(2)}
\end{aligned}
\ee
and writing $1= \left(\dt_{x>0}+\tfrac{1}{2}\dt_{x=0}\right)+\left(\dt_{x<0}+\tfrac{1}{2}\dt_{x=0}\right)$, 
$$\begin{aligned}
\lefteqn{\sum_{k,\ell\geq M}Q_k^{(1)}Q_{\ell}^{(2)}\dt_{x,\ell -k}} 
\\
&= \sum_{k,\ell\geq M}Q_k^{(1)}Q_{\ell}^{(2)}\dt_{x,\ell -k}\left(  \left(\dt_{x>0}+\tfrac{1}{2}\dt_{x=0}\right)+\left(\dt_{x<0}+\tfrac{1}{2}\dt_{x=0}\right)\right) 
\\
&= \left(\dt_{x>0}+\tfrac{1}{2}\dt_{x=0}\right)\sum_{k\geq M}Q_k^{(1)}Q^{(2)}_{k+x}  
 + \left(\dt_{x<0}+\tfrac{1}{2}\dt_{x=0}\right)\sum_{k\geq M}Q_{k -x}^{(1)}Q^{(2)}_{k} 
\end{aligned}$$
\be \begin{aligned}
&= \left(\dt_{x>0}+\tfrac{1}{2}\dt_{x=0}\right)\left(\sum_{k\geq M}Q_k^{(1)}g^{(2)}_{k+x}+\sum_{k,\ell\geq M}Q_k^{(1)}K^{(1)}_{k+x,\ell}(0)Q_{\ell}^{(2)}\right)
\\
&  +\left(\dt_{x<0}+\tfrac{1}{2}\dt_{x=0}\right)\left(\sum_{k\geq M}Q_k^{(2)}g^{(1)}_{k-x}+\sum_{k,\ell\geq M}Q_k^{(2)}K^{(2)}_{k-x,\ell}(0)Q_{\ell}^{(1)}\right).
\end{aligned}\ee
Also, by (\ref{35a}),
\be\begin{aligned}
{\sum_{k,\ell\geq 2m+1\atop{\al,\beta\geq 0}}Q_k^{(1)}Q^{(2)}_{\ell}g^{(2)}_{k+\al}g^{(1)}_{\ell +\beta}\dt_{x,\al -\beta}} 
&= \left(\dt_{x<0}+\frac{1}{2}\dt_{x=0}\right)\sum_{k,\ell\geq M}Q_{\ell}^{(1)}Q_k^{(2)}K^{(2)}_{k-x,\ell}(0) 
\\
&  ~~~+\left(\dt_{x>0}+\frac{1}{2}\dt_{x=0}\right)\sum_{k,\ell\geq M}Q_{k}^{(1)}Q^{(2)}_{\ell}K^{(1)}_{k+x,\ell}(0).
\end{aligned}\ee
Hence, subtracting these two formulas, one finds
\be \begin{aligned}
\lefteqn{\sum_{k,\ell\geq M}Q_k^{(1)}Q_{\ell}^{(2)}\dt_{x,\ell -k}-\sum_{k,\ell\geq M\atop{\al,\beta\geq 0}}Q_k^{(1)}Q_{\ell}^{(2)}g^{(2)}_{k+\al}g^{(1)}_{\ell +\beta}\dt_{x,\al -\beta}} 
\\
&= \left(\dt_{x>0}+\tfrac{1}{2}\dt_{x=0}\right)\sum_{k\geq M}Q_k^{(1)}g_{k+x}^{(2)}+\left(\dt_{x<0}+\tfrac{1}{2}\dt_{x=0}\right)Q_k^{(2)}g^{(1)}_{k-x},
\end{aligned}
\ee
thus ending the proof of identity (\ref{93}) and Lemma \ref{L7.2}.  \qed

\end{document}